%% file: main.tex
\documentclass[a4paper,10pt,intlimits,sumlimits]{amsart}
\usepackage[utf8]{inputenc}

\usepackage{conf}

\usepackage[color=blue!30]{todonotes} 

\begin{document}

\input{main/titlepage.tex}
\input{main/abstract.tex}
\maketitle
\footnotesize
\tableofcontents
\normalsize

\section{Introduction}
\label{sec:intro}
\input{main/intro.tex}

\section{Preliminaries}
\label{sec:preliminaries}
\input{main/preliminaries.tex}

\section{Analysis on the time-frequency-scale space}
\label{sec:time-frequency}
\input{main/time-frequency.tex}

\section{Local size-H\"older / the single-tree estimate}
\label{sec:single-tree}
\input{main/single-tree.tex}

\section{Local size domination}
\label{sec:size-domination}
\input{main/size-domination.tex}

\section{Embeddings into non-iterated outer Lebesgue spaces}
\label{sec:embeddings-noniter}
\input{main/embeddings-noniter.tex}

\section{Embeddings into iterated outer Lebesgue spaces}
\label{sec:embeddings-iter}
\input{main/embeddings-iter.tex}

\section{Applications to bilinear Hilbert transforms}
\label{sec:BHF}
\input{main/bhf.tex}


%
\footnotesize
\bibliographystyle{amsplain}
\bibliography{bibliography}

\end{document}

%% file: main/titlepage.tex
\title[Bilinear Hilbert Transform in UMD spaces]{The bilinear Hilbert transform in UMD spaces}
\date{\today}

\author[A. Amenta]{Alex Amenta}
\address{\noindent Mathematisches Institut \newline \indent Universit\"at Bonn, Bonn, Germany}
\email{amenta@math.uni-bonn.de}

\author[G. Uraltsev]{Gennady Uraltsev}
\address{\noindent Department of Mathematics \newline \indent Cornell University, Ithaca, NY, USA}
\email{gennady.uraltsev@gmail.com}

\subjclass[2010]{Primary 42B20, Secondary 42B25, 47A56}
\keywords{Bilinear Hilbert transform, time-frequency analysis, UMD Banach spaces, outer Lebesgue spaces, $\gamma$-radonifying operators, $R$-bounds, interpolation spaces}


%% file: main/abstract.tex
\begin{abstract}
  We prove $L^p$-bounds for the bilinear Hilbert transform acting on functions valued in intermediate UMD spaces.
  Such bounds were previously unknown for UMD spaces that are not Banach lattices.
  Our proof relies on bounds on embeddings from Bochner spaces $L^p(\R;X)$ into outer Lebesgue spaces on the time-frequency-scale space $\R^3_+$.
\end{abstract}


%% file: main/intro.tex
The \emph{bilinear Hilbert transform} is the bilinear singular integral operator defined on scalar-valued Schwartz functions $f_1, f_2 \in \Sch(\R;\C)$ by
\begin{equation}\label{eq:BHT-defn-scalar}
  \BHT(f_{1},f_{2})(x) := \pv \int_{\R} f_{1}(x-y) f_{2}(x+y) \, \frac{\dd y}{y} \qquad \forall x \in \R.	
\end{equation}
This operator was introduced by Calder\'on in the 1960s in connection with the first Calder\'on commutator.
The first $L^p$-bounds for $\BHT$ were proven by Lacey and Thiele \cite{LT97,LT99} using a newly-developed form of time-frequency analysis, extending techniques introduced by Carleson and Fefferman \cite{lC66,cF73}.
Since then the bilinear Hilbert transform has served as the fundamental example of a modulation-invariant singular integral (see for example \cite{MTT02}) and as a proving ground for time-frequency techniques.
Since we are mainly interested in the vector-valued theory we direct the reader to \cite{MS13-2,cT06} as starting points for further reading.

Consider three complex Banach spaces $X_{1}$, $X_{2}$, $X_{3}$ and a bounded trilinear form $\map{\Pi}{X_{1} \times X_{2} \times X_{3}}{\C}$.
With respect to this data one can define a trilinear singular integral form on Schwartz functions $f_i \in \Sch(\R;X_i)$ by
\begin{equation*}
  \BHF_{\Pi}(f_{1},f_{2},f_{3}) := \int_\R \pv \int_{\R}\Pi \big( f_{1}(x-y),  f_{2}(x+y) , f_{3}(x) \big) \, \frac{\dd y}{y} \, \dd x.
\end{equation*}
This is the dual form to a bilinear operator $\BHT_{\Pi}$, defined analogously to \eqref{eq:BHT-defn-scalar}, which will be discussed later in the introduction.
We are interested in $L^p$-bounds
\begin{equation}\label{eq:intro-Lpbd}
  |\BHF_{\Pi}(f_1, f_2, f_3)| \lesssim \prod_{i=1}^3 \|f_i\|_{L^{p_i}(\R;X_i)} \qquad \forall f_i \in \Sch(\R;X_i)
\end{equation}
for a range of exponents $p_i$ quantified in terms of the the geometry of the Banach spaces $X_i$.
In this article we prove the following result for intermediate UMD spaces.\footnote{We extend this to sparse domination, provide a result for exponents $p < 1$, and discuss particular examples of Banach spaces $X_i$ and trilinear forms $\Pi$ in Section \ref{sec:BHF}. There we also discuss the condition \eqref{eq:intro-Lpbd}, which is optimal in a certain sense.}

\begin{thm}\label{thm:intro-BHT}
  Let $(X_{i})_{i=1}^3$ be $r_i$-intermediate UMD spaces for some $r_i \in [2,\infty)$, and let $\map{\Pi}{X_1 \times X_2 \times X_3}{\CC}$ be a bounded trilinear form.
  Then for all exponents $p_1,p_2,p_3 \in (1,\infty)$ satisfying $\sum_{i=1}^3 p_i^{-1} = 1$ and
  \begin{equation}\label{eq:intro-BHT-exponents}
    \sum_{i=1}^3 \frac{1}{\min(p_i,r_i)'(r_i - 1)} > 1,
  \end{equation}
  the trilinear form $\BHF_\Pi$ satisfies the bound \eqref{eq:intro-Lpbd}.
\end{thm}

The UMD property of a Banach space $X$ has many equivalent characterisations.
The most natural from the viewpoint of harmonic analysis is given in terms of the (linear) Hilbert transform on $X$-valued functions, defined by
\begin{equation*}
  \HT f(x) := \pv \int_\R f(x-y) \, \frac{\dd y}{y} \qquad \forall f \in \Sch(\R;X) \quad \forall x \in \R:
\end{equation*}
the Banach space $X$ is UMD if and only if $\HT$ is bounded on $L^p(\R;X)$ for all $p \in (1,\infty)$.
UMD stands for `Unconditionality of Martingale Differences', a probabilistic concept whose equivalence with boundedness of the Hilbert transform is due to Burkholder and Bourgain \cite{jB83,dB83}.
Examples of UMD spaces include separable Hilbert spaces, most classical reflexive function spaces (including Lebesgue, Sobolev, Besov, and Triebel--Lizorkin spaces), and non-commutative $L^p$ spaces with $p \in (1,\infty)$.
For more information see for example \cite{dB01,HNVW16,gP16}.

All known examples of UMD spaces actually satisfy the formally stronger property of being \emph{intermediate UMD}:
\begin{defn}
  For $r \in [2,\infty)$, a Banach space $X$ is \emph{$r$-intermediate UMD} if there exists a UMD space $Y$ and a Hilbert space $H$, forming a compatible couple, such that $X$ is isomorphic to the complex interpolation space $[Y,H]_{2/r}$.\footnote{For a discussion of complex interpolation see \cite{BL76} or \cite[Appendix C]{HNVW16}.}
  We say $X$ is \emph{intermediate UMD} if it is $r$-intermediate UMD for some $r \in [2,\infty)$.
\end{defn}

A basic example is $X = L^p(\R)$, which is $r$-intermediate UMD for all $r > \max(p,p')$, or for all $r\geq 2$ in the special case $p=2$.
In \cite{RdF86} Rubio de Francia conjectured that every UMD space is in fact intermediate UMD, and the question remains open.\footnote{In Section \ref{sec:sharpness} we propose a quantitative strengthening of this conjecture with a view towards improvement of Theorem \ref{thm:intro-BHT}.}

Theorem \ref{thm:intro-BHT} is stated for the trilinear form $\BHF_\Pi$, which is dual to a bilinear operator $\BHT_\Pi$ defined as follows: for $f_1 \in \Sch(\R;X_1)$ and $f_2 \in \Sch(\R;X_2)$, $\BHT_\Pi(f_1,f_2)$ is an $X_3^*$-valued function on $\R$ defined by
\begin{equation}\label{eq:BHT-defn}
  \BHT_\Pi(f_{1},f_{2})(x) := \pv \int_{\R} \Pi(f_{1}(x-y),f_{2}(x+y),\cdot) \, \frac{\dd y}{y} \qquad \forall x \in \R,	
\end{equation}
where for vectors $\mb{x}_{1} \in X_1$, $\mb{x}_{2} \in X_2$, the functional $\Pi(\mb{x}_{1},\mb{x}_{2},\cdot) \in X_{3}^{*}$ is defined by
\begin{equation*}
  \langle \Pi(\mb{x}_{1},\mb{x}_{2},\cdot) ; \mb{x}_{3} \rangle := \Pi(\mb{x}_{1},\mb{x}_{2},\mb{x}_{3}) \qquad \forall \mb{x}_{3} \in X_{3}.
\end{equation*}
An $L^p$-bound of the form \eqref{eq:intro-Lpbd} for $\BHF_\Pi$ is then equivalent to the bound
\begin{equation}\label{eq:intro-Lpbd-BHT}
  \|\BHT_\Pi(f_1, f_2)\|_{L^{p_3'}(\R;X_3^*)} \lesssim \|f_1\|_{L^{p_1}(\R;X_1)} \|f_2\|_{L^{p_2}(\R;X_2)}  
\end{equation}
for the bilinear Hilbert transform $\BHT_{\Pi}$.\footnote{Such a bound can hold even when $p_3' < 1$, although in this case it is not equivalent to a bound for $\BHF_\Pi$. This is discussed in Section \ref{sec:bounds-for-BHT}.}

As already stated, scalar-valued estimates for $\BHF_\Pi$ (that is, letting $\map{\Pi}{\CC \times \CC \times \CC}{\CC}$ be the ordinary product) date back to the landmark papers of Lacey and Thiele.
Their technique handles the case in which $X_1 = X_2 = H$ is a separable Hilbert space with inner product $\langle \cdot ; \cdot \rangle$, $X_3 = \C$, and $\Pi(\mb{x},\mb{y},\lambda) := \lambda \langle \mb{x}; \mb{y}\rangle$ with only minor modifications, as orthogonality is more relevant here than finite-dimensionality.
Beyond the Hilbert space setting, the first estimates for $\BHF_\Pi$ are due to Silva \cite[Theorem 1.7]{pS14}, who studied the case $X_1 = \ell^R$, $X_2 = \ell^\infty$, $X_3 = \ell^{R'}$ with $R \in (4/3,4)$, where $\Pi$ is the H\"older form.
This result shows that the Banach spaces $X_i$ need not be UMD (as $\ell^\infty$ is not).
Benea and Muscalu \cite{BM16,BM17} extended this result to mixed-norm spaces, including $L^\infty$ and some quasi-Banach spaces, by a new `helicoidal' method.\footnote{The papers of Silva and Benea--Muscalu treat more general operators than just $\BHF_\Pi$, but this is beyond our focus.}
The multilinear extrapolation theorems of Lorist and Nieraeth \cite{LN19, LN20, bN19} allow for even more general function spaces, using the weighted scalar-valued bounds of \cite{CUM18, CDPO18} or the sparse domination proved in \cite{BM18,CDPO18}.

In the aforementioned work the Banach spaces $X_i$ are Banach lattices, i.e. Banach spaces equipped with a partial order compatible with the norm (see for example \cite{LT79}). This excludes many important Banach spaces, including Sobolev spaces and non-commutative $L^p$-spaces.
The main interest of our results is that they make no use of lattice structure.
Previously the only results in this general setting were for discrete Walsh models of $\BHT_\Pi$, namely the quartile and tritile operators, which were established by Hyt\"onen, Lacey, and Parissis \cite{HLP13}.\footnote{We recently proved the same bounds using the outer Lebesgue space framework \cite{AU20-Walsh}.}
These models are an important step on the way to the `continuous' operator, as although there are no known methods to deduce bounds for $\BHT_\Pi$ from its Walsh models, analysing the models is a great aid to understanding $\BHT_\Pi$ itself.

Our proof of Theorem \ref{thm:intro-BHT} exploits a well-known wave packet representation of $\BHF_\Pi$.
For all Schwartz functions $\phi\in\Sch(\R)$ and any point $(\eta,y,t) \in \R^{3}_{+} = \R \times \R \times \R_+$, define the translation, modulation, and dilation operators 
\begin{equation*}
  \begin{aligned}
    &  \Tr_{y} \phi(z):= \phi(z-y) &&\Mod_{\eta} \phi(z):= e^{2\pi i \eta z}\phi(z) && \Dil_{t} \phi(z):= t^{-1}\phi\Big(\frac{z}{t}\Big)
  \end{aligned}
\end{equation*}
as well as the composition
\begin{equation*}
  \Lambda_{(\eta,y,t)} := \Tr_{y}\Mod_{\eta}\Dil_{t}.
\end{equation*}
The function $\Lambda_{(\eta,y,t)}\phi$ is called a \emph{wave packet} at the point $(\eta, y,t) \in \R^3_+$.
For $f \in \Sch(\R;X)$ (where $X$ is any Banach space) we define the \emph{wave packet embedding} of $f$ with respect to $\phi$ at $(\eta,y,t) \in \R^3_+$ by
\begin{equation}
  \label{eq:embedding}
  \Emb[f][\phi](\eta,y,t):= \langle f ; \Lambda_{(\eta,y,t)}\phi\rangle = \int_{\R} f(x)\,t^{-1}  e^{-2\pi i \eta(x-y)} \bar{\phi}\Bigl( \frac{x-y}{t} \Bigr) \, \dd x,
\end{equation}
so that $\map{\Emb[f][\phi]}{\R^3_+}{X}$.
To allow for different choices of $\phi$ we consider each $\Emb[f](\eta,y,t)$ as a linear operator from $\Sch(\R)$ to $X$, i.e. as an $X$-valued tempered distribution on $\R$.
For technical reasons we replace $\Sch(\R)$ with a particular subspace $\Phi$ of band-limited functions, and we view $\Emb[f]$ as a $\Lin(\Phi;X)$-valued function.

We work with a modified operator $\widetilde{\BHF}_\Pi$ which has a simpler wave packet representation:
for $f_i \in \Sch(\R;X_i)$ we define
\begin{equation*}
  \begin{aligned}
    &\widetilde{\BHF}_\Pi(f_{1},f_{2},f_{3}) \\
    &:=\int_{\R^{3}} e^{2 \pi i x (\xi_{1}+\xi_{2})}\1_{(0,\infty)}(\xi_{1}-\xi_{2}) \Pi(\FT{f_{1}}(\xi_{1}), \FT{f_{2}}(\xi_{2}), f_{3}(x)) \,  \dd \xi_{1} \, \dd \xi_{2} \, \dd x
    \\
    &= \frac{1}{2} \int_\R \Pi(f_{1}(x),f_{2}(x),f_3(x)) \, \dd x +\frac{1}{2 \pi i} \BHF_\Pi(f_1,f_2,f_3).
  \end{aligned}
\end{equation*}
Since $\widetilde{\BHF}_\Pi$ is a nontrivial linear combination of $\BHF_\Pi$ and the H\"older form, $L^p$-bounds for $\widetilde{\BHF}_\Pi$ are equivalent to those for $\BHF_\Pi$.
The wave packet representation of $\widetilde{\BHF}_\Pi$ is as follows: there exists a Schwartz function $\phi_0 \in \Sch(\R)$ with Fourier transform supported in $B_{2^{-5}}$ such that
\begin{equation}\label{eq:BHF-wave-packet-repr} \begin{aligned}[t]
    &\widetilde{\BHF}_\Pi(f_{1},f_{2},f_{3})
    \\ &=\int_{\R^{3}_{+}} \Pi \big( \langle f_{1}; \Lambda_{(y, \eta+t^{-1}, t)} \phi_0 \rangle,  \langle f_{2}; \Lambda_{(y, \eta-t^{-1}, t)} \phi_0 \rangle,  \langle f_{3}; \Lambda_{(y, -2\eta, t)} \phi_0 \rangle \big) \, \dd \eta \,\dd y\,\dd t
    \\ & = \int_{\R^{3}_{+}} \Pi \big( \Emb[f_{1}][\phi_0](\eta-t^{-1}\!\!,y,t) , \Emb[f_{2}][\phi_0](\eta+t^{-1}\!\!,y,t),  \Emb[f_{3}][\phi_0](-2\eta,y,t) \big) \, \dd \eta  \dd y  \dd t.
  \end{aligned}
\end{equation}
This integral converges absolutely as long as $f_{i}\in\Sch(\R;X_i)$.
This motivates the definition of the following $\BHF$-type \emph{wave packet forms} on $\Lin(\Phi;X_i)$-valued functions: for $\map{F_i}{\R^3_+}{\Lin(\Phi;X_i)}$, $\phi \in \Phi$, and any compact $\mbb{K}\subset \R^{3}_{+}$,
\begin{equation*}
  \begin{aligned}
    &\BHWF_{\Pi}^{\phi,\mbb{K}}(F_1,F_2,F_3) \\
    &:= \int_{\mbb{K}} \Pi \big( F_1[\phi](\eta-t^{-1},y,t) , F_2[\phi](\eta+t^{-1},y,t), 
    F_3[\phi](-2\eta,y,t) \big) \, \dd \eta \, \dd y \, \dd t.
  \end{aligned}
\end{equation*}
Then the preceding discussion says that 
\begin{equation*}
  \widetilde{\BHF}_\Pi(f_1,f_2,f_3) = \lim_{\mbb{K}\uparrow \R^{3}_{+}} \BHWF_{\Pi}^{\phi_0,\mbb{K}}(\Emb[f_1], \Emb[f_2], \Emb[f_3]),
\end{equation*}
where the limit is taken over any increasing sequence of compact sets covering $\R^{3}_{+}$. 

We proceed using the framework of outer Lebesgue spaces, as introduced by Do and Thiele \cite{DT15} and successfully utilised in a number of papers (see for example \cite{CDPO18,DPDU18,DPGTZK18,DPO18,DPO18-2,DMT17,MT17,TTV15,gU16,gU-thesis,mW-thesis}).
The idea of this method is to construct quasinorms $\|F_i\|_{L_{\mu_i}^{p_i} \OS_i}$ on functions $\map{F_i}{\R^3_+}{\Lin(\Phi;X_i)}$ satisfying two key properties: first, a Hölder-type inequality
\begin{equation}\label{eq:intro-size-holder}
  |\BHWF_{\Pi}^{\phi_{0},\mbb{K}}(F_1, F_2, F_3)| \lesssim \prod_{i=1}^3 \|F_i\|_{L_{\mu_i}^{p_i} \OS_i}
\end{equation}
for arbitrarily large compact sets $\mbb{K}\subset\R^{3}_{+}$ of appropriate shape,
and second, bounds for the embedding map of the form
\begin{equation}\label{eq:intro-embedding-bd}
  \|\Emb[f_i]\|_{L_{\mu_i}^{p_i} \OS_i} \lesssim \|f_i\|_{L^{p_i}(\R;X_i)} \qquad \forall f_i \in \Sch(\R;X_i).
\end{equation}
If such quasinorms can be constructed, then the chain of inequalities
\begin{align*}
  |\widetilde{\BHF}_\Pi(f_1,f_2,f_3)|
  &=\lim_{\mathbb{K}\uparrow \R^{3}_{+}} |\BHWF_{\Pi}^{\phi_0,\mathbb{K}}(\Emb[f_1], \Emb[f_2], \Emb[f_3])| \\
  &\lesssim \prod_{i=1}^3 \|\Emb[f_i]\|_{L_{\mu_i}^{p_i} \OS_i} 
    \lesssim \prod_{i=1}^3 \|f_i\|_{L^{p_i}(\R;X_i)}
\end{align*}
gives $L^p$-bounds for $\widetilde{\BHF}_\Pi$, and hence also for $\BHF_{\Pi}$.
Without giving too much detail, the quasinorms $L^p_\mu \OS$ are \emph{outer Lebesgue quasinorms}, defined in terms of the following data:
\begin{itemize}
\item a collection $\OB$ of subsets of $\R^3_+$, called \emph{generating sets},
\item a premeasure $\mu$ on $\OB$,
\item a \emph{local size} $\OS$ on $\OB$, which gives a way of measuring the `size' of a function on each generating set $B \in \OB$.
\end{itemize}
The collected data $(\R^3_+, \OB, \mu, \OS)$ is called an \emph{outer space}, and we think of it as a generalised geometric structure and corresponding integration theory.
In this article we construct outer Lebesgue quasinorms $L_{\mu_i}^p \FS_i^{s}$ and $L_{\nu}^p \sL_{\mu_i}^q \FS_i^{s}$ satisfying H\"older-type inequalities (Corollary \ref{cor:BHT-RN-reduction}) and having bounds for the embedding map (Theorems \ref{thm:non-iter-embedding} and \ref{thm:RS-iterated-embeddings}, restated below), leading to Theorem \ref{thm:intro-BHT} by the argument above.

\begin{thm}\label{thm:intro-embeddings}
  Fix $r \in [2,\infty)$ and let $X$ be $r$-intermediate UMD.
  Then the following embedding bounds hold:
  \begin{itemize}
  \item For all $p \in (r,\infty)$ and all $s\in(1,\infty)$,
    \begin{equation*}
      \|\Emb[f] \|_{L_\mu^p \FS^{s}} \lesssim \|f \|_{L^p(\R;X)} \qquad \forall f \in \Sch(\R;X).
    \end{equation*}
  \item For all $p \in (1,\infty)$, all $q \in (\min(p,r)^\prime(r-1),\infty)$, and all $s\in(1,\infty)$
    \begin{equation*}
      \|\Emb[f]\|_{L_\nu^{p} \sL_\mu^q \FS^{s}} \lesssim \|f\|_{L^p(\R;X)} \qquad \forall f \in \Sch(\R;X).
    \end{equation*}
  \end{itemize}
\end{thm}

Our proofs introduce some new ideas in time-frequency analysis.
Here we highlight just two.
First, the local sizes $\FS^{s}$ that we use are defined in terms of $\gamma$-norms (defined in Section \ref{sec:preliminaries}), which represent continuous square functions in the Banach-valued setting.
Since we work on the continuous time-frequency-scale space $\R^3_+$ rather than a discretised version, we can estimate $\gamma$-norms by controlling not only the pointwise values of functions but also their derivatives, essentially arguing via Sobolev embeddings.
Second, to prove the H\"older-type inequality \eqref{eq:intro-size-holder} without having to directly control $R$-bounds of embedded functions (see the argument presented in \cite[\textsection 7.2]{AU20-Walsh}), we introduce \emph{defect operators} on functions $\map{F}{\R^3_+}{\Lin(\Phi;X)}$.
These measure how far $F$ is from being an embedded function $\Emb[f]$.
Outer space theory on $\R^3_+$ is well-adapted to keeping track of these quantities, and while a discrete argument should be possible, it would be technically less convenient.
The definition of the defect operators requires that we consider functions valued in $\Lin(\Phi;X)$ rather than $X$ itself, and exploits the relation between different choices of $\phi \in \Phi$.

Our analysis roughly follows the path laid out in our previous work \cite{AU20-Walsh} on the tritile operator.
We recommend that readers new to time-frequency analysis start by reading that article, as it contains many of the core ideas of our arguments without most of the annoying technicalities.
More experienced readers may prefer to read both articles in parallel.

\begin{rmk}\label{rmk:DPLMV}
  As we were completing the first version of this article, we learned of the article \cite{DPLMV19-3} being prepared independently by Di Plinio, Li, Martikainen, and Vuorinen, in which Theorem \ref{thm:intro-BHT} is proven as a consequence of a multilinear Fourier multiplier theorem rather than embeddings into outer Lebesgue spaces.
  Their techniques are fundamentally different to ours: they reduce matters to the multilinear UMD-valued Calder\'on--Zygmund theory they developed in \cite{DPLMV19}, while our approach can be viewed as a direct (if long) reduction to linear Calder\'on--Zygmund theory facilitated by the outer Lebesgue framework.
  Although we do not prove results for multipliers other than $\BHF_{\Pi}$ in this paper, the embedding bounds we prove in Theorem \ref{thm:intro-embeddings} are sufficiently strong and modularised to be used as a black box in other time-frequency problems (such as bounds for variational Carleson operators in \cite{AU20-vC}, which we completed between the first and second versions of this article).
  We thank the authors of \cite{DPLMV19-3} for making their preprint available to us.
\end{rmk}

\subsection{The structure of the paper}
The paper is long, and perhaps intimidating.
We would like to reassure the reader that part of the reason for its length is that we have aimed for a complete exposition, and for self-containedness.

Here is a brief section-by-section overview of the paper, highlighting the placement of important theorems and concepts.

\textbf{Section \ref{sec:preliminaries}: Preliminaries.}
This section is dedicated to the introduction of key concepts and results in Banach-valued analysis, and of outer structures and outer Lebesgue spaces.
Key results in this section are
\begin{itemize}
\item Corollary \ref{cor:W11-Haar}, a Sobolev-type embedding into $\gamma$-spaces, adapted to the Haar measure on $\R_+$;
\item Theorem \ref{thm:WP-LP}, a continuous Littlewood--Paley estimate for UMD spaces;
\item Corollary \ref{cor:three-gamma-Haar}, a three-factor H\"older-type inequality for $\gamma$-norms.
\end{itemize}

\textbf{Section \ref{sec:time-frequency}: Analysis on the time-frequency-scale space.}
Section \ref{sec:trees-and-strips} introduces the geometric structures of time-frequency-scale analysis (trees, strips, and associated outer structures).
In Section \ref{sec:wave-packets-and-embeddings} we discuss functions $\R^3_+ \to \Lin(\Phi;X)$, where $\Phi$ is the space of band-limited wave packets, the most important example being the embedding $\Emb[f]$ of a function $\R \to X$.
The \emph{wave packet differentials} $\wpD_\theta$, $\wpD_\zeta$, $\wpD_\sigma$ appear in Definition \ref{def:wp-operators}: these act on arbitrary functions $\R^3_+ \to \Lin(\Phi;X)$ by a change of wave packet, exploiting the additional structure of $\Lin(\Phi;X)$-valued functions.
These are used to construct the \emph{defect operators} $\DEF_{\zeta}$, $\DEF_{\sigma}$ \eqref{eq:defect-definition}, which vanish on embedded functions $\Emb[f]$.
In Section \ref{sec:sizes} we define the local sizes of relevance to us, making full use of the structure of $\Lin(\Phi;X)$-valued functions.
The most novel of these local sizes are the \emph{defect local sizes}, which are introduced in order to control quantities that in more na\"ive arguments would lead to $R$-bounds, which are difficult or impossible to control without additional assumptions on the Banach spaces.
The local sizes sum up to form one big \emph{complete local size} $\FS^{s}$.

\textbf{Section \ref{sec:single-tree}: the single-tree estimate.}
Here the H\"older-type inequality \eqref{eq:intro-size-holder} is proven for the local sizes $\FS^{s}$.
In more familiar time-frequency analysis terms this is a \emph{single-tree estimate}: an estimate for the trilinear form $\BHF_{\Pi}$ restricted to a tree.
In the abstract outer Lebesgue framework, $\BHF_{\Pi}$ is replaced by $\BHWF_{\Pi}^{\phi_0,\mathbb{K}}$ and this trilinear form is evaluated on functions mapping $\R^3_+$ to $\Lin(\Phi;X)$.
The single-tree estimate involves a frequency-localised form which is in essence a multilinear Calderón-Zygmund form.
However, we establish bounds not in terms of classical $L^{p}$ norms, but in terms of sizes of $\Lin(\Phi,X)$-valued functions.
Our proof relies on integration by parts and wave packet decomposition arguments, exploiting the potential of varying wave packets in $\Lin(\Phi;X)$-valued functions. The defect operators and sizes allow us to control boundary terms arising from the integration by parts.

\textbf{Section \ref{sec:size-domination}: Local size domination.}
The previous section reduces matters to the boundedness of embedding maps detailed in Theorem \ref{thm:intro-embeddings}.
The outer Lebesgue spaces in these bounds are defined with respect to the complete local sizes $\FS^{s}$; recall that these are sums of smaller local sizes, including the defect local sizes.
In this section we show that while $\FS^{s}$ cannot be controlled by any of these smaller local sizes in general, it can be controlled by just one simpler local size, $\RS_{out}^2$, when restricted to truncations of embedded functions $\Emb[f]$.
The local size $\RS^{2}_{out}$ corresponds to a continuous tree-localised Littlewood--Paley-type square function.

Proposition \ref{prop:embedding-domination} furnishes the abstract mechanism for this domination principle, and Theorem \ref{thm:size-dom-master} is the conclusion of the aforementioned domination of $\FS^{s}$ by $\RS_{out}^{2}$ on embedded functions.
The results in between establish the requisite domination for the individual local sizes making up $\FS^{s}$.
The upshot of this section is that the embedding bounds of Theorem \ref{thm:intro-embeddings} need only be proven with respect to the smaller local size $\RS_{out}^{2}$.

\textbf{Section \ref{sec:embeddings-noniter}: Embeddings into non-iterated outer Lebesgue spaces.}
In this section we prove the first bounds of Theorem \ref{thm:intro-embeddings}, of the embedding map $\Emb$ into \emph{non-iterated} outer Lebesgue spaces $L_{\mu}^{p}\RS_{out}^2$ (note that the previous section lets us consider $\RS_{out}^2$ in place of $\FS^{s}$).
By interpolation it suffices to prove an $L^\infty$ endpoint (Proposition \ref{prop:non-iter-Linfty}, whose proof is an application of the UMD property) and a weak endpoint (Proposition \ref{prop:non-iter-weak}).
Our proof of the weak endpoint uses a generalised orthogonality property for embedded functions on collections of disjoint trees.
Such an estimate, which is similar to the notion of \emph{Fourier tile-type} introduced by Hyt\"onen and Lacey \cite{HL13}, needs geometric assumptions on the Banach spaces.
We prove it as a consequence of the assumed $r$-intermediate UMD property (Theorem \ref{thm:lac-type}), combining embedding bounds in the Hilbert-valued setting with a bilinear interpolation argument involving outer Lebesgue quantities. 

\textbf{Section \ref{sec:embeddings-iter}: Embeddings into iterated outer Lebesgue spaces.}
To obtain the full range of exponents in Theorem \ref{thm:intro-BHT} we need the second embedding bounds of Theorem \ref{thm:intro-embeddings}, involving the iterated outer Lebesgue spaces $L_{\nu}^{p} \sL_{\mu}^{q}\RS_{out}^2$.
These are proven by `localising' the embedding bounds into non-iterated outer Lebesgue spaces; the precise localisation required is established in Lemmas \ref{lem:localised-single-tree}, \ref{lem:localised-Lq-tail-lemma}, and \ref{lem:localised-Lpq-in-out-lemma}.

\textbf{Section \ref{sec:BHF}: Applications to $\BHF_{\Pi}$.}
Here we combine the results of the previous sections to prove Theorem \ref{thm:intro-BHT}.
In Section \ref{sec:sparse-dom} we prove sparse domination of $\BHF_{\Pi}$ as a consequence of the embedding bounds into iterated outer Lebesgue spaces.
This implies $L^p$-bounds for $\BHT_{\Pi}$ outside the range $p \in (1,\infty)$, stated in Theorem \ref{thm:BHT-bounds} (the bounds in the range $p \in (1,\infty)$ follow from Theorem \ref{thm:intro-BHT} by duality).
In Section \ref{sec:examples} we give examples of Banach spaces $(X_i)_{i=1}^{3}$ and trilinear forms $\Pi$ which do or do not fall within the scope of our results.
Finally, in Section \ref{sec:sharpness} we show that our results are essentially sharp within the general scope of bounded trilinear forms $\Pi$.

\subsection{Notation}

For topological vector spaces $X$ and $Y$ we let $\Lin(X,Y)$ denote the space of continuous linear operators from $X$ to $Y$, and we let $\Lin(X) := \Lin(X,X)$.
For $p \in [1,\infty]$, we let $L^p(\R;X)$ denote the Bochner space of strongly measurable functions $\R \to X$ such that the function $x \mapsto \|f(x)\|_X$ is in $L^p(\R)$; for technical details see \cite[Chapter 1]{HNVW16}.
The $X$-valued Schwartz space $\Sch(\R;X)$ consists of the smooth functions $\map{f}{\R}{X}$ such that $\sup_{z\in\R}\|\langle z \rangle^{N}\partial_{z}^{N} f(z)\|_{X}<\infty$ for each $N\in\N$, where we use the Japanese bracket notation
\begin{equation*}
  \langle x \rangle := (1 + |x|^2)^{1/2} \qquad \forall x \in \R.
\end{equation*}
We use the notation $\langle \cdot ; \cdot \rangle$ to denote the duality pairing between a Banach space $X$ and its dual $X^*$, as well as to denote the integral pairing between $f \in \Sch(\R;X)$ and $g \in \Sch(\R;\C)$ or $g \in \Sch(\R; X^{*})$:
\begin{equation*}
  \langle f; g \rangle := \int_\R f(x)g(x) \, \dd x \quad \text{or} \quad \int_\R \langle f(x) ; g(x) \rangle \, \dd x.
\end{equation*}
The correct interpretation will always be unambiguous.

We denote open balls in $\R$ by $B_{r}(x):=(x-r,x+r)$, and we write $B_{r}:= B_{r}(0)$. When $B \subset \R$ is a ball and $f \in L^p_{\loc}(\R;X)$ we let
\begin{equation*}
  \|f\|_{\sL^p(B;X)} := \Big( \frac{1}{|B|} \int_B \|f(x)\|_X^p \, \dd x \Big)^{1/p}
\end{equation*}
denote the $L^p$-average over $B$, and
\begin{equation*}
  M_p f(x) := \sup_{B \ni x} \|f\|_{\sL^p(B;X)} \qquad \forall x \in \R
\end{equation*}
denote the Hardy--Littlewood $p$-maximal operator.
As always, we write $M := M_1$.

For $p \in [1,\infty]$ we let $p'$ denote the conjugate exponent $p' := p/(p-1)$.
We say that a triple of exponents $(p_i)_{i=1}^3$ with $p_i \in [1,\infty]$ is a \emph{H\"older triple} if $\sum_{i=1}^3 p_i^{-1} = 1$.

{\footnotesize
\subsection{Acknowledgements}
We thank Mark Veraar, Christoph Thiele, and Francesco Di Plinio for their encouragement and suggestions.
We also thank Bas Nieraeth for discussions on sparse domination and extrapolation, and for their overall enthusiasm for the outer Lebesgue framework. 
Finally, we thank the anonymous referee for their detailed reading and comments.
The first author was supported by a Fellowship for Postdoctoral Researchers from the Alexander von Humboldt Foundation.
}


%% file: main/preliminaries.tex
We begin with a preparatory discussion of vector-valued analysis and outer Lebesgue spaces, including both new and well-known results.

\subsection{Rademacher sums}\label{sec:rademacher-sums}

First we introduce the fundamental `randomised norms' used in Banach-valued harmonic analysis: the \emph{Rademacher} (or $\varepsilon-$)norms.
For an in-depth discussion we refer the reader to \cite[Chapter 6]{HNVW17}.

Fix a Banach space $X$, and consider the quantity
\begin{equation*}
  \E \Big\| \sum_{n=1}^N \varepsilon_n \mb{x}_{n} \Big\|_X = \int_\Omega \Big\| \sum_{n=1}^N \varepsilon_n(\omega) \mb{x}_{n} \Big\|_X \, \dd \omega = \Big\| \sum_{n=1}^N \varepsilon_n \mb{x}_{n} \Big\|_{L^1(\Omega;X)}
\end{equation*}
where $(\mb{x}_{n})_{n=1}^N$ is a finite sequence of vectors in $X$ and $(\varepsilon_n)_{n=1}^N$ is a sequence of independent Rademacher variables (random variables taking the values $\pm 1$ with equal probability) on a probability space $\Omega$.
By the Kahane--Khintchine inequalities we have
\begin{equation*}
  \E \Big\| \sum_{n=1}^N \varepsilon_n \mb{x}_{n} \Big\|_X
  \simeq_p
  \Big( \E \Big\| \sum_{n=1}^N \varepsilon_n \mb{x}_{n} \Big\|_X^p \Big)^{1/p}
  = \Big\| \sum_{n=1}^N \varepsilon_n \mb{x}_{n} \Big\|_{L^p(\Omega;X)}
\end{equation*}
for all $p \in (0,\infty)$, with implicit constant independent of $N$.
Thus the convergence of an \emph{infinite} Rademacher sum $\sum_{n \in \N} \varepsilon_n \mb{x}_{n}$ in $L^p(\Omega;X)$ is independent of the choice of exponent $p \in (0,\infty)$, and with this mind we make the following definition.

\begin{defn}
  Given a Banach space $X$, we define $\varepsilon(\N; X)$ to be the Banach space of sequences $(\mb{x}_{n})_{n \in \N}$ in $X$ such that $\sum_{n=1}^\infty \varepsilon_n \mb{x}_{n}$ converges in $L^1(\Omega;X)$, equipped with the norm\footnote{In \cite{HNVW17} $L^2(\Omega;X)$ is used in place of $L^1(\Omega;X)$; of course, this leads to an equivalent norm.}
  \begin{equation*}
    \|(\mb{x}_{n})_{n \in \N} \|_{\varepsilon(X)} = \E \Big\| \sum_{n \in \N} \varepsilon_n \mb{x}_{n} \Big\|_X = \Big\| \sum_{n \in \N} \varepsilon_n \mb{x}_{n} \Big\|_{L^1(\Omega;X)}.
  \end{equation*}
\end{defn}

As explained in \cite[\textsection 3.1]{AU20-Walsh}, these norms play the role of discrete square functions for $X$-valued functions.
Indeed, when $X = \CC$ we have
\begin{equation*}
  \E \Big| \sum_{n \in \N} \varepsilon_n \mb{x}_{n} \Big| \simeq  \Big( \sum_{n \in \N} |\mb{x}_{n}|^2 \Big)^{1/2} = \|\mb{x}_{n}\|_{\ell^2(\CC)},
\end{equation*}
and when $X = L^q(S)$ for some measure space $S$ and some $q \in [1,\infty)$, for all sequences $(f_n)_{n \in \N}$ of functions in $L^q(S)$ we have the square function representation
\begin{equation*}
  \|(f_n)_{n \in \N} \|_{\varepsilon(L^p(S))} \simeq \|(f_n)_{n \in \N} \|_{L^p(S; \ell^2(\NN))} = \Big\| \Big( \sum_{n \in \N} |f_n|^2 \Big)^{1/2} \Big\|_{L^p(S)}.
\end{equation*}
However, for general Banach spaces $X$, no simpler representation of the space $\varepsilon(X)$ is available.
Properties of $\varepsilon(X)$, and of Rademacher sums in general, are strongly related to the concepts of type, cotype, and $K$-convexity.

\begin{defn}
  Let $p \in [1,2]$ and $q \in [2,\infty]$.
  A Banach space $X$ is said to have \emph{type $p$} and \emph{cotype $q$} if for all finite sequences $(\mb{x}_i)_{i=1}^N$ in $X$,
  \begin{equation*}
    \Big( \sum_{i=1}^N \|\mb{x}_i\|_X^q \Big)^{1/q} \lesssim \E \Big\| \sum_{i=1}^N \varepsilon_i \mb{x}_i \Big\|_X \lesssim \Big(\sum_{i=1}^N \|\mb{x}_i\|_X^p \Big)^{1/p},
  \end{equation*}
  where $(\varepsilon_i)_{i=1}^N$ is a sequence of independent Rademacher variables, with the usual modification when $q = \infty$.
  One says that $X$ has \emph{nontrivial type} if it has type $p$ for some $p > 1$, and \emph{finite cotype} if it has cotype $q$ for some $q < \infty$.
\end{defn}

\begin{defn}\label{def:K-convexity}
  A Banach space $X$ is \emph{$K$-convex} if for all finite sequences $(\mb{x}_{n})_{n=1}^N$ in $X$,
  \begin{equation*}
    \E \Big\| \sum_{n=1}^N \varepsilon_n \mb{x}_{n} \Big\|_X \simeq \sup_{(\mb{x}_{n}^*)} \Big| \sum_{n=1}^N \mb{x}_{n}^*(\mb{x}_{n}) \Big|
  \end{equation*}
  with implicit constant independent of $N$, where the supremum is taken over all sequences $(\mb{x}_{n}^*)_{n=1}^N$ in $X^*$ such that
  \begin{equation*}
    \E \Big\| \sum_{n=1}^N \varepsilon_n \mb{x}_{n}^* \Big\|_{X^*} = 1.
  \end{equation*}
\end{defn}

Every Banach space has type $1$ and cotype $\infty$. At the other extreme, a Banach space has both type $2$ and cotype $2$ if and only if it is isomorphic to a Hilbert space \cite[Theorem 7.3.1]{HNVW17}.
The Lebesgue space $L^p(\R)$ with $p \in [1,\infty)$ has type $\min(p,2)$ and cotype $\max(p,2)$.
However, $L^\infty(\R)$ is a special case, having neither nontrivial type nor finite cotype.
Every $K$-convex space has nontrivial type and finite cotype \cite[Propsition 7.4.12]{HNVW17}, and every UMD space is $K$-convex \cite[Proposition 4.3.10]{HNVW16}.
Throughout the paper we will try to state the minimal Banach space assumptions required for each result, but the reader should keep in mind that UMD spaces always satisfy these assumptions.

\subsection{\texorpdfstring{$\gamma$}{gamma}-radonifying operators and \texorpdfstring{$\gamma$}{gamma}-norms}\label{sec:gamma-norms}

In general, $\varepsilon(X)$ is like a better-behaved version of  $\ell^2(\N;X)$, as Rademacher sums are more sensitive to the geometry of $X$.
This analogy can be extended to more general measure spaces $S$ (in place of $\NN$), yielding spaces $\gamma(S;X)$ which play the role of $L^2(S;X)$---but again, better adapted to the geometry of $X$.
These spaces are constructed in terms of \emph{$\gamma$-radonifying operators}. 

\begin{defn}
  Let $H$ be a Hilbert space and $X$ a Banach space.
  A linear operator $\map{T}{H}{X}$ is called \emph{$\gamma$-summing} if
  \begin{equation*}
    \sup \E \Big\| \sum_{n=1}^N \gamma_n T\mb{h}_n \Big\|_X^2 < \infty,
  \end{equation*}
  where the supremum is taken over all finite orthonormal systems $(\mb{h}_n)_{n=1}^N$ in $H$ and $(\gamma_n)_{n=1}^N$ is a sequence of independent standard Gaussian random variables.
  If $T$ is $\gamma$-summing we define
  \begin{equation*}
    \|T\|_{\gamma_\infty(H,X)}^2 := \sup \E \Big\| \sum_{n=1}^N \gamma_n T\mb{h}_n \Big\|_X^2 
  \end{equation*}
  with supremum as before, yielding a Banach space $\gamma_\infty(H,X)$.
  All finite rank operators $H \to X$ are $\gamma$-summing, and we define $\gamma(H,X)$ to be the closure of the finite rank operators in $\gamma_\infty(H,X)$.
  The operators in $\gamma(H,X)$ are called \emph{$\gamma$-radonifying}.
\end{defn}

If $X$ does not contain a closed subspace isomorphic to $c_0$, and in particular if $X$ has finite cotype, then $\gamma(H,X) = \gamma_\infty (H,X)$ for all Hilbert spaces $H$.

\begin{defn}
  For a Banach space $X$ and a measure space $(S,\mc{A},\mu)$ we define
  \begin{equation*}
    \gamma(S;X) := \gamma(L^2(S),X).
  \end{equation*}
\end{defn}

We write $\gamma_\mu(S;X) := \gamma(S;X)$ when the measure needs to be emphasised.
This will often be done when considering subsets of $\R_+$, where both the Lebesgue measure $\dd t$ and the Haar measure $\dd t/t$ are relevant.

Elements of $\gamma(S;X)$ are by definition operators from $L^2(S)$ to $X$, but we can also interpret functions $\map{f}{S}{X}$ as members of $\gamma(S;X)$.
Recall that a function $\map{f}{S}{X}$ is \emph{weakly $L^2$} if for all $\mb{x}^* \in X^*$, the function $\map{\langle f, \mb{x}^* \rangle}{S}{\C}$ belongs to $L^2(S)$.
If $f$ is furthermore strongly $\mu$-measurable, then for all $g \in L^2(S)$ the product $gf$ is Pettis integrable, and we can define a bounded operator $\map{\II_f}{L^2(S)}{X}$, the \emph{Pettis integral operator} with kernel $f$, by
\begin{equation}
  \II_f g := \int _S g(s) f(s) \, \dd \mu(s).
\end{equation}

\begin{defn}
  For a function $\map{f}{S}{X}$, we write $f \in \gamma(S;X)$ to mean that the operator $\map{\II_f}{L^2(S)}{X}$ is in $\gamma(S;X) = \gamma(L^2(S),X)$, and we write
  \begin{equation*}
    \|f\|_{\gamma(S;X)} := \|\II_f\|_{\gamma(S;X)}.
  \end{equation*}
\end{defn}

Using this identification we can think of $\gamma(S;X)$ as a space of generalised $X$-valued functions on $S$.
When $X$ is a Hilbert space we indeed have coincidence of these spaces (see \cite[Proposition 9.2.9]{HNVW17}).

\begin{prop}\label{prop:gamma-l2-hilb}
  Let $H$ be a Hilbert space and $S$ a measure space.
  Then $\gamma(S;H) = L^2(S;H)$ with equivalent norms.
\end{prop}

As already mentioned, the space $\gamma(S;X)$ is the continuous analogue of the Rademacher space $\varepsilon(X)$; indeed, we have $\gamma(\N;X) = \varepsilon(X)$ with equivalent norms when $X$ has finite cotype.

Pushing the function space analogy further, the $\gamma$-norms satisfy the following H\"older-type property \cite[Theorem 9.2.14]{HNVW17}:

\begin{prop}\label{prop:gamma-holder}
  Let $(S,\mc{A},\mu)$ be a measure space and $X$ a Banach space.
  Suppose $\map{f}{S}{X}$ and $\map{g}{S}{X^*}$ are in $\gamma(S;X)$ and $\gamma(S;X^*)$ respectively.
  Then $\map{\langle f; g\rangle}{S}{\CC}$ is integrable, with
  \begin{equation*}
    \int_S |\langle f; g \rangle| \, \dd\mu \leq \|f\|_{\gamma(S;X)} \|g\|_{\gamma(S;X^*)}.
  \end{equation*}
  Conversely, if $X$ is $K$-convex, then a strongly measurable and weakly $L^2$ function $\map{f}{S}{X}$ is in $\gamma(S;X)$ if and only if there exists a constant $C < \infty$ such that
  \begin{equation*}
    \Big| \int_S \langle f; g \rangle \, \dd\mu \Big| \leq C\|g\|_{\gamma(S; X^*)}
  \end{equation*}
  for all $g \in L^2(S) \otimes Y$, where $Y$ is a closed norming subspace of $X^*$.
  In this case $\|f\|_{\gamma(S;X)} \lesssim_X C$.
\end{prop}

We will use the following form of the dominated convergence theorem for $\gamma$-norms, proved in \cite[Corollary 9.4.3]{HNVW17}.

\begin{prop}\label{prop:gamma-dominated-convergence}
  Let $(S,\mc{A},\mu)$ be a measure space and $X$ a Banach space.
  Consider a sequence of functions $\map{f_n}{S}{X}$, $n \in \NN$, and a function $\map{f}{S}{X}$ such that $\lim_{n \to \infty} \langle f_n; \mb{x}^* \rangle = \langle f; \mb{x}^* \rangle$ in $L^2(S)$ for all $\mb{x}^* \in X^*$.
  Suppose furthermore that there exists a function $F \in \gamma(S;X)$ such that
  \begin{equation*}
    \| \langle f_n; \mb{x}^* \rangle \|_{L^2(S)} \leq \| \langle F ; \mb{x}^* \rangle \|_{L^2(S)}
  \end{equation*}
  for all $n \geq 1$ and $\mb{x}^* \in X^*$.
  Then each $f_n$ is in $\gamma(S;X)$, and $f_n \to f$ in $\gamma(S;X)$.
\end{prop}

It is usually difficult to estimate $\gamma$-norms directly, but various embeddings of Sobolev and Hölder spaces can be used (see \cite[Section 9.7]{HNVW17}).
The only embedding we need requires no assumptions on $X$ and has a relatively simple proof, which we include for completeness (see \cite[Proposition 9.7.1]{HNVW17}).

\begin{prop}\label{prop:W11-gamma}
  If $f \in C^1(0,1;X)$ and $f' \in L_{\dd t}^1(0,1;X)$, then $f$ is in \\ $\gamma_{\dd t}(0,1;X)$ with
  \begin{equation*}
    \|f\|_{\gamma_{\dd t}(0,1;X)} \lesssim \int_0^1 \|f(t)\|_X + \|f^\prime(t)\|_X \, \dd t.
  \end{equation*}
\end{prop}

\begin{proof}
  First note that for all $t \in (0,1)$,
  \begin{equation*}
    f(t) = f(1_-) - \int_0^1 \1_{(0,s)}(t) f'(s) \, \dd s,
  \end{equation*}
  where $f(1_-) := \lim_{t \to 1} f(t)$, so that
  \begin{align*}
    \|f\|_{\gamma_{\dd t}(0,1;X)} &\leq \|f(1_-) \1_{(0,1)}  \|_{\gamma_{\dd t} (0,1;X)} + \int_0^1 \|f'(s)\1_{(0,s)}\|_{\gamma_{\dd t} (0,1;X)} \, \dd s \\
                                  &= \|f(1_-)\|_X \|\1_{(0,1)}\|_{L^2(0,1)} + \int_0^1 \|f'(s)\|_X \|\1_{(0,s)}\|_{L^2(0,1)} \, \dd s \\
                                  &\leq \|f(1_-)\|_X + \int_0^1 \|f'(s)\|_X \, \dd s.
  \end{align*}
  For all $t \in (0,1)$ the fundamental theorem of calculus implies
  \begin{equation*}
    \|f(1_-)\|_X \leq \|f(t)\|_X + \int_0^1 \|f'(s)\|_X \, \dd s,
  \end{equation*}
  and integrating over $t \in (0,1)$ completes the proof.
\end{proof}

The result above is adapted to the Lebesgue measure.
For the Haar measure $\dd t/t$ we have the following analogue, including a useful $L^\infty$ variant.

\begin{cor}\label{cor:W11-Haar}
  Let $0 \leq a < b \leq \infty$ and $f \in C^1(a,b;X)$.
  If $f$ and $t f^\prime(t)$ are in $L^1_{\dd t/t}(a,b;X)$, then $f \in \gamma_{\dd t/t}(a,b ; X)$ and
  \begin{equation*}
    \|f\|_{\gamma_{\dd t/t} (a,b ; X)} \lesssim \int_a^b \|f(t)\|_X + t\|f'(t)\|_X \, \frac{dt}{t}
  \end{equation*}
  with implicit constant independent of $a$ and $b$.
  Furthermore, if $0 < a < b < \infty$, then
  \begin{equation*}
    \|f\|_{\gamma_{\dd t/t} (a,b ; X)} \lesssim \log(b/a)\big(\|f\|_{L^\infty(a,b;X)} + \|tf'(t)\|_{L_{\dd t}^\infty(a,b;X)} \big). 
  \end{equation*}
\end{cor}

\begin{proof}
  The second statement follows from the first, and by approximation and changing variables it suffices to prove the first statement for $(a,b) = (0,1)$.
  Changing variables again, using the triangle inequality, and then using Proposition \ref{prop:W11-gamma},
  \begin{equation*} \begin{aligned}[t]
      & \|f\|_{\gamma_{\dd t/t} (0,1 ; X)} = \|f \circ \exp\|_{\gamma_{\dd t} (-\infty,0 ;X)}
      \\ &\leq \sum_{k=-\infty}^0 \|f \circ \exp\|_{\gamma_{\dd t} (k-1,k ;X)}\lesssim \sum_{k=-\infty}^0 \int_{k-1}^k \|(f \circ \exp)(t)\|_X + \|(f \circ \exp)'(t)\|_X \, \dd t
      \\ & = \int_{-\infty}^0 \|f(e^t)\|_X + e^t \|f'(e^t)\|_X \, \dd t = \int_0^1 \|f(t)\|_X + t\|f'(t)\|_X \, \frac{\dd t}{t}.
    \end{aligned} \end{equation*}
  This completes the proof.
\end{proof}

We use this corollary to bound $\gamma$-norms of quantities related to embedded functions $\Emb[f][\varphi] \colon \R^3_+ \to X$, such as the following tail estimate.

\begin{lem}\label{lem:gamma-tail-estimate}
  Let $X$ be a Banach space and $N \in \NN$, and suppose $\map{f}{\R}{X}$ is measurable and supported outside $B_2$.
  Fix $\zeta \in B_1$ and  $\phi \in \Sch(\R)$.
  Then
  \begin{equation*}
    \| \langle f; \Tr_\zeta \Dil_t \phi \rangle\|_{\gamma_{\dd t/t}(0,1;X)} \lesssim_{\phi,N} \int_{\RR \sm B_2} \|f(z)\|_X \langle z \rangle^{-N} \, \dd z,
  \end{equation*}
  and in particular
  \begin{equation*}
    \|\langle f; \Tr_\zeta \Dil_t \phi \rangle\|_{\gamma_{\dd t/t}(0,1;X)} \lesssim_{\phi} \|f\|_{L^\infty(\R;X)}.
  \end{equation*}
\end{lem}

\begin{proof}
  Consider the function $\map{F}{(0,1)}{X}$ given by $F(t) := \langle f; \Tr_\zeta \Dil_t \phi \rangle$.
  Then by Corollary \ref{cor:W11-Haar},
  \begin{equation}\label{eq:G-int}
    \big\| \langle f; \Tr_\zeta \Dil_{t} \phi \rangle \big\|_{\gamma_{\dd t/t}(0,1;X)}
    \lesssim \int_0^1 \|F(t)\|_X + t\|F^\prime(t)\|_X \, \frac{\dd t}{t}.
  \end{equation}

  We have
  \begin{equation*} \begin{aligned}[t]
      \|F(t)\|_X &= \|\langle f; \Tr_\zeta \Dil_t \phi \rangle\|_X
      \\ & \lesssim_{\phi,N} \int_{\RR \sm B_2}\|f(z)\|_X \frac{1}{t} \Big\langle \frac{z-\zeta}{t} \Big\rangle^{-N-1} \, \dd z \lesssim \int_{\RR \sm B_2} \|f(z)\|_X \frac{1}{t} \Big\langle \frac{z}{t} \Big\rangle^{-N-1} \, \dd z
    \end{aligned} \end{equation*}
  using that $|\zeta| < 1$ and $|z| > 2$.
  Next, we have
  \begin{equation*}
    F^\prime(t) = -t^{-1}\langle f, \Tr_\zeta \Dil_t ( \phi + \cdot \phi'(\cdot) ) \rangle.
  \end{equation*}
  Thus we can estimate
  \begin{equation*} \begin{aligned}[t]
      t\|F^\prime(t)\|_X &\lesssim  \int_{\RR \sm B_2} \|f(z)\|_X |\Tr_\zeta \Dil_t (\phi + \cdot \phi^\prime(\cdot))(z)| \, \dd z
      \\ &\lesssim_{\phi,N}  \int_{\RR \sm B_2} \|f(z)\|_X \frac{1}{t} \Big\langle \frac{z-\zeta}{t} \Big\rangle^{-N-1} \, \dd z  \lesssim \int_{\RR \sm B_2} \|f(z)\|_X \frac{1}{t} \Big\langle \frac{z}{t} \Big\rangle^{-N-1} \, \dd z
    \end{aligned} \end{equation*}
  again using that $t < 1$ and $z > 2$ in the last line.
  Putting this together,
  \begin{equation*} \begin{aligned}[t]
      \big\| \langle f; \Tr_\zeta \Dil_{t} \phi \rangle \big\|_{\gamma_{\dd t/t}(0,1;X)}
      &\lesssim \int_{\RR \sm B_2} \|f(z)\|_X  \int_0^1 \frac{1}{t} \Big\langle \frac{z}{t} \Big\rangle^{-N-1} \, \frac{\dd t}{t} \, \dd z \\
      &\lesssim  \int_{\RR \sm B_2} \|f(z)\|_X \langle z \rangle^{-N} \, \dd z
    \end{aligned} \end{equation*}
  as required.
\end{proof}

Finally, we list a convenient way of bounding operators between $\gamma$-spaces: the $\gamma$-extension theorem.
This says that operators that are essentially `scalar' and bounded on a scalar-valued $L^2$-spaces are automatically bounded on the corresponding $\gamma$-spaces, no matter what Banach space is involved.
This is a special case of \cite[Theorem 9.6.1]{HNVW17}.

\begin{prop}\label{prop:gamma-extension}
  Let $H$ be a Hilbert space and $U \in \Lin(H)$ a bounded linear operator on $H$.
  Then for all Banach spaces $X$, the tensor extension $\map{U \otimes I}{H \otimes X}{H \otimes X}$, defined on elementary tensors by $(U \otimes I)(\mb{h} \otimes \mb{x}) := U\mb{h} \otimes \mb{x}$, extends uniquely to a bounded linear opreator $\tilde{U} \in \Lin(\gamma(H^*, X))$ of the same norm, and for $T \in \gamma(H^*, X) \subset \Lin(H^*, X)$ we have $\tilde{U} T = T \circ U^*$.
\end{prop}

The statement of this result is a bit confusing at first, but thankfully we will only apply it in the following simple setting.

\begin{cor}\label{cor:gamma-convolution}
  Let $\map{k}{\R_+ \times \R_+}{\C}$ be an integral kernel such that there exists a function $g \in L^1_{\dd\sigma/\sigma}(\R_+)$ with
  \begin{equation}\label{eq:kernel}
    |k(\sigma, \tau)| \leq |g(\sigma/\tau)|.
  \end{equation}
  Then for all Banach spaces $X$ and all functions $f \in \gamma_{\dd\sigma/\sigma}(\R_+; X)$,
  \begin{equation*}
    \Big\| \int_0^\infty k(\sigma,\tau) f(\tau) \, \frac{\dd\tau}{\tau} \Big\|_{\gamma_{\dd\sigma/\sigma}(\R_+; X)} \leq \|g\|_{L_{\dd\sigma/\sigma}^1(\R_+)} \|f\|_{\gamma_{\dd\sigma/\sigma}(\R_+; X)}.
  \end{equation*}
\end{cor}

\begin{proof}
  For all $h \in L^2(\R_+)$ and $\mb{x} \in X$ we have
  \begin{equation*}
    \int_0^\infty k(\sigma, \tau) (h \otimes \mb{x})(\tau) \, \frac{\dd\tau}{\tau}
    = \Big( \int_0^\infty k(\sigma, \tau) h(\tau) \, \frac{\dd\tau}{\tau} \Big) \otimes \mb{x},
  \end{equation*}
  so the operator under consideration is the tensor extension of an integral operator on $L^2_{\dd\tau/\tau}(\R_+)$, which is bounded using \eqref{eq:kernel} and Young's convolution inequality.
  The result then follows from Proposition \ref{prop:gamma-extension}.
\end{proof}

\subsection{\texorpdfstring{$R$}{R}-bounds}\label{sec:R-bounds}

\begin{defn}
  Let $X$ and $Y$ be Banach spaces, and let $\mc{T} \subset \Lin(X,Y)$ be a set of operators.
  We say that $\mc{T}$ is \emph{$R$-bounded} if there exists a constant $C < \infty$ such that for all finite sequences $(T_n)_{n=1}^N$ in $\mc{T}$ and $(\mb{x}_{n})_{n=1}^N$ in $X$,
  \begin{equation*}
    \E \Big\|\sum_{n=1}^N \varepsilon_n T_n \mb{x}_{n} \Big\|_Y \leq C \E \Big\| \sum_{n=1}^N \varepsilon_n \mb{x}_{n} \Big\|_X
  \end{equation*}
  (recall that $\varepsilon_n$ denotes a sequence of independent Rademacher variables).
  The infimum of all possible $C$ in this estimate is called the $R$-bound of $\mc{T}$, and denoted by $R(\mc{T})$.
  If $\mc{T}$ is not $R$-bounded we write $R(\mc{T}) = \infty$.
\end{defn}

$R$-boundedness arises as a sufficient (and often necessary) condition in various operator-valued multiplier theorems.
See \cite[Theorem 9.5.1, Proposition 9.5.6, and Remark 9.5.8]{HNVW17} for the following theorem, and see \cite[Section 9.5.b]{HNVW17} for further information on metric measure spaces and strong Lebesgue points.

\begin{thm}[$\gamma$-multiplier theorem]\label{thm:gamma-multiplier}
  Let $X$ and $Y$ be Banach spaces with finite cotype, and $(S,\mc{A},\mu)$ a measure space.
  Let $\map{A}{S}{\Lin(X,Y)}$ be such that for all $\mb{x} \in X$ the $Y$-valued function $s \mapsto A(s)\mb{x}$ is strongly $\mu$-measurable, and that the range $A(S)$ is $R$-bounded.
  Then for every function $\map{f}{S}{X}$ in $\gamma(S;X)$, the function $\map{Af}{S}{Y}$ is in $\gamma(S;Y)$, and
  \begin{equation*}
    \|Af\|_{\gamma(S;Y)} \lesssim R(A(S)) \|f\|_{\gamma(S;X)}.
  \end{equation*}
  Conversely, let $(S, \mc{A}, \mu)$ be a non-atomic metric measure space, and suppose that $\map{A}{S}{\Lin(X,Y)}$ is strongly locally integrable.
  If there exists a constant $C < \infty$ such that for all $\mu$-simple functions $\map{f}{S}{X}$ the function $Af$ is in $\gamma(S;Y)$ with
  \begin{equation*}
    \|Af\|_{\gamma(S,Y)} \leq C\|f\|_{\gamma(S;X)},
  \end{equation*}
  then the set
  \begin{equation*}
    \mc{A} := \{A(s) : \text{$s \in S$ is a strong Lebesgue point for $A$} \} \subset \Lin(X,Y)
  \end{equation*}
  is $R$-bounded, with $R(\mc{A}) \lesssim C$.
  In particular, if $A$ is continuous, then $R(A(S)) \lesssim C$.
\end{thm}

One useful consequence is a contraction principle for $\gamma$-norms: since the family of scalar operators $\mc{S} := \{cI : c \in \CC, |c| \leq 1\} \subset \mc{L}(X)$ is $R$-bounded with $R(\mc{S}) \simeq 1$ by Kahane's contraction principle, we have
\begin{equation*}
  \|a f\|_{\gamma(S;X)} \lesssim \|f\|_{\gamma(S;X)}
\end{equation*}
for all $a \in L^\infty(X)$.
This is particularly useful when applied to characteristic functions, as it yields the quasi-monotonicity property
\begin{equation*}
  \|f\|_{\gamma(S';X)} = \|\1_{S'} f\|_{\gamma(S;X)} \lesssim \|f\|_{\gamma(S;X)}
\end{equation*}
whenever $S' \subset S$.

Given a uniformly bounded set of operators, one has to exploit additional structure of the set to establish $R$-boundedness.
Here we point out two useful methods.
First, the family of $L^1$ integral means of an operator-valued function with $R$-bounded range is $R$-bounded; this is a rewriting of \cite[Theorem 8.5.2]{HNVW17}.

\begin{prop}\label{prop:integral-means-Rbd}
  Let $X$ and $Y$ be Banach spaces, let $(S,\mc{A},\mu)$ be a measure space, and let $\map{f}{S}{\Lin(X,Y)}$ be strongly $\mu$-measurable.
  Then the family of $L^1$-integral means
  \begin{equation*}
    \Big\{ \mb{x} \mapsto \int_S \phi(s) f(s)\mb{x} \, \dd\mu(s) : \phi \in L^1(S), \|\phi\|_1 \leq 1 \Big\} \subset \Lin(X,Y)
  \end{equation*}
  has $R$-bound less than $R(f(S))$.
\end{prop}

One can also establish $R$-boundedness of the range of an operator-valued function with integrable derivative, analogously to Proposition \ref{prop:W11-gamma} for $\gamma$-norms.
The following proposition is a special case of \cite[Proposition 8.5.7]{HNVW17}.

\begin{prop}\label{prop:W11-rbd}
  Let $X$ and $Y$ be Banach spaces, and let $-\infty < a < b < \infty$.
  Suppose $\map{f}{(a,b)}{\Lin(X,Y)}$ is differentiable, with integrable derivative.
  Then
  \begin{equation*}
    R(f((a,b))) \leq  \|\lim_{t \downarrow a} f(t)\|_{\Lin(X,Y)} + \int_a^b \|f^\prime(t)\|_{\Lin(X,Y)} \, \dd t
  \end{equation*}
  where the limit is in the strong operator topology.
\end{prop}

We will need one more technical result, which gives an $R$-bound for the set of maps $X \to \gamma(H^*;X)$ given by tensoring with elements of $H$  \cite[Theorem 9.6.13]{HNVW17}.

\begin{thm}\label{lem:half-gamma-extension}
  Let $X$ be a Banach space with finite cotype and $H$ a Hilbert space.
  For each $\mb{h} \in H$ define the operator $\map{T_h}{X}{\gamma(H^*,X)}$ by $T_{\mb{h}} \mb{x} := \mb{h} \otimes \mb{x}$.
  Then the set of operators $\{T_{\mb{h}} : \|\mb{h}\|_H \leq 1 \}$ is $R$-bounded.
\end{thm}

We apply this to prove a continuous Littlewood--Paley estimate for UMD spaces, highlighting the use of both $\gamma$-norms and $R$-bounds.\footnote{Plenty of Littlewood--Paley bounds are known for UMD spaces, but this particular one doesn't seem to be explicitly written in the literature.}
First we recall one form of the operator-valued Mihlin multiplier theorem, proven in \cite[Corollary 8.3.11]{HNVW17}.

\begin{thm}\label{thm:mihlin}
  Let $X$ and $Y$ be UMD spaces, and $p \in (1,\infty)$.
  Consider a symbol $m \in C^1(\RR \sm \{0\} : \Lin(X,Y))$ such that
  \begin{equation*}
    R_m := R\big(\{m(\xi), \xi m^\prime(\xi) : \xi \in \RR \sm \{0\}\}\big) < \infty.
  \end{equation*}
  Then the Fourier multiplier $T_m$ with symbol $m$ is bounded $L^p(\R;X) \to L^p(\R;Y)$ with norm controlled by $R_m$.
\end{thm}

\begin{thm}[Continuous Littlewood--Paley estimate]\label{thm:WP-LP}
  Let $X$ be a UMD space and $p \in (1,\infty)$.
  Fix a Schwartz function $\psi \in \Sch(\R)$ such that $\hat{\psi}$ vanishes in an open neighbourhood of the origin.
  Then for all $f \in L^p(\R;X)$,
  \begin{equation*}
    \big\| \langle f, \Tr_z \Dil_t \psi \rangle \big\|_{L^p_{\dd z}(\RR;\gamma_{\dd t/t}
      (\R_+ ; X))} \lesssim_{X,p,\psi} \|f\|_{L^p(\R;X)}.
  \end{equation*}
\end{thm}

\begin{proof}
  Let $A_\psi$ denote the operator sending $X$-valued functions on $\RR$ to $X$-valued functions on $\RR \times \RR_+$, defined by
  \begin{equation*}
    (A_\psi f)(z,t) := \langle f, \Tr_z \Dil_t \psi \rangle = (f \ast \Dil_t \rho)(x) = T_{\hat{\rho}(t\cdot)} f(z)
  \end{equation*}
  where $\rho(z) := \psi(-z)$ and $T_{\hat{\rho}(t\cdot)}$ is the Fourier multiplier with symbol $\xi \mapsto \hat{\rho}(t\xi)$.
  This operator can be seen as a Fourier multiplier with symbol
  \begin{equation*}
    \map{m}{\RR}{\Lin(X;\gamma_{\dd t/t}(\R_+;X))}
  \end{equation*}
  defined by
  \begin{equation*}
    (m(\xi)\mb{x})(t) := \hat{\rho}(t\xi)\mb{x} \qquad (\xi \in \R, \mb{x} \in X, t \in \R_+).
  \end{equation*}
  The derivative $m'$ of this symbol is given by
  \begin{equation*}
    (m'(\xi)\mb{x})(t) = (\hat{\rho}(t\cdot))'(\xi)\mb{x}  = t\hat{\rho}'(t\xi)\mb{x} \qquad (\xi \in \R, \mb{x} \in X, t \in \R_+).
  \end{equation*}
  Thus for each $\xi \in \R$ and $\mb{x} \in X$ we have
  \begin{equation*}
    m_\psi(\xi)(\mb{x}) = \hat{\rho}(\cdot \xi) \otimes \mb{x} \quad \text{and} \quad \xi m_\psi^\prime(\xi)(\mb{x}) = \cdot \xi \hat{\rho}^\prime(\cdot \xi) \otimes \mb{x}.
  \end{equation*}
  By Lemma \ref{lem:half-gamma-extension}, since UMD spaces have finite cotype,
  \begin{align*}
    &R\big(\{m_\eta(\xi), \xi m_\eta'(\xi) : \xi \in \R \sm \{0\} \} \big) \\
    &\qquad \lesssim \sup_{\xi \in \R \sm \{0\}} \max( \|\hat{\rho}(t \xi)\|_{L^2_{\dd t/t}(\R_+)},  \| t\xi \hat{\rho}^\prime(t\xi)\|_{L_{\dd t/t}^2(\R_+)} ) \leq C_{\rho},
  \end{align*}
  controlling the $L^2$-norms by multiplicative invariance of the Haar measure and the fact that $\hat{\rho}$ and its derivative are both Schwartz and vanish near the origin.
  By the operator-valued Mihlin theorem \eqref{thm:mihlin}, this proves boundedness of $A_\psi$ from $L^p(\R;X)$ to $L^p(\R;\gamma_{\dd t/t}(\R_+;X))$, completing the proof.
\end{proof}

\subsection{Novelties in the trilinear setting}\label{sec:trilinear-novelties}

In the previous sections we considered linear vector-valued analysis: that is, we considered linear operators acting on functions valued in a single Banach space $X$.
In this article we are interested in bilinear operators and trilinear forms; thus we consider not only $X$-valued functions, but also the interaction between functions valued in different Banach spaces.

As in the introduction, we consider three Banach spaces $X_1, X_2, X_3$ unified by a bounded trilinear form $\map{\Pi}{X_1 \times X_2 \times X_3}{\CC}$.
In this section we establish a few useful lemmas which show how $\gamma$-norms and $R$-bounds interact with this structure.
First we introduce \emph{$R$-bounds with respect to a trilinear form}; we used this previously in \cite{AU20-Walsh}, but to our knowledge it was first used by Di Plinio and Ou in \cite{DPO18}.

\begin{defn}
  Let $\map{\Pi}{X_1 \times X_2 \times X_3}{\CC}$ be a bounded trilinear form on the product of three Banach spaces.
  Indexing $\{1,2,3\} = \{i,j,k\}$ arbitrarily, we define a map $\map{\iota_\Pi^i}{X_i}{\Lin(X_j, X_k^*)}$ by
  \begin{equation*}
    \langle \iota_\Pi^i(\mb{x}_{i} )(\mb{x}_{j} ) ; \mb{x}_{k} \rangle = \Pi(\mb{x}_{1},\mb{x}_{2},\mb{x}_{3}) 
  \end{equation*}
  for all $\mb{x}_{i} \in X_i$, $\mb{x}_{j} \in X_j$, and $\mb{x}_{k} \in X_k$.
  For all sets of vectors $V \subset X_i$ we define $R_\Pi(V) \in [0,\infty]$ to be the $R$-bound of the set of operators $\iota_{\Pi}^i(V) \subset \Lin(X_j, X_k^*)$, i.e.
  \begin{equation*}
    R_\Pi(V) := R(\iota_{\Pi}^i(V)).
  \end{equation*}
  Note that although the maps $\iota_\Pi^i$ depend on the indexing $\{1,2,3\} = \{i,j,k\}$, the resulting quantities $R_\Pi(V)$ do not.
\end{defn}

The $\gamma$-multiplier theorem implies a trilinear Hölder-type inequality for two $\gamma$-norms and an $R$-bound. Conversely, along with the duality relation for $\gamma$-norms in Proposition \ref{prop:gamma-holder}, it yields an equivalent representation of the $R$-bound as the smallest constant in such an inequality.\footnote{This could be taken as an alternative definition of the $R$-bound, which may be of interest in situations where $K$-convexity or finite cotype cannot be assumed.}

\begin{cor}\label{cor:R-gamma-tri}
  Let $(X_{i})_{i=1}^{3}$ be Banach spaces with finite cotype.
  Let $\map{\Pi}{X_1 \times X_2 \times X_3}{\CC}$ be a bounded trilinear form, and let $(S,\mc{A},\mu)$ be a measure space.
  Suppose that $\map{f_i}{S}{X_i}$ for each $i \in \{1,2,3\}$.
  Then for each $i$,
  \begin{equation} \label{eq:R-bound-via-tri}
    \Big| \int_S \Pi(f_1(s),f_2(s),f_3(s)) \, \dd \mu(s) \Big| \lesssim R_\Pi(f_i(S)) \prod_{j \neq i} \|f_j\|_{\gamma(S;X_j)}.
  \end{equation}
  Conversely, suppose that $(S, \mc{A}, \mu)$ is a metric measure space and that the spaces $(X_i)_{i=1}^{3}$ are $K$-convex.\footnote{The proof only requires the weaker assumption that one of the spaces $X_j$ with $j \neq i$ is $K$-convex.}
  Fix $i \in \{1,2,3\}$ and a strongly $\mu$-measurable $\map{f_i}{S}{X_i}$, and suppose that there exists $C < \infty$ such that
  \begin{equation} \label{eq:R-bound-via-tri-converse}
    \Big| \int_S \Pi(f_1(s),f_2(s),f_3(s)) \, \dd \mu(s) \Big| \leq C \prod_{j \neq i} \|f_j\|_{\gamma(S;X_j)}
  \end{equation}
  for all $f_j \in L^2(S) \otimes X_j$, $j \in \{1,2,3\} \sm \{i\}$.
  Then the set
  \begin{equation*}
    \mc{A} := \{f_i(s) : \text{$s \in S$ is a strong Lebesgue point for $f_i$} \} \subset X_i
  \end{equation*}
  satisfies $R_\Pi(\mc{A}) \lesssim C$.
  In particular, if $f_i$ is continuous, then $R_\Pi(f_i(S)) \lesssim C$.
\end{cor}

\begin{proof}
  To prove \eqref{eq:R-bound-via-tri}, let $\{j,k\} = \{1,2,3\} \sm \{i\}$ and estimate
  \begin{equation*} \begin{aligned}[t]
      &\Big| \int_S \Pi(f_1(s),f_2(s),f_3(s)) \, \dd\mu(s) \Big|  = \Big| \int_S \langle \iota_\Pi^i(f_i(s))(f_j(s)), f_k(s) \rangle \, \dd\mu(s) \Big|
      \\ & \qquad \leq \|\iota_\Pi^i(f_i)(f_j)\|_{\gamma(S;X_k^*)} \| f_k \|_{\gamma(S;X_k)} \lesssim R_\Pi(f_i(S)) \|f_j\|_{\gamma(S;X_j)}  \| f_k \|_{\gamma(S;X_k)}
    \end{aligned} \end{equation*}
  using the $\gamma$-Hölder inequality (Proposition \ref{prop:gamma-holder}) in the second line and the $\gamma$-multiplier theorem (Theorem \ref{thm:gamma-multiplier}) in the last line.
  For the converse statement, let $\map{A := f_i \circ \iota_\Pi^i}{S}{\Lin(X_j, X_k^*)}$, and let $\map{f_j}{S}{X_j}$ be a $\mu$-simple function.
  Then for all $f_k \in L^2(S) \otimes X_k \subset L^2(S) \otimes X_k^{**}$, \eqref{eq:R-bound-via-tri-converse} gives
  \begin{equation*} \begin{aligned}[t]
      & \Big| \int_S \langle Af_j(s); f_k(s) \rangle \, \dd\mu(s) \Big| = \Big| \int_S \langle \Pi(f_1(s), f_2(s), f_3(s)) \, \dd\mu(s) \Big|
      \\ &\qquad \leq C \|f_j\|_{\gamma(S;X_j)} \|f_k\|_{\gamma(S;X_k)} = C \|f_j\|_{\gamma(S;X_j)} \|f_k\|_{\gamma(S;X_k^{**})}.
    \end{aligned} \end{equation*}
  Since $X_k$ is a closed norming subspace of $X_k^{**}$, Proposition \ref{prop:gamma-holder} yields that
  \begin{equation*}
    \|Af_j\|_{\gamma(S;X_k^*)} \lesssim C\|f_j\|_{\gamma(S;X_j)},
  \end{equation*}
  and $R_\Pi(\mc{A}) \lesssim C$ then follows from Theorem \ref{thm:gamma-multiplier}.
\end{proof}

Thus $R$-bounds are analogous to $L^\infty$-norms in a trilinear sense, bearing in mind the H\"older inequality $L^\infty \times L^2 \times L^2 \to \C$.
In the discrete setting, where $\gamma$-norms are equivalent to Rademacher norms, something analogous to the sequence space embedding $\ell^2 \subset \ell^\infty$ can be shown to hold for $R$-bounds.
The following lemma appears in a more generalised form as \cite[Lemma 4.1]{DPLMV19}; for the reader's convenience we reproduce the proof of the relevant special case here.

\begin{lem}\label{lem:three-rad}
  Let $\map{\Pi}{X_1 \times X_2 \times X_3}{\CC}$ be a bounded trilinear form on the product of three Banach spaces.
  Then for all sequences $(\mb{x}_{i,n})_{n \in \N}$ in $X_i$, $i \in \{1,2,3\}$,
  \begin{equation*}
    \Big| \sum_{n \in \N} \Pi(\mb{x}_{1,n}, \mb{x}_{2,n}, \mb{x}_{3,n}) \Big| \lesssim_{\Pi}\prod_{i=1}^3 \|(\mb{x}_{i,n})\|_{\varepsilon(X_i)}.
  \end{equation*}
\end{lem}

Corollary \ref{cor:R-gamma-tri} in the case $S = \N$ says that the conclusion of Lemma \ref{lem:three-rad} can be rephrased as
\begin{equation*}
  R_\Pi\big( \{\mb{x}_{i,n} : n \in \N\} \big)
  \lesssim_{\Pi}
  \|(\mb{x}_{i,n})_{n \in \N}\|_{\varepsilon(X_i)},
\end{equation*}
which, as mentioned before, is analogous to the inclusion $\ell^2 \subset \ell^\infty$.

\begin{proof}[Proof of Lemma \ref{lem:three-rad}]
  Let $(\varepsilon_n)_{n \in \N}$ and $(\varepsilon'_n)_{n \in \N}$ be two sequences of independent Rademacher variables on different probability spaces $\Omega$ and $\Omega'$.
  Let $\E$ and $\E'$ denote the expectation with respect to the two probability spaces.
  Then independence, Cauchy--Schwartz, the Kahane--Khintchine inequalities, and the contraction principle imply
  \begin{equation*}
    \begin{aligned}
      &\Big| \sum_{n \in \N} \Pi(\mb{x}_{1,n}, \mb{x}_{2,n}, \mb{x}_{3,n}) \Big| \\
      &= \Big| \E \E' \Pi\Big(\sum_{n_1 \in \N} \varepsilon_{n_1} \mb{x}_{1,n_1}, \sum_{n_2 \in \N} \varepsilon_{n_2} \varepsilon_{n_2}' \mb{x}_{2,n_2}, \sum_{n_3 \in \N} \varepsilon_{n_3}' \mb{x}_{3,n_3} \Big) \Big| \\
      &\lesssim_{\Pi} \E \bigg[ \Big\| \sum_{n_1 \in \N} \varepsilon_{n_1} \mb{x}_{1,n_1} \Big\|_{X_1}
      \E' \Big( \Big\| \sum_{n_2 \in \N} \varepsilon_{n_2} \varepsilon_{n_2}' \mb{x}_{2,n_2} \Big\|_{X_2} \Big\| \sum_{n_3 \in \N} \varepsilon_{n_3}' \mb{x}_{3,n_3} \Big\|_{X_3} \Big) \bigg] \\
      &\leq \E \bigg[ \Big\| \sum_{n_1 \in \N} \varepsilon_{n_1} \mb{x}_{1,n_1} \Big\|_{X_1} 
      \Big(\E' \Big\| \sum_{n_2 \in \N} \varepsilon_{n_2} \varepsilon_{n_2}' \mb{x}_{2,n_2} \Big\|_{X_2}^2 \Big)^{1/2} \bigg]
      \Big(\E' \Big\| \sum_{n_3 \in \N} \varepsilon_{n_3}' \mb{x}_{3,n_3} \Big\|_{X_3}^2 \Big)^{1/2} \\
      &\simeq \E \bigg[ \Big\| \sum_{n_1 \in \N} \varepsilon_{n_1} \mb{x}_{1,n_1} \Big\|_{X_1}
      \|(\mb{x}_{2,n})_{n \in \N}\|_{\varepsilon(X_2)} \bigg]
      \|(\mb{x}_{3,n})_{n \in \N}\|_{\varepsilon(X_3)} 
      = \prod_{i=1}^3 \|(\mb{x}_{i,n})\|_{\varepsilon(X_i)}
    \end{aligned}
  \end{equation*}
  as claimed.
\end{proof}

We would like to prove a continuous version of Lemma \ref{lem:three-rad}, i.e. an estimate like
\begin{equation*}
  \Big| \int_\R \Pi(f_1(t), f_2(t), f_3(t)) \, \dd t \Big|
  \lesssim_\Pi
  \prod_{i=1}^3 \|f_i\|_{\gamma(\R;X_i)}.
\end{equation*}
Of course, this estimate is false, as can be seen by taking each $X_i$ to be $\C$ and $\Pi$ to be the natural product of scalars.
It can be salvaged by asking for the functions in question to be `almost constant at scale 1', quantified in terms of derivatives.

\begin{thm}\label{thm:three-gamma}
  Let $\map{\Pi}{X_1 \times X_2 \times X_3}{\CC}$ be a bounded trilinear form on the product of $K$-convex Banach spaces.
  Then for all $f_i \in C^1(\R;X_i)$ with $f_i, f_i' \in \gamma(\R;X_i)$,
  \begin{equation*}
    \Big| \int_\R \Pi(f_1(t), f_2(t), f_3(t)) \, \dd t \Big|
    \lesssim_\Pi
    \prod_{i=1}^3 \big(\|f_i\|_{\gamma(\R;X_i)}+\|f_i'\|_{\gamma(\R;X_i)} \big).
  \end{equation*}
\end{thm}

We will prove Theorem \ref{thm:three-gamma} as a consequence of two lemmas.

\begin{lem}\label{lem:rad-gamma-avg}
  Let $X$ be a $K$-convex Banach space and $f \in L^1_{\loc}(\R;X)$, and let
  \begin{equation*}
    f_n := \int_n^{n+1} f(t) \, \dd t \in X \qquad \forall n \in \Z.
  \end{equation*}
  Then
  \begin{equation*}
    \|(f_n)_{n \in \Z}\|_{\varepsilon(\Z; X)} \lesssim \|f\|_{\gamma(\R; X)}.
  \end{equation*}
\end{lem}

\begin{proof}
  $K$-convexity allows estimation of the Rademacher norm by duality with $\varepsilon(\Z,X^*)$.
  For all finitely supported sequences $(x^*_n)_{n \in \Z}$ in $\varepsilon(\Z;X^*)$ estimate
  \begin{equation*}
    \begin{aligned}
      \Big|\sum_{n \in \Z} \langle f_n ; \mb{x}_{n}^* \rangle \Big|
      &= \Big| \int_\R \langle f(t) ; \sum_{n \in \Z} \1_{[n, n+1)}(t) \mb{x}_{n}^* \rangle \, \dd t \Big| \\
      &\lesssim \|f\|_{\gamma(\R;X)} \Big\| \sum_{n \in \Z} \1_{[n, n+1)} \mb{x}_{n}^* \Big\|_{\gamma(\R;X^*)}.
    \end{aligned}
  \end{equation*}
  Since $(\1_{[n, n+1)})_{n \in \Z}$ is an orthonormal sequence in $L^2(\R)$, \cite[Proposition 9.1.3]{HNVW17} gives
  \begin{equation*}
    \Big\| \sum_{n \in \Z} \1_{[n, n+1)} \mb{x}_{n}^* \Big\|_{\gamma(\R;X^*)}
    \simeq \|(\mb{x}_{n}^*)_{n \in \Z} \|_{\varepsilon(\Z; X^*)},
  \end{equation*}
  completing the proof.
\end{proof}

\begin{lem}\label{lem:rad-gamma-deriv}
  Let $X$ be a $K$-convex Banach space and $f \in C^1(\R;X)$, and suppose that both $f$ and $f'$ are in $\gamma(\R;X)$.
  Then for all $t \in \R$ we have
  \begin{equation*}
    \|(f(n+t))_{n \in \Z}\|_{\varepsilon(\Z;X)} \lesssim \|f\|_{\gamma(\R;X)} + \|f'\|_{\gamma(\R;X)}.
  \end{equation*}
\end{lem}

\begin{proof}
  By translation invariance we may assume without loss of generality that $t = 0$.
  For all $n \in \N$ we can write
  \begin{equation*} \begin{aligned}[t]
      f(n) &= \int_n^{n+1} \Big( f(u) - \int_n^u f'(v) \, \dd v \Big) \, \dd u = f_n + \int_n^{n+1} f'(v) \int_v^{n+1} \, \dd u \, \dd v
      \\ &= f_n + \int_n^{n+1} f'(v) (n+1 - v) \, \dd v = f_n + g_n
    \end{aligned} \end{equation*}
  where $g(v) := (\lceil v \rceil - v)f'(v)$, $f_n := \int_n^{n+1} f(v) \, \dd v$, and $g_n := \int_n^{n+1} g(v) \, dv$.
  By Lemma \ref{lem:rad-gamma-avg} we find that
  \begin{equation*} \begin{aligned}[t]
      \|(f(n))_{n \in \Z}\|_{\varepsilon(\Z; X)} & \leq \|(f_n)_{n \in \Z}\|_{\varepsilon(\Z; X)} + \|(g_n)_{n \in \Z}\|_{\varepsilon(\Z; X)}
      \\ &\lesssim \|f\|_{\gamma(\R; X)} + \|g\|_{\gamma(\R;X)} \lesssim \|f\|_{\gamma(\R;X)} + \|f'\|_{\gamma(\R;X)}
    \end{aligned} \end{equation*}
  using Theorem \ref{thm:gamma-multiplier} and that $|\lceil v \rceil - v| \leq 1$ for all $v \in \R$ in the last line.
\end{proof}

\begin{proof}[Proof of Theorem \ref{thm:three-gamma}]
  Write
  \begin{equation*}
    \begin{aligned}
      \Big| \int_\R \Pi(f_1(t), f_2(t), f_3(t)) \, \dd t \Big| 
      &= \Big| \sum_{n \in \Z} \int_n^{n+1} \Pi(f_1(t), f_2(t), f_3(t)) \, \dd t \Big| \\
      &\leq \Big| \sum_{n \in \Z} \int_n^{n+1} \Pi(f_1(n), f_2(t), f_3(t)) \, \dd t \Big| \\
      &\qquad + \Big| \sum_{n \in \Z} \int_n^{n+1} \Pi(f_1(t) - f_1(n), f_2(t), f_3(t)) \, \dd t \Big|.
    \end{aligned}
  \end{equation*}
  By Corollary \ref{cor:R-gamma-tri} and Lemmas \ref{lem:three-rad} and \ref{lem:rad-gamma-deriv}, the first summand is bounded by
  \begin{equation*}
    \begin{aligned}
      &R_\Pi\big( f_1(\Z) \big)  \|f_2\|_{\gamma(\R; X_2)} \|f_3\|_{\gamma(\R; X_3)}  \\
      &\qquad \lesssim \|(f_1(n))_{n \in \Z}\|_{\varepsilon(\Z; X_1)} \|f_2\|_{\gamma(\R; X_2)} \|f_3\|_{\gamma(\R; X_3)} \\
      &\qquad \lesssim \big( \|f_1\|_{\gamma(\R; X_1)} + \|f_1'\|_{\gamma(\R; X_1)} \big) \|f_2\|_{\gamma(\R; X_2)} \|f_3\|_{\gamma(\R; X_3)}
    \end{aligned}
  \end{equation*}
  which is consistent with the estimate being proven.
  As for the second term, write
  \begin{equation*}
    \begin{aligned}
      &\Big| \sum_{n \in \Z} \int_n^{n+1} \Pi(f_1(t) - f_1(n), f_2(t), f_3(t)) \, \dd t \Big| \\
      &= \Big| \sum_{n \in \Z} \int_n^{n+1} \int_n^t \Pi(f_1'(s), f_2(t), f_3(t)) \, \dd s \, \dd t \Big| \\
      &= \Big| \int_0^{1}  \sum_{n \in \Z} \Pi\Big( \int_0^{u} f_1'(n+s) \dd s, f_2(n+u), f_3(n+u)\Big) \, \dd u\Big| \\
      &\leq \int_0^{1} \Big|  \sum_{n \in \Z} \Pi\Big( \int_0^{u}  f_1'(n+s) \dd s, f_2(n+u), f_3(n+u) \Big)  \Big| \, \dd u.
    \end{aligned}
  \end{equation*}
  For all $u \in (0,1)$, Lemmas \ref{lem:three-rad}, \ref{lem:rad-gamma-avg}, and \ref{lem:rad-gamma-deriv}, and Theorem \ref{thm:gamma-multiplier}, imply
  \begin{equation*}
    \begin{aligned}
      &\Big| \sum_{n \in \Z} \Pi\Big( \int_{0}^{u}f_1'(n+s)\dd s, f_2(n+u), f_3(n+u) \Big) \Big| \\
      &\lesssim \Big\|\Big(\int_{0}^{u}f_1'(n+s) \dd s\Big)_{n \in \Z}\Big\|_{\varepsilon(\Z;X_1)} \|(f_2(n+u))_{n \in \Z}\|_{\varepsilon(\Z; X_2)} \|(f_3(n+u))_{n \in \Z}\|_{\varepsilon(\Z; X_3)} \\
      &\lesssim \Big\|\Big(\sum_{n\in\N}\1_{[n,n+u)} \Big) f_1'\Big\|_{\gamma(\R; X_1)}  \prod_{i=1}^2 \big( \|f_i\|_{\gamma(\R; X_i)} + \|f_i'\|_{\gamma(\R; X_i)} \big) \\
      &\lesssim \|f_1'\|_{\gamma(\R; X_1)} \prod_{i=1}^2 \big( \|f_i\|_{\gamma(\R; X_i)} + \|f_i'\|_{\gamma(\R; X_i)} \big).
    \end{aligned}
  \end{equation*}
  Integrating in $u$ completes the proof.
\end{proof}

By using the change of variables $\sigma(t) = e^t$, which defines an isometric isomorphism $L^2(\R) \to L^2_{\dd \sigma/\sigma}(0,\infty)$, we obtain the following multiplicative version of Theorem \ref{thm:three-gamma}.
This argument is similar to that used to prove Corollary \ref{cor:W11-Haar}.

\begin{cor}\label{cor:three-gamma-Haar}
  Let $\map{\Pi}{X_1 \times X_2 \times X_3}{\CC}$ be a bounded trilinear form on the product of $K$-convex Banach spaces.
  Then for all $g_i \in C^1(\R_+;X_i)$ with $g_i, \sigma\partial_\sigma g_i(\sigma) \in \gamma_{\dd\sigma/\sigma}(\R_+;X_i)$,
  \begin{equation*} \begin{aligned}[t]
      & \Big| \int_0^\infty \Pi(g_1(\sigma), g_2(\sigma), g_3(\sigma)) \, \frac{\dd\sigma}{\sigma} \Big| \\
      &\qquad \lesssim_\Pi \prod_{i=1}^3 \big( \|g_i\|_{\gamma_{\dd\sigma/\sigma}(\R_+;X_i)} + \|\sigma \partial_\sigma g_i(\sigma)\|_{\gamma_{\dd\sigma/\sigma}(\R_+;X_i)} \big).
    \end{aligned}  \end{equation*}
\end{cor}

Finally we prove a technical lemma whose utility will be apparent in the proof of Theorem \ref{thm:size-holder} later on.

\begin{lem}\label{lem:3gamma-decomp}
  Let $\map{\Pi}{X_1 \times X_2 \times X_3}{\CC}$ be a bounded trilinear form on the product of $K$-convex Banach spaces.
  Let $\map{f_i}{\R_+}{X_i}$  be functions such that:
  \begin{itemize}
  \item $f_i$ is in $\gamma_{\dd\sigma/\sigma}(\R_+; X_i)$,
  \item $f_i = g_i + h_i$ for some functions $\map{g_i, h_i}{\R_+}{X_i}$, with $g_i$ and $\sigma \partial_\sigma g_i(\sigma)$ in $\gamma_{\dd\sigma/\sigma}(\R_+; X_i)$, and with $R_\Pi(h_i(\R_+)) < \infty$.
  \end{itemize}
  Then
  \begin{equation*}
    \begin{aligned}
      &\Big|\int_{\R_+} \Pi(f_1(\sigma), f_2(\sigma), f_3(\sigma)) \, \frac{\dd\sigma}{\sigma} \Big| \\
      &\lesssim \prod_{i=1}^3 \Big( \|f_i\|_{\gamma_{\dd\sigma/\sigma}(\R_+; X_i)} + \|g_i\|_{\gamma_{\dd\sigma/\sigma}(\R_+; X_i)} \\
      &\hspace{3cm}+ \|\sigma\partial_\sigma g_i(\sigma)\|_{\gamma_{\dd\sigma/\sigma}(\R_+; X_i)} + R_\Pi(h_i(\R_+)) \Big).
    \end{aligned}
  \end{equation*}
\end{lem}

\begin{proof}
  Abbreviating
  \begin{equation*}
    I(F,G,H) := \Big|\int_{\R_+} \Pi(F(\sigma), G(\sigma), H(\sigma)) \, \frac{\dd\sigma}{\sigma} \Big|,
  \end{equation*}
  we have by multilinearity and the triangle inequality
  \begin{equation*}
    \begin{aligned}
      I(f_1,f_2,f_3) &\leq I(g_1, f_2, f_3) + I(h_1, f_2, f_3) \\
      &\leq I(g_1, g_2, f_3) + I(g_1, h_2, f_3) + I(h_1, f_2, f_3) \\ 
      &\leq I(g_1, g_2, g_3) + I(g_1, g_2, h_3) + I(g_1, h_2, f_3) + I(h_1, f_2, f_3).
    \end{aligned}
  \end{equation*}
  The result then follows by applying Corollary \ref{cor:three-gamma-Haar} to the first term and Corollary \ref{cor:R-gamma-tri} to the remaining terms.
\end{proof}

\subsection{Outer Lebesgue spaces}\label{sec:outer-lebesgue-spaces}

In this section we give a brief overview of the definition and basic properties of abstract outer spaces and the associated outer Lebesgue quasinorms, which were introduced in \cite{DT15}. For a topological space $\OX$ we let $\Bor(\OX)$ denote the $\sigma$-algebra of Borel sets in $\OX$, and for a topological vector space $X$ we let $\Bor(\OX;X)$ denote the set of  Borel measurable functions $\OX \to X$.

\begin{defn}
  Let $\OX$ be a topological space.
  \begin{itemize}
  \item
    A \emph{$\sigma$-generating collection on $\OX$} is a subset $\OB \subset\Bor(\OX)$ such that $\OX$ can be written as a union of countably many elements of $\OB$.
    We write
    \begin{equation*}
      \OB^\cup := \bigg\{\bigcup_{n=1}^\infty B_n : \text{$B_n \in \OB$ for all $n \in \N$}\bigg\}.
    \end{equation*}
  \item
    A \emph{local measure} (on $\OB$) is a $\sigma$-subadditive function $\mu \colon\OB \to [0,+\infty]$ such that $\mu(\emptyset)=0$.

  \item Given a topological vector space $X$, an $X$-valued \emph{local size} (on $\OB$) is a collection of `quasi-norms' $\OS = (\| \cdot \|_{\OS(B)})_{B \in \OB}$, with each $\map{\|\cdot\|_{\OS(B)}}{\Bor(\OX;X)}{[0,+\infty]}$, satisfying
    \begin{description}
    \item[positive homogeneity]
      \begin{equation*}
        \| \lambda F \|_{\OS(B)}=|\lambda|\| F \|_{\OS(B)} \qquad \forall F \in \Bor(\OX;X) \quad \forall \lambda\in \C;
      \end{equation*}
    \item[global positive-definiteness]
      \begin{equation*}
        \text{$\|F\|_{\OS(B)}=0$ for all $B\in\OB$} \quad \Leftrightarrow \quad F = 0;
      \end{equation*}
    \item[quasi-triangle inequality] there exists a constant $C \in [1,\infty)$ such that for all $B \in \OB$,
      \begin{equation*}
        \| F+G \|_{\OS(B)}\leq C ( \| F \|_{\OS(B)}+\| G \|_{\OS(B)} ) \qquad \forall F,G \in \Bor(\OX;X).
      \end{equation*}
    \end{description}

  \item An ($X$-valued) \emph{outer space} is a tuple $(\OX,\OB, \mu, \OS)$ consisting of a topological space $\OX$, a $\sigma$-generating collection $\OB$ on $\OX$, a local measure $\mu$, and an $X$-valued local size $\OS$, all as above.
    We often do not explicitly reference $X$.
  \end{itemize}
\end{defn}

Consider an outer space $(\OX,\OB, \mu, \OS)$.
We extend $\mu$ to an outer measure on $\OX$ via countable covers: for all $E \subset \OX$,
\begin{equation}
  \label{eq:outer-measure}
  \mu(E):=\inf \bigg\{  \sum_{n\in\N}\mu(B_{n}) : B_{n}\in\OB,\, \bigcup_{n\in\N} B_{n}\supset E\bigg\}.
\end{equation}
We abuse notation and write $\mu$ for both the local measure and the corresponding outer measure.
We define the \emph{outer size} (or \emph{outer supremum}) of $F\in \Bor(\OX;X)$ by
\begin{equation}
  \label{eq:full-size}
  \| F \|_{\OS} := \sup_{B\in\OB}\sup_{V\in\OB^{\cup}}\| \1_{\OX \setminus V}F \|_{\OS(B)}.
\end{equation}
We say that two local sizes $\OS_1$, $\OS_2$ on $\OB$ are \emph{equivalent} if
\begin{equation*}
  \| F\|_{\OS_1} \simeq \|F \|_{\OS_2}
\end{equation*}
for all $F \in \Bor(\OX;X)$.
The conjunction of the notions of outer measure and outer size allows us to define the \emph{outer super-level measure} of a function $F\in \Bor(\OX;X)$ as
\begin{equation}
  \label{eq:superlevel-measure}
  \mu ( \| F \|_{\OS}>\lambda ) := \inf \{\mu(V) : V\in\OB^{\cup},\, \| \1_{\OX\setminus V}F \|_{\OS}\leq \lambda \}.
\end{equation}
This quantity need not be the measure of any specific set.
Instead, it is an intermediate quantity between the outer measure $\mu$ and the outer size $\OS$.
For any $F \in \Bor(\OX;X)$ we define
\begin{equation}
  \label{eq:spt-measure}
  \mu(\spt(F)):= \mu \bigl( \| F \|_{\OS}>0 \bigr).
\end{equation}

\begin{defn}
  Let $(\OX,\OB,\mu,\OS)$ be an $X$-valued outer space. We define the \emph{outer Lebesgue quasinorms} of a function $F\in\Bor(\OX;X)$, and their weak variants, by setting
  \begin{align*}
    &  \|F\|_{L_{\mu}^{p} \OS} := \Big( \int_{0}^{\infty} \lambda^{p} \mu(\| F \|_{\OS} > \lambda) \, \frac{\dd \lambda}{\lambda} \Big)^{1/p}   &&\forall p\in(0,\infty), \\
    &  \|F\|_{L_{\sigma}^{p,\infty} \OS} := \sup_{\lambda > 0} \lambda \,\mu(\| F \|_{\OS} > \lambda)^{1/p} &&\forall p\in(0,\infty), \\
    &\|F\|_{L_{\mu}^{\infty} \OS} := \| F \|_{\OS}. 
  \end{align*}
  It is straightforward to check that these are indeed quasinorms (modulo functions $F$ with $\mu(\spt (F)) = 0$). 
\end{defn}

One can construct \emph{iterated outer spaces} by using an outer Lebesgue quasinorm itself to define a local size.

\begin{defn}
  Let $(\OX, \OB, \mu, \OS)$ be an outer space.
  Let $\OB'$ be a $\sigma$-generating collection on $\OX$, and let $\nu$ be a local measure on $\OB'$.
  Then for all $q \in (0,\infty)$ define the \emph{iterated local size} $\sL_\mu^q \OS$ on $\OB'$ (which depends on $\nu$) by
  \begin{equation}
    \label{eq:iterated-size}
    \| F \|_{\sL^{q}_{\mu}\OS (B')} := \frac{1}{\nu(B')^{1/q} }\| \1_{B'} F \|_{L^{q}_{\mu} \OS} \qquad(B' \in \OB').
  \end{equation}
  It is straightforward to check that $\sL^{q}_{\mu} \OS$ is a local size on $\OB'$, so that $(\OX,\OB',\nu,\sL^{q}_{\mu} \OS)$ is an outer space.
\end{defn}

We will use a few key properties of outer Lebesgue quasinorms.
The following Radon--Nikodym-type domination result lets us compare classical Lebesgue integrals with outer Lebesgue quasinorms. The proof is a straightforward modification of \cite[Lemma 2.2]{gU16} and \cite[Proposition 3.6]{DT15}

\begin{prop}\label{prop:outer-RN}
  Let $(\OX,\OB,\mu,\OS)$ be an outer space such that the outer measure generated by $\mu$ is $\sigma$-finite. Let $\mf{m}$ be a positive Borel measure on $\OX$ such that there exists a constant $C_{\mf{m}} < \infty$ with 
  \begin{equation*}
    \Big\|\int_B  F(x) \, \dd \mf{m} (x) \Big\|_{X} \leq C_{\mf{m}} \| F \|_{\OS(B)} \mu(B) \qquad \forall B \in \OB, \, \forall F \in L^1(B,\mf{m}|_B;X)
  \end{equation*}
  and
  \begin{equation*}
    \mu(A) = 0 \, \Rightarrow \, \mf{m}(A) = 0 \qquad  \forall A \in \Bor(\OX). 
  \end{equation*}
  Then we have 
  \begin{equation*}
    \Big\|\int_\OX  F(x)  \, \dd \mf{m}(x)\Big\|_{X}  \lesssim C_{\mf{m}} \|F\|_{L^1_\mu \OS} \qquad \forall F \in L^1(\OX,\mf{m};X).
  \end{equation*}
\end{prop}

The previous proposition is usually followed by the following \emph{outer Hölder inequality}.
A slightly weaker version has been proven before, but we need a version that supports multiple outer spaces.

\begin{prop}\label{prop:outer-holder}
  For each $i \in \{0,1,\ldots,M\}$ let $X_i$ be a topological vector space and let $(\OX_i,\OB_i,\mu_i,\OS_i)$ be an $X_i$-valued outer space.
  Let
  \begin{equation*}
    \mb{\Pi} \colon \prod_{i=1}^M \Bor(\OX_i;X_i) \to \Bor(\OX_0;X_0)
  \end{equation*}
  be an $M$-sublinear map, and suppose that there is a constant $C < \infty$ such that for all $i \in \{1,\ldots,M\}$ and $F_i \in \Bor(\OX_i;X_i)$, the outer sizes $\OS_i$ and outer measures $\mu_i$ satisfy the bounds
  \begin{equation}
    \label{eq:abstract-size-holder}
    \| \mb{\Pi}(F_1,\ldots,F_M) \|_{\OS_0} \leq C \prod_{i=1}^{M} \| F_i \|_{\OS_{i}}
  \end{equation}
  and
  \begin{equation*}
    \mu_0(\spt \mb{\Pi}(F_1,\ldots,F_M)) \leq C \min_{i=1,\ldots,M} \mu_i(\spt F_i). 
  \end{equation*}
  Then for all $p_i \in (0,\infty]$ we have 
  \begin{equation}
    \label{eq:abstract-Lp-holder}
    \| \mb{\Pi}(F_1,\ldots,F_M) \|_{L_{\mu_0}^{p_0} \OS_0} \lesssim_{p_i} C \prod_{i=1}^{M} \| F_{i} \|_{L_{\mu_i}^{p_{i}} \OS_{i}} 
  \end{equation}
  with $p_0^{-1}=\sum_{i=1}^{M}p_{i}^{-1}$.
\end{prop}

\begin{proof}
  Assume that the factors on the right hand side of \eqref{eq:abstract-Lp-holder} are finite and non-zero, for otherwise there is nothing to prove.
  By homogeneity we may assume that $\| F_{i} \|_{L^{p_{i}}_{\mu_{i}}\OS_{i}} =1$ for each $i\in\{1,\dots,M\}$.  For each such $i$ and all $n\in\Z$, let $A_{i}^{n}\subset\OX$ be such that
  \begin{equation*}
    \sum_{n\in\Z} 2^{n} \mu_{i}(A_{i}^{n}) \lesssim 1 \quad \text{and} \quad \| \1_{\OX\setminus A_{i}^{n}} F_{i}\|_{\OS_{i}}\lesssim2^{n/p_{i}}.
  \end{equation*}
  We may assume that $A_{i}^{n}\subset A_{i}^{n-1}$ by considering $\tilde{A}_{i}^{n}=\bigcup_{k\geq n} A_{i}^{k}$ and noticing that $\tilde{A}_{i}^{n}$ satisfies the conditions above.  Let
  \begin{equation*}
    F_{i}^{<n} := \1_{X\setminus A_{i}^{n}} \qquad   F_{i}^{>n} := \1_{A_{i}^{n}}  F_{i}
  \end{equation*}
  so that
  \begin{equation*}
    \mb{\Pi}(F_1,\ldots,F_M)  = \sum_{\square \in \{<,>\}^{M}}     \mb{\Pi}(F_1^{\square_{1}n},\ldots,F_M^{\square_{M}n}).
  \end{equation*}
  By assumption we have that 
  \begin{equation*}
    \| \mb{\Pi}(F_1^{<n},\ldots,F_M^{<n})\|_{\OS_{0}}\lesssim C 2^{n \sum_{i=1}^{M}p_{i}^{-1}},
  \end{equation*}
  while for any $\square\ne (<, <,\dots,<,<)$ it holds that 
  \begin{equation*}
    \mu_{0}\big(\spt \mb{\Pi}(F_1^{\square_{1}n},\ldots,F_M^{\square_{M}n})\big)\lesssim C \max_{i\in\{1,\dots,M\}}\big(\mu_{i}(A_{i}^{n})\big).
  \end{equation*}
  Thus
  \begin{equation*}
    \mu_{0}\Big(\| \mb{\Pi}(F_1^{<n},\ldots,F_M^{<n})\|_{\OS_{0}}>2^{n\sum_{i=1}^{M}p_{i}^{-1}}\Big)\lesssim C \sum_{i=1}^{M}\mu_{i}(A_{i}^{n}),
  \end{equation*}
  which concludes the proof.
\end{proof}

After using an outer Hölder inequality, one typically needs to estimate outer Lebesgue quasinorms.
This can be done by interpolation, using either or both of the following two results.
The first is proven in \cite[Proposition 3.3]{DT15}, and the second in \cite[Proposition 3.5]{DT15} (see also \cite[Proposition 7.4]{DPO18}).

\begin{prop}[Logarithmic convexity]
  Let $(\OX, \OB, \mu, \OS)$ be an $X$-valued outer space.
  Then for any $p_0, p_1 \in (0,\infty]$ with $p_0 < p_1$ and any $\theta \in (0,1)$, we have
  \begin{equation*}
    \|F\|_{L^{p_\theta}_\mu \OS} \lesssim_{p_0,p_1,\theta} \|F\|_{L^{p_0,\infty}_\mu \OS} \|F\|_{L^{p_1,\infty}_\mu \OS}
  \end{equation*}
  for all $F \in \Bor(\OX;X)$, where $p_\theta = [p_0,p_1]_\theta$.
\end{prop}

\begin{prop}[Marcinkiewicz interpolation]\label{prop:outer-interpolation}
  Let $(\OX, \OB, \mu, \OS)$ be an $X$-valued outer space.
  Let $\Omega$ be a $\sigma$-finite measure space, and let $T$ be a quasi-sublinear operator mapping $L^{p_0}(\Omega;X) + L^{p_1}(\Omega;X)$ into $\Bor(\OX;X)$ for some $1 \leq p_0 < p_1 \leq \infty$.
  Suppose that
  \begin{equation*}
    \begin{aligned}
      \|Tf\|_{L_\mu^{p_0,\infty} \OS} &\lesssim \|f\|_{L^{p_0}(\Omega;X)}, \\
      \|Tf\|_{L_\mu^{p_1,\infty} \OS} &\lesssim \|f\|_{L^{p_1}(\Omega;X)}
    \end{aligned}
    \qquad \forall f \in L^{p_0}(\Omega;X) + L^{p_1}(\Omega;X).
  \end{equation*}
  Then for all $p \in (p_0,p_1)$,
  \begin{equation*}
    \|Tf\|_{L^p_\mu \OS} \lesssim \|f\|_{L^p(\Omega;X)} \qquad \forall f \in L^p(\Omega;X).
  \end{equation*}
\end{prop}


%% file: main/time-frequency.tex
By the \emph{time-frequency-scale space} we mean $\RR^3_+$, whose points parametrise the operators $\Lambda_{(\eta,y,t)} = \Tr_y \Mod_\eta \Dil_t$ representing the fundamental symmetries of $\BHF_\Pi$.
It is natural to think of $\R^3_+$ as a metric space, equipped with the pushforward of the Euclidean metric on $\R^3$ by the map $(x,y,\tau) \mapsto (e^{\tau} x,e^{-\tau} y,e^\tau)$.
This metric does not play an important role in our analysis, but it is worth keeping in mind.
For $(\xi,x,s)\in\R^{3}_{+}$ we define mutually inverse \emph{local coordinate maps} $\pi_{(\xi,x,s)}$, $\pi_{(\xi,x,s)}^{-1}$, both mapping $\R^{3}_{+}$ to itself, by
\begin{equation}
  \label{eq:chart-R3}
  \begin{aligned}
    & \pi_{(\xi,x,s)}(\theta,\zeta,\sigma)= \bigl( \xi+\theta (s \sigma)^{-1},x+s \zeta,  s \sigma\bigr),
    \\
    & \pi_{(\xi,x,s)}^{-1}(\eta,y,t) := \Bigl(t(\xi-\eta),\frac{y-x}{s}, \frac{t}{s}  \Bigr).
  \end{aligned}
\end{equation}

With a view towards applications to the bilinear Hilbert transform, we fix a small parameter $\mf{b}>0$ and an bounded open interval $\Theta$ (a frequency band) with $B_{2\mf{b}} \subsetneq B_{1}\subsetneq \Theta$.
The constructions below depend on both of these choices.
In applications we will need multiple choices of $\Theta$, so we sometimes reference it in the notation, but only when the particular choice of $\Theta$ is important.
We will only ever need one choice of $\mf{b}$ (and $\mf{b} = 2^{-4}$ will do), so we will always suppress it.

\subsection{Trees and strips}\label{sec:trees-and-strips}

\begin{defn}\label{def:tree}
  Define the \emph{model tree} by  
  \begin{equation}
    \label{eq:model-tree}
    \begin{aligned}
      &\mT_{\Theta}:=\big\{(\theta,\zeta,\sigma)\in \R^{3}_{+}\colon \theta \in \Theta,\,|\zeta|<1-\sigma \big\}
    \end{aligned}
  \end{equation}
  and define the \emph{tree with top $(\xi,x,s) \in \R^{3}_{+}$} to be the set
  \begin{equation*}
    T_{(\xi,x,s),\Theta} := \pi_{(\xi,x,s)}(  \mT_{\Theta}).
  \end{equation*}
  For a tree $T = T_{(\xi,x,s),\Theta}$ we use the shorthand $\pi_T := \pi_{(\xi,x,s)}$ and $(\xi_T, x_T, s_T) = (\xi,x,s)$.
  We define the \emph{inner} and \emph{outer} parts of $T$ by
  \begin{align*}
    T^{in}&:=\pi_{(\xi_{T},x_{T},s_{T})}\big( \mT_{\Theta} \cap \{\theta \in B_{2\mf{b}}\} \big), &
                                                                                                    T^{out}&:=\pi_{(\xi_{T},x_{T},s_{T})}\big( \mT_{\Theta} \cap \{\theta\in \Theta \setminus B_{2\mf{b}}\} \big),
  \end{align*}
  and we denote the family of all trees by $\TT_{\Theta}$.
  The set $\TT_{\Theta}$ is a $\sigma$-generating collection on $\R^{3}_+$, and we define a local measure $\mu_\Theta$ on $\TT_{\Theta}$ by
  \begin{equation}\label{eq:TT-premeasure}
    \mu_\Theta(T)  := s_T.
  \end{equation}
\end{defn}

See Figure \ref{fig:trees} for a sketch of the model tree $\mT_\Theta$, and Figure \ref{fig:tree-freq} for how two trees look in local coordinates with respect to one of them.

\begin{figure}[h]
  \includegraphics{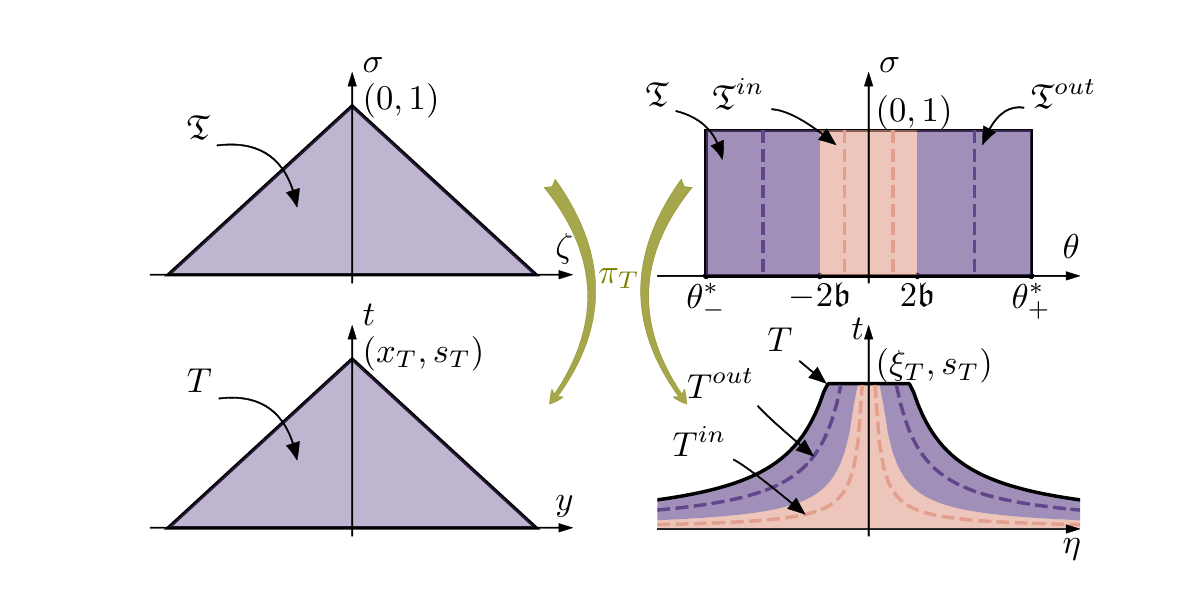}
  \caption{
    Top: the model tree $\mT_{\Theta}$ in the time-scale and frequency-scale planes.
    Bottom: the tree $T_{(\xi_{T},x_{T},s_{T}),\Theta}$ in the $\eta=\xi_{T}$ and $y=x_{T}$ planes.
  }
  \label{fig:trees}
\end{figure}

\begin{figure}[h]
  \includegraphics{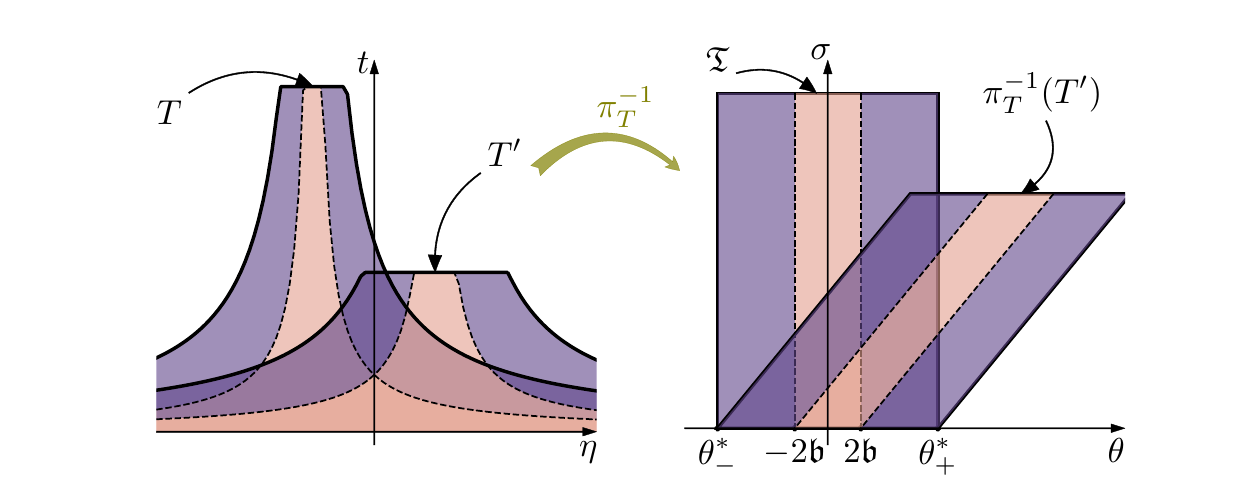}
  \caption{Left: Two trees in the frequency-scale plane.
    Right: The same two trees, viewed in local coordinates with respect to $T$, in the frequency-scale plane.
  }
  \label{fig:tree-freq}
\end{figure}

A tree $T$ represents a region of time-frequency-scale space in which frequency is localised around $\xi_T$ (with precision measured by the rescaled frequency band $s_T^{-1} \Theta$), time is approximately localised to $B_{s_{T}}(x_T)$,\footnote{The time variable can also be interpreted as a spatial variable.} and the maximum scale is $s_T$.
Time-frequency analysis restricted to a single tree essentially corresponds to Calder\'on--Zygmund theory; handling the contributions of multiple trees is the main difficulty.

\begin{defn}\label{def:strip}
  Define the \emph{model strip}  by 
  \begin{equation}
    \label{eq:model-strip}
    \begin{aligned}
      &\mD:=\big\{   (\zeta,\sigma)\in \R^{2}_{+}\colon |\zeta|<1-\sigma \big\},
    \end{aligned}
  \end{equation}
  and define the \emph{strip with top $(x,s) \in \R^{2}_+$}to be the set
  \begin{equation}\label{eq:strip}
    D_{(x,s)} := \pi_{(0,x,s)}(\RR \times \mD).
  \end{equation}
  We let $(x_D,s_D) := (x,s)$, and we denote the family of all strips by $\DD$.
  Of course $\DD$ is a $\sigma$-generating collection on $\R^{3}_{+}$, and we define a local measure $\nu$ on $\DD$ by 
  \begin{equation}\label{eq:DD-premeasure}
    \nu(D_{(x,s)}):= s.
  \end{equation}
\end{defn}

A strip $D$ represents a region of $\R^3_+$ in which time is localised to $B(x_D,s_D)$ and the maximum scale is $s_D$, while frequency is unrestricted.
Note that
\begin{equation*}
  D_{(x,s)}=\bigl\{(\eta,y,t)\in \R^{3}_{+} \colon |y-x|<s-t\bigr\}=\bigcup_{\xi\in\Q} \pi_{(\xi,x,s)}(\mT_{\Theta}),
\end{equation*}
so in particular each strip can be written as a countable union of trees, and it follows that $\DD^{\cup}\subset \TT_{\Theta}^{\cup}$.

\subsection{Wave packets and embeddings}
\label{sec:wave-packets-and-embeddings}

Let $X$ be a Banach space.
As discussed in the introduction, we will consider not only $X$-valued functions but also $\Lin(\Phi;X)$-valued functions, where $\Phi$ is a space of testing wave packets.

\begin{defn}\label{dfn:WP-space}
  Let
  \begin{equation*}
    \begin{aligned}
      \Phi :=\big\{ \phi \in \Sch(\R): \spt\FT{\phi}\subset [-\mf{b},\mf{b}]\big\},
    \end{aligned}
  \end{equation*}
  and equip $\Phi$ with the Fréchet topology induced by the norms  $\|\phi\|_{N}=\|\FT{\phi}\|_{C^{N}}$ for $N \in \N$.
  We let
  \begin{equation*}
    \Phi^{N}_{1}:= \{\phi\in\Phi\colon \| \phi \|_{N} \leq 1\}.
  \end{equation*}
  Often we write $\Phi_{1} := \Phi_{1}^{N}$ when the natural number $N$ is understood from context (generally it is fixed, and very large).
  For $\theta \in \R$ define $\Psi(\theta)\subset\Phi$ by
  \begin{equation}
    \label{eq:lacunary_wave_packets}
    \Psi(\theta):=\{\psi\in\Phi\colon \FT{\psi}(-\theta)=0\}
  \end{equation}
  and let  $\Psi^{N}_{1}(\theta):= \Psi \cap \Phi^{N}_{1}$.
  Note that $\Psi(\theta) = \Phi$ for $\theta \notin B_{\mf{b}}$.
\end{defn}

Having defined the testing wave packet space $\Phi$, we view the embedding \eqref{eq:embedding} as a map
\begin{equation*}
  \map{\Emb}{\Sch(\R;X)}{\Bor(\R^3_+; \Lin(\Phi; X))}
\end{equation*}
where $\Lin(\Phi; X)$ is endowed with the weak-$*$ topology.
Now consider another topological vector space $Y$ (not necessarily $X$ or $\Lin(\Phi;X)$).
Given a tree $T \in \TT_\Theta$ and a function $F \in \Bor(\R^3_+;Y)$, we can look at $F$ in the local coordinates with respect to $T$.
The way we do this is modelled on the behaviour of embedded functions under change of coordinates.
Given $T\in \TT$ and $f \in \Sch(\R;X)$, notice that
\begin{equation}\label{eq:embedded-pullback}
  \begin{aligned}
    &(\Emb[f][\phi] \circ \pi_{T})(\theta,\zeta,\sigma) \\
    &\qquad = e^{2\pi i \xi_{T}(x_{T}+\sigma\zeta)} \int_{\R}\big( \Mod_{-\xi_{T}}f \big)(x_{T}+s_{T} z) \, \bar{\Tr_{\zeta} \Dil_{\sigma} \Mod_{\theta}\phi(z)} \, \dd z
  \end{aligned}
\end{equation}
for all $\phi \in \Sch(\R)$.
With this relation in mind, we make the following definition.

\begin{defn}\label{def:CL-pullback}
  Let $Y$ be a topological vector space and $F \in \Bor(\R^3_+; Y)$.
  For each $T \in \TT_{\Theta}$ define the function $\pi_T^* F \in \Bor(\R^3_{+}; Y)$ by
  \begin{equation}
    \label{eq:CL-pullback}
    (\pi_{T}^{*}F)\,(\theta,\zeta,\sigma) := 
    \1_{\bar{\mT_{\Theta}}}(\theta,\zeta,\sigma) \;e^{-2\pi i \xi_{T}(x_{T}+s_{T}\zeta )} F\circ\pi_{T}(\theta,\zeta,\sigma)
  \end{equation}
  where $\bar{\mT_{\Theta}}$ denotes the closure of $\mT_{\Theta}$.
\end{defn}

Now consider $f \in \Sch(\R;X)$ as before.
It follows from \eqref{eq:embedded-pullback} that 
\begin{equation}\label{eq:embedded-DEs}
  \begin{aligned}
    \sigma \partial_{\zeta}\bigl(  \pi_{T}^{*}\Emb[f]\bigr) [\phi](\theta,\zeta,\sigma) &=  \bigl( \pi_{T}^{*}\Emb[f] \bigr) [(2\pi i \theta -\partial_z) \phi(z)](\theta,\zeta,\sigma)
    \\
    \sigma \partial_{\sigma}\bigl(  \pi_{T}^{*}\Emb[f]\bigr) [\phi](\theta,\zeta,\sigma)&=  \bigl( \pi_{T}^{*}\Emb[f] \bigr) [(2\pi i \theta -\partial_z) (z\phi(z))](\theta,\zeta,\sigma)
    \\
    \partial_{\theta}\bigl(  \pi_{T}^{*}\Emb[f]\bigr) [\phi](\theta,\zeta,\sigma)&=  \bigl( \pi_{T}^{*}\Emb[f] \bigr) [ 2\pi i z  \phi(z)](\theta,\zeta,\sigma)
  \end{aligned}
\end{equation}
for all $\phi \in \Phi$, where $z$ is a dummy variable.
Thus differentiation of embedded functions corresponds to changing the wave packet.
The identities \eqref{eq:embedded-DEs} need not hold for general functions in $\Bor(\R^{3}_+; \Lin(\Phi; X))$, so we use the right hand sides as a new definition.

\begin{defn}\label{def:wp-operators} 
  The \emph{wave packet differentials} are the operators
  \begin{equation*}
    \wpD_{\zeta}, \wpD_{\sigma}, \wpD_{\theta}\colon \Bor(\R^{3}_+; \Lin(\Phi; X)) \to \Bor(\R^{3}_+; \Lin(\Phi; X))
  \end{equation*}
  defined by
  \begin{equation}
    \label{eq:wp-operators}
    \begin{aligned}
      \sigma\wpD_{\zeta}F[ \phi ](\theta,\zeta,\sigma) &:= F\big[ (2\pi i\theta-\partial_{z}) \phi(z) \big](\theta,\zeta,\sigma)
      \\
      \sigma \wpD_{\sigma}F[ \phi ](\theta,\zeta,\sigma)& := F\big[ (2\pi i \theta-\partial_{z})(z\phi(z)) \big](\theta,\zeta,\sigma)
      \\
      \wpD_{\theta}F[ \phi ](\theta,\zeta,\sigma)&:=  F \big[ 2\pi i z  \phi(z)\big](\theta,\zeta,\sigma)
    \end{aligned}
  \end{equation}
  for all $F \in \Bor(\R^{3}_+; \Lin(\Phi; X))$ and $\phi \in \Phi$.
\end{defn}

Thus for $f \in \Sch(\R;X)$ we can write the equations \eqref{eq:embedded-DEs} as
\begin{equation}\label{eq:wpDEs}
  (  \partial_{\zeta}-\wpD_{\zeta}) (\pi_T^* \Emb[f]) =
  (  \partial_{\sigma}-\wpD_{\sigma}) (\pi_T^* \Emb[f]) =  (  \partial_{\theta}-\wpD_{\theta}) (\pi_T^* \Emb[f])=0,
\end{equation}
so that the \emph{defect operators}
\begin{equation}\label{eq:defect-definition}
  \DEF_{\zeta}:=(\partial_{\zeta}-\wpD_{\zeta})\quad  \DEF_{\sigma}:=(\partial_{\sigma}-\wpD_{\sigma})\quad   \DEF_{\theta}:=(\partial_{\theta}-\wpD_{\theta})
\end{equation}
quantify how much $\pi_{T}^{*}F$ differs from the pullback on a tree of an embedded function.
In general, the partial derivatives used in the definition of the defect operators are to be interpreted in the distributional sense.

\begin{rmk}\label{rmk:wpd-lacunary-size}
  For any wave packet $\phi\in\Phi$ and any $\theta \in \R$, the modified wave packets appearing in the definition of the wave packet differentials $\wpD_{\zeta}$ and $\wpD_{\sigma}$ are both in $\Psi(\theta)$, with 
  \begin{equation*}
    \|(2\pi i \theta -\partial_{z})\phi(z)\|_{N}\lesssim \|\phi\|_{N}
    \quad \text{and} \quad
    \|(2\pi i \theta -\partial_{z})z\phi(z)\|_{N-1}\lesssim \|\phi\|_{N}
  \end{equation*}
  for all $N \geq 1$.
  This innocuous observation will turn out to be quite important.
\end{rmk}

The definitions above (among other considerations) lead us to quantities of the form $F[\phi_{\theta,\zeta,\sigma}](\theta,\zeta,\sigma)$, in which the wave packet being tested against may vary over $\R^3_+$.
Here we present a useful lemma allowing for an arbitrary wave packet in $\Phi_1$ to be represented as a superposition of a fixed sequence of wave packets.
In applications this essentially lets us assume that the wave packet $\phi = \phi_{\theta,\zeta,\sigma}$ does not depend on $(\theta,\zeta,\sigma)$.

\begin{lem}\label{lem:uniformly-bounded-wave-packets}
  For all $N\in\N$ there exists a sequence of wave packets $\upsilon_{k}\in\Phi^{N}_{1}$ with the following property:
  every $\phi\in\Phi^{N'}_{1}$ with $N'\geq2N+1$ can be represented as an infinite linear combination
  \begin{equation*}
    \phi(z)= \sum_{k\in\Z} \mf{a}_{\phi,N}(k) \upsilon_{k}(z)
  \end{equation*}
  with coefficients $(\mf{a}_{\phi,N}(k))_{k \in \Z}$ satisfying
  \begin{equation*}
    |\mf{a}_{\phi,N}(k)|\lesssim_{N,N'} \langle k \rangle^{-N'+2N+1}.
  \end{equation*}
  
\end{lem}

\begin{proof}
  First suppose there exists a sequence of functions $(\upsilon_{k})_{k \in \Z}$ in $C^\infty(\R)$ such that
  \begin{equation}\label{eq:workhose-weak-assns}
    \spt\FT{\upsilon_{k}}\subset [-\mf{b},\mf{b}], \quad  \FT{\upsilon_{k}}\in C^{N}(\R), \quad       \|\upsilon_{k}\|_{N}<1
  \end{equation}
  satisfying the conclusion of the Lemma (note that $\FT{\upsilon}$ need not be smooth).
  We deduce the full claim by an approximation argument.
  For each $k \in \Z$ there exists a sequence $(\upsilon_{k,\kappa})_{\kappa \in \N}$ with $\lim_{\kappa \to \infty} \upsilon_{k,\kappa} = \upsilon_{k}$ in $\|\cdot\|_{N}$ and $\|\upsilon_{k,\kappa+1}-\upsilon_{k,\kappa}\|_{N}< 2^{-\kappa}$ for all $\kappa \in \N$.
  We then have
  \begin{equation*}
    \phi
    = \sum_{k\in\Z} \mf{a}_{\phi,N}(k) \upsilon_{k}
    = \sum_{k\in\Z} \sum_{\kappa=-1}^{\infty} 2^{-\kappa}\mf{a}_{\phi,N}(k) \tilde{\upsilon}_{k,\kappa}
  \end{equation*}
  with
  \begin{equation*}
    \tilde{\upsilon}_{k,\kappa} := \begin{cases}
      2^{\kappa}\upsilon_{k,0} &\text{if } \kappa=-1 \\
      2^{\kappa}(\upsilon_{k,\kappa+1} - \upsilon_{k,\kappa}) &\text{if } \kappa\geq 0.
    \end{cases}
  \end{equation*}
  The claim follows by reindexing the summation.

  It remains to prove the weakened claim with $(\upsilon_k)_{k \in \Z}$ satisfying \eqref{eq:workhose-weak-assns}.
  Let $\FT{\upsilon}\in C^{N}(\R)$ be even, smooth on $B_{2\mf{b}/3}$, with $0 < \FT{\upsilon}(\FT{z}) \leq 1$, and
  \begin{equation*}
    \FT{\upsilon}(\FT{z}) =
    \begin{cases}
      1 &  \mf{b}\leq \FT{z}<\mf{b}/3 \\
      |\FT{z}-\mf{b}|^{N+1} &  \mf{b}/2\leq\FT{z}<\mf{b} \\
      0 & \FT{z}\geq\mf{b}.
    \end{cases}
  \end{equation*}
  
  For all $\phi\in\Phi^{N'}_{1}$ we have $\tilde{\phi}:=(\FT{\phi}/\FT{\upsilon})^{\vee}\in\Phi$ (where $\vee$ denotes the inverse Fourier transform) with $\|\tilde{\phi}\|_{N'-N-1}\lesssim \|\phi\|_{N}$.
  Fourier inversion yields
  \begin{equation*}
    \begin{aligned}
      \FT{\phi}(\xi)
      =\FT{\tilde{\phi}}(\xi) \FT{\upsilon}(\xi)
      &=\frac{1}{2\mf{b}} \sum_{k\in\Z}\Big(  \int_{B_{\mf{b}}} \FT{\tilde{\phi}}(\FT{z}) e^{- 2\pi i \frac{k}{2\mf{b}} \FT{z}} \dd \FT{z}\Big) \FT{\upsilon}(\xi) e^{2 \pi i \frac{k}{2\mf{b}}  \xi  } \\
      &= \frac{1}{2\mf{b}} \sum_{k\in\Z} \tilde{\phi} \Big(\frac{k}{2\mf{b}}\Big) \Mod_{k/2\mf{b}} \FT{\upsilon}(\xi),
    \end{aligned}
  \end{equation*}
  so that
  \begin{equation*}
    \phi(x) = \frac{1}{2\mf{b}} \sum_{k\in\Z}\tilde{\phi}\Big(\frac{k}{2\mf{b}}\Big) \Tr_{k/2\mf{b}} \upsilon(x).
  \end{equation*}
  Since
  \begin{equation*}
    \big\|\Tr_{k/2\mf{b}} \upsilon \big\|_{N} < C_{\upsilon,N} \Big\langle\frac{k}{2\mf{b}}\Big\rangle^{N}
    \quad \text{and} \quad
    \Big|\tilde{\phi}\Big(\frac{k}{2\mf{b}}\Big)\Big| \lesssim \Big\langle\frac{k}{2\mf{b}}\Big\rangle^{-N'+N+1} \|\phi\|_{N},
  \end{equation*}
  we can set
  \begin{equation*}
    \upsilon_{k,N}:=C_{\upsilon,N}^{-1} \Big\langle \frac{k}{2\mf{b}} \Big\rangle^{-N}\Tr_{k/2\mf{b}} \upsilon
    \quad \text{and} \quad
    \mf{a}_{\phi,N}(k):= C_{\upsilon, N}\Big\langle \frac{k}{2\mf{b}} \Big\rangle^{N} \tilde{\phi}\Big( \frac{k}{2\mf{b}} \Big),
  \end{equation*}
  completing the proof.
\end{proof}

\subsection{Local sizes on trees}\label{sec:sizes}

Given a Banach space $X$, we define various classes of $X$-valued and $\Lin(\Phi;X)$-valued local sizes on $\TT_{\Theta}$.
The first class is the same as that used in scalar-valued time-frequency analysis.

\begin{defn}[Lebesgue local sizes]\label{defn:lebesgue-local-sizes}
  For $s \in [1,\infty]$ and $N \in \N$ we define the local sizes $\lL^{s}_{\Theta,N}$ as follows: for $ F \in \Bor(\R^{3}_{+};\Lin(\Phi;X))$ and $T \in \TT_{\Theta}$,
  \begin{equation}
    \label{eq:lL-size}
    \| F \|_{\lL^{s}_{\Theta,N}(T)}
    =   \Big( \int_{\R^{3}_{+}} \sup_{\phi \in \Phi_{1}^{N}} \big\| \pi_{T}^{*}F[\phi] (\theta,\zeta,\sigma) \big\|_{X}^{s} \, \dd \theta \, \dd \zeta \, \frac{\dd \sigma}{\sigma}\Big)^{\frac{1}{s}}
  \end{equation}
  with the usual modification when $s=\infty$.
  We will drop $N\in\N$ from the notation unless it is relevant.
  For $G \in \Bor(\R^{3}_{+};X)$ we abuse notation and write
  \begin{equation*}
    \| G \|_{\lL^{s}_{\Theta}(T)}
    :=   \|  \pi^{*}_{T}G \|_{L_{\dd \theta \dd \zeta \frac{\dd \sigma}{\sigma}}^{s}(\R^{3}_{+};X)};
  \end{equation*}
  that is, $\| G \|_{\lL^{s}_{\Theta}(T)}$ is defined as in \eqref{eq:lL-size}, where we drop the supremum over wave packets from the definition.
  These local sizes have `inner' and `outer' variants 
  \begin{equation*}
    \| F \|_{\lL^{s}_{\Theta,in}(T)} :=   \|\1_{T^{in}} F \|_{\lL^{s}_{\Theta}(T)}  \qquad   \| F \|_{\lL^{s}_{\Theta,out}(T)} :=   \|\1_{T^{out}} F \|_{\lL^{s}_{\Theta}(T)}.
  \end{equation*}
\end{defn}

The scalar-valued Lebesgue local sizes satisfy the following local size-Hölder inequality, which has a straightforward proof.

\begin{prop}\label{prop:lebesgue-size-holder}
  Suppose $(s_{i})_{i=1}^3$ and $(\tilde{s}_{i})_{i=1}^3$ are two Hölder triples of exponents, with $s_i,\tilde{s}_i \in (0,\infty]$.
  Then for any $T \in \TT_\Theta$ and $F_1,F_2,F_3 \in \Bor(\R^3_+;\C)$,
  \begin{equation*}
    \|F_1 F_2 F_3\|_{\lL_\Theta^1(T)} \lesssim \prod_{i=1}^3 \|F_i\|_{(\lL_{\Theta,out}^{s_{i}} + \lL^{\tilde{s}_{i}}_{\Theta,in})(T)}.
  \end{equation*}
\end{prop}

The next local sizes use the $\gamma$-norm defined in Section \ref{sec:gamma-norms}.\footnote{We used a discrete version with Rademacher sums in place of $\gamma$-norms in \cite{AU20-Walsh}; this discrete version was also used in \cite{DPO18}.}

\begin{defn}[`outer' $\gamma$ local size]
  For $s \in [1,\infty]$ and $N \in \N$ we define the local sizes  $\RS_{\Theta,N,out}^{s}$ as follows: for $F \in \Bor(\R^3_+; \Lin(\Phi;X))$ and $T \in \TT_{\Theta}$,
  \begin{equation}\label{eq:R-lac-def}
    \|F\|_{\RS_{\Theta,N,out}^{s}(T)}
    :=   \Big( \int_{\R^{2}} \sup_{\varphi \in \Phi^{N}_{1}} \|\pi_T^*(\1_{T^{out}} F)[\varphi] (\theta,\zeta,\sigma) \|_{\gamma_{\dd \sigma / \sigma}(\R_{+}; X)}^{s} \, \dd \zeta \dd \theta  \Big)^{\frac{1}{s}}.
  \end{equation}
  As with the Lebesgue local sizes, we drop $N$ from the notation whenever possible.
  For $G \in \Bor(\R_{+}^{3}; X)$ we abuse notation and define $\|G\|_{\RS_{\Theta,N,out}^{s}(T)}$ as in \eqref{eq:R-lac-def}, but without the supremum over $\Phi^{N}_{1}$:
  \begin{equation*}
    \|F\|_{\RS_{\Theta,N,out}^{s}(T)}
    :=   \Big( \int_{\R^{2}} \|\pi_T^*(\1_{T^{out}} F) (\theta,\zeta,\sigma) \|_{\gamma_{\dd \sigma / \sigma}(\R_{+}; X)}^{s} \, \dd \zeta \dd \theta  \Big)^{\frac{1}{s}}.    
  \end{equation*}
  
\end{defn}

When $X$ is isomorphic to a Hilbert space we have that $\RS_{\Theta,out}^{2}$ and $\lL^2_{\Theta,out}$ are equivalent, by Proposition \ref{prop:gamma-l2-hilb}.
In general, unless $X$ has type $2$ or cotype $2$, there is no comparison between these two local sizes.

\begin{rmk}
  To illustrate the definition of the local size $\RS_{out}^{s}$, consider $F = \Emb[f]$ for some $f \in \Sch(\R;X)$ and $T=T_{(0,0,1),\Theta}$. For $(\theta,\zeta,\sigma)\in\mT_{\Theta}$ we have
  \begin{equation*}
    \pi_{T^{*}}\Emb[f][\phi] (\theta,\zeta,\sigma)= f * \Dil_{\sigma}\tilde{\phi_{\theta}}(\zeta) \qquad \tilde{\phi_{\theta}}:=\bar{\Mod_{\theta}\phi}\circ(-\Id).
  \end{equation*}
  Since $\1_{T^{out}}F$ is measured in the definition of $\|F\|_{\RS_{out}^s}$, the only $\theta$ which contribute satisfy $|\theta| > 2\mf{b}$.
  The wave packets $\phi$ that are used all satisfy $\spt \FT{\phi}\subset B_{\mf{b}}$, so $\tilde{\phi_{\theta}}$ has Fourier support vanishing in $B_{\mf{b}}$ and in particular $\int_{\R}\tilde{\phi_{\theta}}(z)\dd z=0$.
  Thus $\|F\|_{\RS_{out}^{s}}$ represents a localised $L^s$-norm of a Littlewood--Paley-type square function. When $X$ is a UMD space, these can be controlled via Theorem \ref{thm:WP-LP}.
\end{rmk}

We also introduce local sizes which represent square functions on the `inner' frequencies $\theta\in B_{2\mf{b}}$.
These will not be taken with respect to all wave packets $\phi \in \Phi$, but only those in $\Psi(\theta)$, which have Fourier transform vanishing at $-\theta$.

\begin{defn}[`full' $\gamma$ local size]\label{def:full-gamma-size}
  For $N \in \N$ and $s \in [1,\infty)$, we define the local size $\RS_{\Theta,N}^{s}$ for $F \in \Bor(\R^3_+; \Lin(\Phi;X))$ and $T \in \TT_{\Theta}$ by
  \begin{equation}\label{eq:R-full-def}
    \|F\|_{\RS_{\Theta,N}^{s}(T)}
    :=  \Big( \int_{\R^{2}} \sup_{\psi \in \Psi^{N}_{1}(\theta)}  \|\pi_T^*F[\psi] (\theta,\zeta,\sigma) \|_{\gamma_{\dd \sigma / \sigma}(\R_{+}; X)}^{s} \, \dd \zeta \dd \theta  \Big)^{\frac{1}{s}}.
  \end{equation}
  As above, reference to $N$ will be omitted whenever reasonable.
\end{defn}

The final classes of local sizes measure how far a function $\map{F}{\R^3_+}{\Lin(\Phi;X)}$ differs from an embedded function $\Emb[f]$.
These local sizes are also exploited in the scalar-valued theory in \cite{UW19}; a discrete version was used in \cite{AU20-Walsh}.

\begin{defn}[Defect local sizes]
  For $N \in \N$ we define the the \emph{$\sigma$-defect} and \emph{$\zeta$-defect} local sizes $\DS^{\sigma}_{\Theta.N}$ and $\DS^{\zeta}_{\Theta,N}$ as follows: for $F \in \Bor(\R^3_+; \Lin(\Phi;X))$ and $T \in \TT_{\Theta}$,
  \begin{equation}\label{eq:defect-size}
    \begin{aligned}
      \|F\|_{\DS _{\Theta,N}^{\sigma}(T)}
      &
      := \sup_{ \phi\in \Phi^{N}_{1}} \sup_{g}\Big| \int_{\R^{3}_{+}}\Big\langle (\sigma\DEF_{\sigma} \pi_T^* F)[\varphi](\theta,\zeta,\sigma); g(\theta,\zeta,\sigma) \Big\rangle \, \frac{\dd \sigma}{\sigma} \, \dd \zeta \, \dd \theta\Big|
      \\
      \|F\|_{\DS_{\Theta,N}^{\zeta}(T)}
      &
      := \sup_{\phi\in\Phi^{N}_{1}} \sup_{g}\Big| \int_{\R^{3}_{+}} \Big\langle (\sigma\DEF_{\zeta} \pi_T^* F)[\varphi](\theta,\zeta,\sigma); g(\theta,\zeta,\sigma) \Big\rangle \, \frac{\dd \sigma}{\sigma} \, \dd \zeta \,\dd \theta\Big|
    \end{aligned}
  \end{equation}
  with inner supremum taken over all $g \in C_c^\infty(\R^{3}_{+} ; X^{*})$ satisfying
  \begin{equation*}
    \int_{\R^{2}}  \sup_{\sigma\in \R_{+}} \|g(\theta,\zeta,\sigma)\|_{X^*} \, \dd \zeta \,\dd \theta\leq 1.
  \end{equation*}
  Thus the defect sizes behave like $L^1$ in scale and $L^\infty$ in time and frequency.
  The defect operators $\DEF_{\sigma}$ and $\DEF_{\zeta}$ involve distributional derivatives, and the integral is an abuse of notation for the pairing of $X$-valued distributions with $X^*$-valued test functions.
  For $F \in \Bor(\R^3_+; \Lin(\Phi;X))$ which is not locally integrable, we define $\|F\|_{\DS _{\Theta}^{\sigma}(T)} = \|F\|_{\DS _{\Theta}^{\zeta}(T)} = \infty$.
\end{defn}

The defect local sizes are not actually local sizes: they fail global positive-definiteness, as they vanish on every embedded function $\Emb[f]$.
However, the `complete' local sizes defined below are actual local sizes.

\begin{defn}[Complete local size]
  Let $X$ be a Banach space.
  For $N \in \N$, $s \in [1,\infty)$, $F \in \Bor(\R^{3}_{+} ; \Lin(\Phi;X))$, and $T \in \TT_{\Theta}$, we define 
  \begin{equation*}
    \| F \|_{\FS_{\Theta,N}^{s}(T)}:= \| F \|_{\lL^{\infty}_{\Theta,N}(T)} + \| F \|_{\RS_{\Theta,N}^{s}(T)} + \| F \|_{\DS^{\sigma}_{\Theta,N}(T)}+\| F \|_{\DS^{\zeta}_{\Theta,N}(T)}.
  \end{equation*}
\end{defn}

\begin{rmk}\label{rmk:invariances}
  The outer space structures introduced on $\R^{3}_{+}$ are invariant under translation, modulation, and dilation symmetries, in the sense that
  \begin{equation}\label{eq:space-symmetry}
    \begin{aligned}
      \mu (E) = t_0 \mu\big( \{(t_0\eta+\eta_{0}, t_0 y+y_{0},t_0 t) \colon (\eta,y,t)\in E\} \big)
    \end{aligned}
  \end{equation}
  for any $(\eta_{0},y_{0},t_0)\in\R^{3}_{+}$ and $E\subset \R^{3}_{+}$, and similarly for $\nu$. Furthermore it holds that
  \begin{equation}\label{eq:embedding-symmetry}
    t_0 ^{-1}\Emb[f](  (\eta-\eta_{0})t_0,(y-y_{0})/t_0,t/t_0 )= e^{-2\pi i \eta_{0}(y -y_{0})}\Emb[\Lambda_{(\eta_0, y_0, t_0)} f ](\eta, y,t).
  \end{equation}
  The local sizes defined above possess analogous invariance properties.
\end{rmk}

To end this section we note that for $F \in \Bor(\R^3_+; \Lin(\Phi;X))$, the local sizes $\lL^{s}$, $\RS^{s}_{out}$, and $\RS^{s}$ (all of which involve testing against wave packets which generally vary over $\R^3_+$) can be controlled by the values of the corresponding sizes of $F[\phi]\in\Bor(\R^{3}_{+};X)$ with $\phi \in \Phi_1$ constant over $\R^3_+$.
This is a direct corollary of Lemma \ref{lem:uniformly-bounded-wave-packets}.

\begin{prop}\label{prop:sup-wavepackets}
  For all $N'>2N+1$,  $s \in (0,\infty)$, and $F\in\Bor(\R^{3}_{+};\Lin( \Phi;X ))$,
  \begin{equation*}
    \begin{aligned}[t]
      &\|F\|_{\lL^{s}_{\Theta,N'}(T)} \lesssim \sup_{\phi\in\Phi_{1}^{N}}\|F[\phi]\|_{\lL^{s}_{\Theta,N,out}(T)}
      \\
      &\|F\|_{\RS^{s}_{\Theta,N',out}(T)} \lesssim \sup_{\phi\in\Phi_{1}^{N}}\|F[\phi]\|_{\RS^{s}_{\Theta,N,out}(T)}.
    \end{aligned}
  \end{equation*}  
\end{prop}


%% file: main/single-tree.tex
In this section we will prove a local size-H\"older inequality, which will ultimately reduce bounds on $\BHF_{\Pi}$ to outer Lebesgue-valued bounds for the embedding map $\Emb$.
Such an estimate is more familiarly known as a \emph{single-tree estimate}, being an estimate for a wave packet form localised to a single tree.

\subsection{Statement of the result and consequences for \texorpdfstring{$\BHWF_{\Pi}$}{BHWF\_Pi}}\label{sec:size-holder}

Fix $\mf{b} = 2^{-4}$ and $\phi_{0} \in \Phi$ with Fourier support in $B_{\mf{b}/2}$.\footnote{Any sufficiently small $\mf{b}$ will work here.}
For Banach spaces $X_{1}$, $X_{2}$, $X_{3}$ and a bounded trilinear form $\map{\Pi}{X_{1} \times X_{2} \times X_{3}}{\CC}$, we can write the $\BHF$-type wave packet form associated with $\Pi$ on functions $\map{F_{i}}{\R^{3}_{+}}{\Lin(\Phi,X_{i})}$ as follows: setting  $\alpha = (1,1,-2)$ and $\beta = (-1,1,0)$ for convenience,
\begin{equation*}
  \BHWF_{\Pi}^{\phi_{0},\mbb{K}}(F_{1},F_{2}, F_{3})  = \int_{\mbb{K}} \Pi_{i = 1}^{3} \big( F_{i}[\phi_{0}](\alpha_{i} \eta - \beta_{i} t^{-1}, y, t) \big) \, \dd \eta \, \dd y \, \dd t,
\end{equation*}
where $\mbb{K}\subset \R^{3}_{+}$ is any compact subset of $\R^{3}_{+}$ and  $\Pi_{i=1}^{3} (x_{i}) := \Pi(x_{1},x_{2},x_{3})$. Let $\mb{\Pi}$
be the trilinear operator
\begin{equation*} \begin{aligned}[t]
    &\map{\mb{\Pi}}{\prod_{i=1}^{3} \Bor(\R^{3}_{+}; \Lin(\Phi,X_{i}))}{\Bor(\R^{3}_{+}; \C)}
    \\ & \big(\mb{\Pi}_{i=1}^{3}(F_{i}) \big) (\eta,y,t) := \Pi_{i=1}^{3} \big( F_{i}[\phi_{0}](\alpha_{i} \eta - \beta_{i} t^{-1}, y, t) \big),
  \end{aligned} \end{equation*}
so that we can write
\begin{equation*}
  \BHWF_{\Pi}^{\phi_{0},\mbb{K}}(F_{1},F_{2}, F_{3})  = \int_{\mbb{K}} \big( \mb{\Pi}_{i=1}^{3}(F_{i}) \big)(\eta,y,t) \, \dd \eta \, \dd y \, \dd t.
\end{equation*}

Fix the frequency band $\Theta = B_{2}$, and for $i \in \{1,2,3\}$ define the translated frequency bands
\begin{equation*} \begin{aligned}[t]
    & \Theta_{i} := \alpha_{i} \Theta + \beta_{i}
    \\ & \Theta_{i}^{*}=\big\{ \theta\in \Theta \colon \alpha_{i}\theta+\beta_{i} \in B_{3\mf{b}} \big\},    
  \end{aligned} \end{equation*}
so that $B_{2\mf{b}} \subsetneq B_{1}\subsetneq \Theta_{i}$ for all $i$ and $\Theta_{i}^{*}$ are pairwise disjoint.

\begin{thm}\label{thm:size-holder}
  Fix Banach spaces $X_{1}$, $X_{2}$, $X_{3}$ with finite cotype and a bounded trilinear form $\map{\Pi}{X_{1} \times X_{2} \times X_{3}}{\CC}$.
  Let $(s_{i})_{i=1}^{3}$ be a H\"older triple of exponents in $(1,\infty)$.
  Then for any $T\in\TT_{\Theta}$, any $A\in\TT^{\cup}_{\Theta}$, and $F_{i}\in \Bor\big( \R^{3}_{+}; \Lin(\Phi;X_{i}) \big)$, 
  \begin{equation} \label{eq:S1-holder}
    \big\| \1_{\R^{3}_{+}\setminus A}  \mb{\Pi}(F_{1},F_{2},F_{3}) \big\|_{(\lI_{\Theta,\mathbb{K}}+\lL^{\infty})(T)}
    \lesssim \prod_{i=1}^{3} \| F_{i} \|_{\FS_{\Theta_{i}}^{s_{i}}},
  \end{equation}
  for all compact $\mathbb{K} = V_{+} \sm V_{-}$  with $V_{\pm}\in\TT_{\Theta}^{\cup}$, where
  \begin{equation} \label{eq:non-abs-size}
    \| G\|_{\lI_{\Theta,\mathbb{K}}(T)}:=\Big|\int_{\mT_{\Theta}} \big((\1_{\mbb{K}}\,G)\circ \pi_{T}\big)(\theta,\zeta,\sigma) \frac{\dd\sigma }{\sigma} \dd \zeta \dd \theta \Big|
  \end{equation}
  so that $\lI_{\Theta,\mathbb{K}}+\lL^{\infty}$ is a local size.
  Furthermore it holds that
  \begin{equation} \label{eq:holder-spt}
    \mu\big(\spt \mb{\Pi}\big(F_{1},F_{2},F_{3} \big)\big)  \leq \min_{i\in\{1,2,3\}}\mu_{i}(\spt F_{i}).
  \end{equation}
\end{thm}

This has the following consequence for $\BHWF_{\Pi}$, whose proof is a straightforward combination of Theorem \ref{thm:size-holder} with Proposition \ref{prop:outer-RN}.\footnote{See \cite[Corollary 4.13]{AU20-Walsh} for the argument for the Walsh model.}

\begin{cor}\label{cor:BHT-RN-reduction}
  Fix Banach spaces $X_{1},X_{2},X_{3}$ with finite cotype and a bounded trilinear form $\map{\Pi}{X_{1} \times X_{2} \times X_{3}}{\CC}$.
  Let $(p_{i})_{i=1}^{3}$, $(q_{i})_{i=1}^{3}$, and $(s_{i})_{i=1}^{3}$ be Hölder triples of exponents, with $p_{i},q_{i} \in (0,\infty]$ and $s_{i} \in (1,\infty)$.
  Then for all $F_{i} \in \Bor(\R^{3}_{+} ; \Lin(\Phi;X_{i}))$ and all compact $\mbb{K}=V_{+}\setminus V_{-}$ with $V_{\pm}\in\TT^{\cup}_{\Theta}$,
  \begin{equation}\label{eqn:BHWF-estimate}
    |\BHWF_{\Pi}^{\phi_{0}, \mathbb{K}} (F_{1},F_{2} ,F_{3})| \lesssim \prod_{i=1}^{3} \|F_{i}\|_{L_{\nu}^{p_{i}} \sL_{\mu_{i}}^{q_{i}} \FS_{\Theta_{i}}^{s_{i}}},
  \end{equation}
  where $\mu_{i} := \mu_{\Theta_{i}}$ and the implicit constant is independent of $\mbb{K}$.
\end{cor}

\begin{rmk}
  Note that we prove \eqref{eqn:BHWF-estimate} rather than the stronger estimate
  \begin{equation}\label{eq:BHWF-unc}
    \int_{\mbb{K}} \big| (\mb{\Pi}_{i=1}^{3} (F_{i}))(\eta,y,t) \big| \, \dd\eta \, \dd y \, \dd t \lesssim \prod_{i=1}^{3} \|F_{i}\|_{L_{\nu}^{p_{i}} \sL_{\mu_{i}}^{q_{i}} \FS_{\Theta_{i}}^{s_{i}}}
  \end{equation}
  with absolute value inside the integral.
  Such an estimate would imply
  \begin{equation}\label{eqn:BHWF-mult}
    \int_{\mbb{K}} a(\eta,y,t) (\mb{\Pi}_{i=1}^{3} (F_{i}))(\eta,y,t)  \, \dd\eta \, \dd y \, \dd t \lesssim \|a\|_{L^{\infty}(\R^{3}_{+})} \prod_{i=1}^{3} \|F_{i}\|_{L_{\nu}^{p_{i}} \sL_{\mu_{i}}^{q_{i}} \FS_{\Theta_{i}}^{s_{i}}}
  \end{equation}
  for all $a \in L^{\infty}(\R^{3}_{+})$.
  In \cite[Proposition 4.12]{AU20-Walsh} we prove a discrete analogue of \eqref{eqn:BHWF-mult}, which models the situation where the multiplier $a$ satisfies $|t \partial_{t}a(\eta,y,t)| \lesssim 1$.
  It should be possible to prove \eqref{eqn:BHWF-mult} with regularity assumptions on $a$, leading to multilinear multiplier theorems along the lines of those proven by Muscalu, Tao, and Thiele \cite{MTT02} (as proven in \cite{DPLMV19-3}).
  It should also be possible to handle more than three factors.
  These extensions are beyond the scope of this article.
\end{rmk}

\subsection{Proof of Theorem \ref{thm:size-holder}}

\subsubsection{Preliminary setup}

Boundedness of $\Pi$ immediately implies
\begin{equation*}
  \big\| \1_{\R^{3}_{+} \setminus A}  \mb{\Pi}\big(F_{1}[\phi_{0}],F_{2}[\phi_{0}],F_{3}[\phi_{0}]\big) \big\|_{\lL^{\infty}}
  \lesssim \prod_{i=1}^{3} \| F_{i} \|_{\lL^{\infty}}  \leq  \prod_{i=1}^{3} \| F_{i} \|_{\FS_{\Theta_{i}}^{s_{i}}},
\end{equation*}
so we need only control $\lI_{\Theta}$.  By definition we have 
\begin{equation*}
  \big(\1_{\mbb{K}\setminus A} \, \mb{\Pi}_{i=1}^{3}( F_{i})\big)\circ \pi_{T}(\theta,\zeta,\sigma)
  = \big(\mb{\Pi}_{i=1}^{3}(\1_{\mbb{K}_{i}} F_{i})\big) \circ \pi_{T}(\theta,\zeta,\sigma),
\end{equation*}
where  $\mbb{K}_{i}=V_{+,i}\setminus\big(V_{-,i}\cup A_{i}\big)$ with 
\begin{equation*} \begin{aligned}[t]
    & A_{i}:= \{(\alpha_{i} \eta - \beta_{i} t^{-1}, y, t) : (\eta,y,t)\in A \}\in\TT_{\Theta_{i}}^{\cup}
    \\ & V_{\pm,i}:= \{(\alpha_{i} \eta - \beta_{i} t^{-1}, y, t) : (\eta,y,t)\in  V_{\pm} \}\in\TT_{\Theta_{i}}^{\cup}.
  \end{aligned} \end{equation*}
The statement of \eqref{eq:holder-spt} follows trivially by noticing that $\mu(V_{\pm})=\mu_{i}(V_{\pm,i})$.

Multiplying by $1 = e^{2\pi i(\alpha_{1} + \alpha_{2} + \alpha_{3})\xi_{T}(x_{T} + s_{T}\zeta)}$, we have
\begin{equation*} \begin{aligned}[t]
    & \big\| \1_{\R^{3}_{+} \setminus A}  \mb{\Pi}\big(F_{1}[\phi_{0}],F_{2}[\phi_{0}],F_{3}[\phi_{0}]\big) \big\|_{\lI_{\Theta}(T)}
    \\ & = \Big|\int_{\mT_{\Theta}} \Pi_{i=1}^{3} \Big( e^{-2\pi i \alpha_{i} \xi_{T}(x_{T}+s_{T}\zeta)}\big(\1_{\mbb{K}_{i}} F_{i}[\phi_{0}] \big)\circ \pi_{T_{i}}(\alpha_{i} \theta + \beta_{i}, \zeta, \sigma) \Big) \frac{\dd\sigma }{\sigma} \, \dd \zeta \, \dd \theta\Big|
    \\ & = \Big| \int_{\mT_{\Theta}}\Pi_{i=1}^{3} \Big(\pi_{T_{i}}^{*} \big(\1_{\mbb{K}_{i}} F_{i}[\phi_{0}]\big)(\alpha_{i} \theta + \beta_{i}, \zeta, \sigma)  \Big)\frac{\dd\sigma }{\sigma} \, \dd \zeta \, \dd \theta\Big|
    \\ & \leq  \sum_{j=1}^{3} \int_{B_{3}}\Big|\int_{B_{2}\times(0,1)}  \1_{\Theta \sm (\cup_{j'\neq j }\Theta^{*}_{j'})}(\theta) \, \Pi_{i=1}^{3} \Big( \pi_{T_{i}}^{*} \big(\1_{\mbb{K}_{i}}F_{i}[\phi_{0}]\big)(\theta_{i}, \zeta, \sigma)  \Big)\, \frac{\dd \sigma}{\sigma} \, \dd \zeta \Big|\, \dd \theta ,
  \end{aligned} \end{equation*}
where we write $\theta_{i} := \alpha_{i} \theta + \beta_{i}$ (i.e. $\theta_{i}$ is implicitly a function of $\theta$) and $T_{i} = T_{(\alpha_{i} \xi_{T}, x_{T}, s_{T}),\Theta_{i}}$ to save space.
We bound each summand individually; here we describe only the case $j=1$, as the others are treated identically.

\subsubsection{Reduction to compactly supported smooth functions}

We may assume that the functions $F_{i}$ are compactly supported, as the expression above depends only on the values of $F_{i}$ on $T_{i}$ and
\begin{equation*}
  \prod_{i=1}^{3} \|\1_{\mbb{K}_{i}} F_{i} \|_{\FS_{\Theta_{i}}^{s_{i}}}\lesssim     \prod_{i=1}^{3} \| F_{i} \|_{\FS_{\Theta_{i}}^{s_{i}}}.
\end{equation*}   
Let
\begin{equation*} \begin{aligned}[t]
    & \mc{F}_{1}(\theta_{1},\zeta,\sigma):=\1_{\Theta \sm (\cup_{j'=2}^{3}\Theta_{j'}^{*})}(\theta) \,\pi_{T_{1}}^{*}\big(\1_{\mbb{K}_{i}} F_{1}  \big)(\theta_{1},\zeta,\sigma )  
    \\ &\mc{F}_{i}(\theta_{i},\zeta,\sigma):= \1_{\Theta \sm B_{3\mf{b}}}(\theta_{i}) \, \pi_{T_{1}}^{*}\big( \1_{\mbb{K}_{i}} F_{i}\big) (\theta_{i},\zeta,\sigma) \qquad i\in\{2,3\}
  \end{aligned} \end{equation*}
so that our claim reduces to showing
\begin{equation}\label{eq:sizeholder-claim} \begin{aligned}[t]
    & \int_{\R}\Big|\int_{\R^{2}_{+}} \Pi_{i=1}^{3} \big(\mc{F}_{i}[\phi_{0}](\theta_{i},\zeta,\sigma)\big)  \, \frac{\dd \sigma}{\sigma} \, \dd \zeta \Big|\, \dd \theta
    \lesssim  \prod_{i=1}^{3} \| F_{i} \|_{\FS_{\Theta_{i}}^{s_{i}}(T_{i})}.
  \end{aligned} \end{equation}
Let us argue that each $\mc{F}_{i}$ can be assumed to be smooth and supported on a small open neighborhood of $\mT$; the set  $B_{3}\times B_{2}\times (0,3)$ will do. Fix a non-negative bump function $\chi\in C^{\infty}_{c}(B_{1})$ with $\int \chi = 1$, and let
\begin{equation*}
  \mc{F}_{i}^{\epsilon} (\theta,\zeta,\sigma) := \int_{\R^{3}_{+}} \mc{F}_{i}(\theta-\theta',\zeta-\zeta',\sigma \sigma'^{-1})  \chi_{\epsilon}(\theta') \chi_{\epsilon}(\zeta') \chi_{\epsilon}(\log(\sigma')) \dd \theta' \dd \zeta' \frac{\dd \sigma'}{\sigma'}
\end{equation*}
where $\chi_{\varepsilon} = \Dil_{\varepsilon} \chi$.
Then the functions $\mc{F}_{i}^{\epsilon}[\phi]$ are smooth and compactly supported for any $\phi\in\Phi$,
and by dominated convergene we have
\begin{equation*}
  \begin{aligned}
    &\int_{B_{3}} \Big|\int_{B_{2}\times(0,1)}  \Pi_{i=1}^{3} \big(\mc{F}_{i}[\phi_{0}](\theta_{i},\zeta,\sigma)\big)  \, \frac{\dd \sigma}{\sigma} \, \dd \zeta \Big|\, \dd \theta \\
    &\qquad = \lim_{\epsilon\to 0}\int_{B_{3}} \Big|\int_{B_{2}\times(0,1)}   \Pi_{i=1}^{3} \big(\mc{F}_{i}^{\epsilon}[\phi_{0}](\theta_{i},\zeta,\sigma)\big)  \, \frac{\dd \sigma}{\sigma} \, \dd \zeta\Big| \, \dd \theta.
  \end{aligned}
\end{equation*}
For $\epsilon>0$ sufficiently small and $i \in \{2,3\}$, $\mc{F}_{i}^{\epsilon}(\theta,\zeta,\sigma)$ vanishes for $\theta\in B_{2\mf{b}}$.

Fix any $N\in\N$. By the $\gamma$-dominated convergence theorem (Proposition \ref{prop:gamma-dominated-convergence}), for $\epsilon > 0$ sufficiently small we have
\begin{equation*}
  \Big( \int_{\R^{2}} \sup_{\psi\in \Psi_{1}^{N}(\theta)}\big\|\mc{F}_{i}^{\epsilon}[\psi] (\theta,\zeta,\cdot)\big\|_{\gamma_{\dd \sigma / \sigma}(\R_{+}; X)}^{s_{i}} \, \dd \zeta \dd \theta  \Big)^{\frac{1}{s_{i}}} \lesssim \|F_{i}\|_{\FS_{\Theta_{i},N}^{s_{i}}} \qquad i\in\{2,3\}.
\end{equation*}
Clearly we also have
\begin{equation*}
  \sup_{\phi\in \Phi_{1}^{N}} \big\|\mc{F}_{i}^{\epsilon}[\phi] (\theta,\zeta,\sigma)\big\|_{L^{\infty}(\R^{3}_{+};X_{i})} \lesssim\|F_{i}\|_{\FS_{\Theta_{i},N}^{s_{i}}}  \qquad i\in\{1,2,3\},
\end{equation*}
and by the definition of the defect local sizes in
\eqref{eq:defect-size} it holds that
\begin{equation*} \begin{aligned}[t]
    \sup_{(\theta,\zeta)\in\R^{2}}    \sup_{\phi\in\Phi_{1}^{N}} \int_{\R_{+}}\big\|\sigma\DEF_{\sigma}\mc{F}_{i}^{\epsilon}[\phi](\theta,\zeta,\sigma)\big\|_{X_{i}} \frac{\dd \sigma}{\sigma} \lesssim \|\1_{\R^{3}_{+}\setminus A_{i}^{\epsilon}}F\|_{\DS_{\Theta,N}^{\sigma}(T_{i})}
  \end{aligned} \end{equation*}
for all $i\in\{1,2,3\}$, and likewise for the $\zeta$-defect term.
Putting all this together, we see that without loss of generality we may assume that each $\mc{F}_{i}$ is smooth.

\subsubsection{Splitting the wave packets}

We proceed towards the claimed bound \eqref{eq:sizeholder-claim}. Fix $\upsilon\in\Phi$ such that $\FT{\upsilon}(\FT{z})=1$ for $\FT{z}\in B_{2\mf{b}/3}$.
Then
\begin{equation}\label{eq:lac-zero-split}
  \begin{aligned}
    &
    \int_{B_{3}}\Big|\int_{B_{2}\times (0,1)} \Pi_{i=1}^{3} \big( \mc{F}_{i}[\phi_{0}] (\theta_{i}, \zeta, \sigma)  \big)  \, \frac{\dd \sigma}{\sigma} \, \dd \zeta \Big|\, \dd \theta.
    \\
    &
    =
    \int_{B_{3}}\Big|\int_{B_{2}\times (0,1)}
    \Pi \Big( \mc{F}_{1}[\phi_{1,l}] (\theta_{1}, \zeta, \sigma), \big(\mc{F}_{i}[\phi_{i}] (\theta_{i}, \zeta, \sigma) \big)_{i=2,3} \Big)  \, \frac{\dd \sigma}{\sigma} \, \dd \zeta \Big|\, \dd \theta 
    \\
    &
    +
    \int_{B_{3}}\Big|\int_{B_{2}\times (0,1)}
    \Pi \Big( \int_{0}^{\sigma} \rho \partial_{\rho} \mc{F}_{1}[\phi_{1,b}] (\theta_{1}, \zeta, \rho) \, \frac{\dd \rho}{\rho}, \big(\mc{F}_{i}[\phi_{i}] (\theta_{i}, \zeta, \sigma) \big)_{i=2,3} \Big)  \, \frac{\dd \sigma}{\sigma} \, \dd \zeta\Big| \, \dd \theta
  \end{aligned}
\end{equation}
where
\begin{equation*}
  \begin{aligned}[t]
    & \phi_{i}:=\phi_{0} \quad i\in\{1,2,3\}
    \\
    &\phi_{1,l}:=\phi_{0}-\FT{\phi_{0}}(-\theta_{1})\upsilon \qquad    \phi_{1,b}:=\FT{\phi_{0}}(-\theta_{1})\upsilon
  \end{aligned}
\end{equation*}
and the splitting $\phi_{1}=\phi_{1,l}+\phi_{1,b}$ implicitly depends on $\theta_{1}$ and thus on $\theta$.
For all $N\in\N$ we also have that $\|\phi_{1,l}\|_{N}+\|\phi_{1,b}\|_{N}\lesssim_{N} 1$, for any $\theta\in B_{3}$ we have $\phi_{1,l} \in \Psi(\theta_{1})$ i.e.  $\FT{\phi_{0}}(-\theta_{1})-\FT{\phi_{0}}(-\theta_{1})\FT{\upsilon}(-\theta_{1})=0$, and finally if $\theta_{1}\notin B_{\mf{b}/2}$ then $\phi_{1,b}=0$. To see this notice that if $\FT{\phi_{0}}(-\theta_{1})\neq 0$ then $-\theta_{1}\in B_{\mf{b}/2}$ and so $\upsilon(-\theta_{1})=1$, while if $\theta_{1}\notin B_{\mf{b}/2}$ then $\FT{\phi_{0}}(-\theta_{1})=0$.
This splitting represents the wave packet $\phi_{1}$ as the sum of a `lacunary' term $\phi_{1,l}$ and a `bulk' term $\phi_{1,b}$, whose properties can be exploited in different ways.

\subsubsection{The lacunary term}
We bound the first summand in \eqref{eq:lac-zero-split}.
In the subsequent computation $\theta \in B_{3}$ and $\zeta \in B_{2}$ are fixed, and we do not reference them in the notation.
For $i \in \{1,2,3\}$ and $\phi \in \Phi$ define
\begin{equation}\label{eq:G-H-decomposition}
  \begin{aligned}
    \mc{H}_i[\phi](\sigma)&:= \int_{0}^{\sigma} \frac{\rho}{\sigma} \rho\DEF_{\sigma}\mc{F}_i[\phi](\rho) \, \frac{\dd\rho}{\rho} \\
    \mc{G}_i[\phi](\sigma) &:= \mc{F}_i[\phi](\sigma) - \mc{H}_i[\phi](\sigma).
  \end{aligned}
\end{equation}
We plan to apply Lemma \ref{lem:3gamma-decomp} to $\mc{F}_i=\mc{G}_i+\mc{H}_i$, but before that let's estimate the norms that will come into play.
First observe that by Corollary \ref{cor:W11-Haar}
\begin{equation*}
  \|\mc{H}_i[\phi]\|_{\gamma_{\dd\sigma/\sigma}} \lesssim \|\mc{H}_i[\phi]\|_{L^{1}_{\dd\sigma/\sigma}}
  + \|\sigma\partial_{\sigma}\mc{H}_i[\phi]\|_{L^{1}_{\dd \sigma/\sigma}},    
\end{equation*}
where the first term satisfies the estimate 
\begin{equation*}
  \begin{aligned}
    \|\mc{H}_i[\phi]\|_{L^{1}_{\dd\sigma/\sigma}} &= \int_{0}^{1} \Big| \int_{0}^{\sigma} \frac{\rho}{\sigma}\, \rho\DEF_{\sigma}\mc{F}_i[\phi](\rho) \, \frac{\dd\rho}{\rho} \Big| \, \frac{\dd \sigma}{\sigma} \\
    &\leq \int_{0}^{1} \Big(\int_{\rho}^{1} \frac{\rho}{\sigma} \, \frac{\dd\sigma}{\sigma} \Big) |\rho\DEF_{\sigma}\mc{F}_i[\phi](\rho) | \, \frac{\dd\rho}{\rho}
    \lesssim \|\sigma\DEF_{\sigma}\mc{F}_i[\phi]\|_{L^{1}_{\dd\sigma/\sigma}},
  \end{aligned}
\end{equation*}
while the second term can be bounded by differentiating under the integral and using the previous estimate:
\begin{equation*}
  \|\sigma\partial_{\sigma}\mc{H}_i[\phi]\|_{L^{1}_{\dd\sigma/\sigma}}
  =  \| \sigma\DEF_{\sigma}\mc{F}_i[\phi] - \mc{H}_i[\phi]\|_{L^{1}_{\dd\sigma/\sigma}}
  \lesssim \|\sigma\DEF_{\sigma}\mc{F}_i[\phi]\|_{L^{1}_{\dd\sigma/\sigma}}.
\end{equation*}
By subtracting this gives us
\begin{equation*}
  \|\mc{G}_i[\phi]\|_{\gamma_{\dd\sigma/\sigma}}
  \lesssim \|\mc{F}_i[\phi]\|_{\gamma_{\dd\sigma/\sigma}}
  +  \|\sigma\DEF_{\sigma}\mc{F}_i[\phi]\|_{L^{1}_{\dd\sigma/\sigma}},
\end{equation*}
while 
\begin{equation*}
  \sigma\partial_{\sigma}\mc{G}_i[\phi]
  = \sigma\wpD_{\sigma}\mc{F}_i[\phi] + \mc{H}_i[\phi] 
\end{equation*}
leads to the estimate
\begin{equation*}
  \| \sigma\partial_{\sigma}\mc{G}_i[\phi]\|_{\gamma_{\dd\sigma/\sigma}} \lesssim
  \|\sigma\wpD_{\sigma}\mc{F}_i[\phi]\|_{\gamma_{\dd\sigma/\sigma}}
  + \|\sigma\DEF_{\sigma}\mc{F}_i[\phi]\|_{L^{1}_{\dd\sigma/\sigma}}.
\end{equation*}
Finally we estimate the $R$-bound of the range of $\mc{H}_i$ via Proposition \ref{prop:W11-rbd} (whose boundary term vanishes since $\mc{F}_i$ is compactly supported) and the estimate from before:
\begin{equation*}
  \begin{aligned}[t]    
    R_{\Pi}\big(\mc{H}_i[\phi]((0,1))\big)
    \lesssim_{\Pi} \|\sigma\partial_{\sigma}\mc{H}_i[\phi]\|_{L^{1}_{\dd\sigma/\sigma}}
    \lesssim \|\sigma\DEF_{\sigma}\mc{F}_i[\phi]\|_{L^{1}_{\dd\sigma/\sigma}}.
  \end{aligned}
\end{equation*}
Thus applying Lemma \ref{lem:3gamma-decomp} results in the estimate 
\begin{equation*}
  \begin{aligned}
    &\Big| \int_{0}^{1}\Pi \Big( \mc{F}_{1}[\phi_{1,l}] (\sigma), \big(\mc{F}_{i}[\phi_{i}] (\sigma) \big)_{i=2,3} \Big)  \, \frac{\dd \sigma}{\sigma}\Big| \\
    &\lesssim 
    \Big(
    \|\mc{F}_{1}[\phi_{1,l}]\|_{\gamma_{\dd \sigma/\sigma}}
    + \|\sigma\wpD_{\sigma}\mc{F}_{1}[\phi_{1,l}]\|_{\gamma_{\dd\sigma/\sigma}}
    + \|\sigma\DEF_{\sigma}\mc{F}_{1}[\phi_{1,l}]\|_{L^{1}_{\dd \sigma/\sigma}}
    \Big)
    \\
    & \qquad \times\prod_{i=2}^{3}
    \Big(\|\mc{F}_{i}[\phi_{i}]\|_{\gamma_{\dd\sigma/\sigma}}
    + \|\sigma\wpD_{\sigma}\mc{F}_{i}[\phi_{i}]\|_{\gamma_{\dd\sigma/\sigma}}
    + \|\sigma\DEF_{\sigma}\mc{F}_{i}[\phi_{i}]\|_{L^{1}_{\dd \sigma/\sigma}} \Big).  
  \end{aligned}
\end{equation*}
Integrating and using the Hölder inequality in $(\theta,\zeta)\in B_{3}\times B_{2}$ controls the first summand of \eqref{eq:lac-zero-split} by
\begin{equation*}
  \prod_{i=1}^{3} \|F_{i}\|_{\FS_{\Theta_{i}}^{s_{i}}(T_{i})},
\end{equation*}
as required, noting that all the wave packets in play are members of the appropriate classes (see  Remark \ref{rmk:wpd-lacunary-size}).

\subsubsection{The bulk term}
Now let us deal with the second summand of \eqref{eq:lac-zero-split}, 
\begin{equation*}
  \int_{B_{3}}\Big| \int_{B_{2}\times (0,1)} \Pi \Big( \int_{0}^{\sigma} \rho \partial_{\rho} \mc{F}_{1}[\phi_{1,b}] (\theta_{1}, \zeta, \rho) \, \frac{\dd \rho}{\rho}, \big(\mc{F}_{i}[\phi_{i}] (\theta_{i}, \zeta, \sigma) \big)_{i=2,3} \Big)  \, \frac{\dd \sigma}{\sigma} \, \dd \zeta\Big| \, \dd \theta,
\end{equation*}
recalling that $\phi_{1,b}=0$ when $\theta_{1}\notin B_{\mf{b}/2}$.
Split $\partial_{\rho} \mc{F}_{1}$ into the sum $\DEF_{\sigma}\mc{F}_{1} + \wpD_{\sigma} \mc{F}_{1}$, yielding two summands, $\mb{B}_{1}$ and $\mb{M}_{1}$.
To estimate $|\mb{B}_{1}|$,  fix $(\theta,\zeta) \in B_{3}\times B_{2}$ with $\theta_{1}\in B_{\mf{b}/2}$, suppress them from the notation, and write
\begin{equation*}\begin{aligned}[t]
    & \Big| \int_{0}^{1}   \Pi \Big( \int_{0}^{\sigma} \rho\DEF_{\sigma}\mc{F}_{1}[\phi_{1,b}] (\rho) \, \frac{\dd \rho}{\rho} , \big(\mc{F}_{i}[\phi_{i}] (\sigma) \big)_{i=2,3} \Big)  \, \frac{\dd \sigma}{\sigma} \Big|
    \\ & \lesssim \int_{0}^{1} \Big| \int_{\rho}^{1} \Pi \Big( \rho \DEF_{\sigma} \mc{F}_{1}[\phi_{1,b}] (\rho),  \big(\mc{F}_{i}[\phi_{i}] (\sigma) \big)_{i=2,3} \Big) \, \frac{\dd \sigma}{\sigma} \Big| \, \frac{\dd \rho}{\rho} .
  \end{aligned}  \end{equation*}
By Proposition \ref{prop:gamma-holder}, exploiting that the first factor is independent of $\sigma$ we have
\begin{equation*} \begin{aligned}[t]
    & \int_{0}^{1} \Big| \int_{\rho}^{1} \Pi \Big( \rho \DEF_{\sigma}\mc{F}_{1}[\phi_{1,b}] (\rho),  \big(\mc{F}_{i}[\phi_{i}] (\sigma) \big)_{i=2,3} \Big)  \, \frac{\dd \sigma}{\sigma} \Big| \frac{\dd \rho}{\rho}
    \\ & \lesssim_{\Pi} \Big(\int_{0}^{1} \|\rho\DEF_{\sigma}\mc{F}_{1}[\phi_{1,b}](\rho)\|_{X_{1}} \, \frac{\dd \rho}{\rho}\Big) \prod_{i=2,3} \| \mc{F}_{i}[\phi_{i}]\|_{\gamma_{\dd \sigma / \sigma}},
  \end{aligned} \end{equation*}
so integrating and using H\"older's inequality yields
\begin{equation*}
  |\mb{B}_{1}| \lesssim \|F_{1}\|_{\DS^{\sigma}_{\Theta_{1}}(T_{1})} \prod_{i=2,3} \|F_{i}\|_{\RS_{\Theta_{i}}^{s_{i}} ( T_{i}) }.
\end{equation*}
Now we deal with the term $\mb{M}_{1}$, given by
\begin{equation*}
  \int_{B_{3}} \Big| \int_{B_{2}} \int_{0}^{1} \Pi \Big( \int_{0}^{\sigma} \rho \wpD_{\rho} \mc{F}_{1}[\phi_{1,b}] (\theta_{1}, \zeta, \rho) \, \frac{\dd \rho}{\rho} , \big(\mc{F}_{i}[\phi_{i}] (\theta_{i} , \zeta, \sigma) \big)_{i=2,3} \Big) \, \frac{\dd \sigma}{\sigma} \, \dd \zeta \Big|\, \dd \theta.
\end{equation*}
Let $\psi_{1,b}(z) = z\phi_{1,b}(z)$, so that $\psi_{1,b} \in \Psi_{1}(\theta_{1})$ for all $\theta_{1}\in\R$; this holds since $2\pi i\FT{\psi_{1,b}}(-\theta_{1})= \FT{\phi_{0}}(-\theta_{1}) \FT{\upsilon}'(-\theta_{1})$, which vanishes because either $\FT{\phi_{0}}(-\theta_{1})=0$ (if $\theta_{1}\notin B_{\mf{b}/2}$) or $\FT{\upsilon}'(-\theta_{1})=0$  (if $\theta_{1}\in B_{2\mf{b}/3}$).
We have that
\begin{equation*}
  \wpD_{\rho} \mc{F}_{1}[\phi_{1,b}] = \wpD_{\zeta} \mc{F}_{1}[\psi_{1,b}] = -\DEF_{\zeta}\mc{F}_{1}[\psi_{1,b}] + \partial_{\zeta}\mc{F}_{1}[\psi_{1,b}].
\end{equation*}
This induces a splitting $\mb{M}_{1} = \mb{B}_{2} + \mb{M}_{2}$, and by the previous argument we have
\begin{equation*}
  |\mb{B}_{2}| \lesssim \|F_{1}\|_{\DS^{\zeta}_{\Theta_{1}}( T_{1} )} \prod_{i=2,3} \|F_{i}\|_{\RS_{\Theta_{i}}^{s_{i}} ( T_{i} )}. 
\end{equation*}
For the term $\mb{M}_2$, given by
\begin{equation*}
  \int_{B_{3}}\Big| \int_{B_{2}\times (0,1)} \Pi \Big( \int_{0}^{\sigma} \rho\partial_{\zeta} \mc{F}_{1}[\psi_{1,b}] (\theta_{1} , \zeta, \rho) \, \frac{\dd \rho}{\rho} , \big(\mc{F}_{i}[\phi_{i}] (\theta_{i} , \zeta, \sigma) \big)_{i=2,3}  \, \frac{\dd \sigma}{\sigma}\Big) \, \dd \zeta \Big|\, \dd \theta,
\end{equation*}
we integrate by parts in $\zeta$: for fixed $\theta$ the negative of the $\zeta$-integral is
\begin{equation*} \begin{aligned}[t]
    & \int_{B_{2}} \Pi \Big( \mc{F}_{1}[\psi_{1,b}] (\theta_{1} , \zeta, \rho) ,  \partial_{\zeta} \mc{F}_{2}[\phi_{2}] (\theta_{2} , \zeta, \sigma), \mc{F}_{3}[\phi_{3}] (\theta_{3} , \zeta, \sigma) \Big) \, \dd \zeta
    \\ & + \int_{B_{2}} \Pi \Big( \mc{F}_{1}[\psi_{1,b}] (\theta_{1} , \zeta, \rho) ,  \mc{F}_{2}[\phi_{2}] (\theta_{2} , \zeta, \sigma), \partial_{\zeta} \mc{F}_{3}[\phi_{3}] (\theta_{3} , \zeta, \sigma) \Big) \, \dd \zeta.
  \end{aligned}  \end{equation*}
There are no boundary terms, as the integrand is compactly supported.
Both of these terms are treated in the same way, so we will only do the first one.
Write
\begin{equation*}
  \partial_{\zeta} \mc{F}_{2} = \DEF_{\zeta}\mc{F}_{2} + \wpD_{\zeta} \mc{F}_{2},
\end{equation*}
which decomposes the corresponding summand of $\mb{M}_{2}$ into two parts, $\mb{B}_{3}$ and $\mb{M}_{3}$.
The integrand of $\mb{B}_{3}$ (with $(\theta,\zeta)\in B_{3}\times B_{2}$ suppressed) is controlled by
\begin{equation*} \begin{aligned}[t]
    & \int_{0}^{1} \Big( \int_{0}^{\sigma}\frac{\rho}{\sigma} \|\mc{F}_{1}[\psi_{1,b}](\rho)\|_{X_{1}} \, \frac{\dd \rho}{\rho} \Big) \, \|\sigma\DEF_{\zeta}\mc{F}_{2}[\phi_{2}](\sigma)\|_{X_{2}} \|\mc{F}_{3}[\phi_{3}](\sigma)\|_{X_{3}} \, \frac{\dd \sigma}{\sigma}
    \\ & \leq  \Big( \esssup_{\rho \in \R_{+}} \|\mc{F}_{1}[\psi_{1,b}](\rho)\|_{X_{1}} \Big) \int_{0}^{1} \|\sigma\DEF_{\zeta}\mc{F}_{2}[\phi_{2}](\sigma)\|_{X_{2}} \|\mc{F}_{3}[\phi_{3}](\sigma)\|_{X_{3}} \, \frac{\dd \sigma}{\sigma},
  \end{aligned} \end{equation*}
which leads to the bound
\begin{equation*}
  |\mb{B}_{3}| \lesssim  \|F_{2}\|_{\DS^{\zeta}_{\Theta_{2}}( T_{2} )} \prod_{i=1,3}\|F_{i}\|_{\lL_{\Theta_{i}}^{\infty}( T_{i})}.
\end{equation*}

It remains to handle $\mb{M}_{3}$. With $(\theta,\zeta)\in B_{3} \times B_{2} $ suppressed, the integrand in $\mb{M}_{3}$ is given by
\begin{equation*}
  \int_{0}^{1} \Pi \Big( \int_{0}^{\sigma}\frac{\rho}{\sigma} \mc{F}_{1}[\psi_{1,b}] (\rho) \, \frac{\dd \rho}{\rho} , \sigma\wpD_{\sigma}\mc{F}_{2}[\phi_{2}] (\sigma), \mc{F}_{3}[\phi_{3}] (\sigma) \Big) \, \frac{\dd \sigma}{\sigma}.
\end{equation*}
This is controlled via a similar argument to that used to bound the first summand of \eqref{eq:lac-zero-split}.
Let $\psi_{2}(z)=(2\pi i \theta_{2}-\partial_{\zeta})\phi_{2}(z)$, so that $\psi_{2} \in \Psi(\theta_2)$ (by Remark \ref{rmk:wpd-lacunary-size}) and $\sigma\wpD_{\sigma}\mc{F}_{2}[\phi_{2}] (\sigma)=\mc{F}_{2}[\psi_{2}] (\sigma)$.
Write $\mc{F}_{2} = \mc{G}_2 + \mc{H}_2$ and $\mc{F}_{3} = \mc{G}_3 + \mc{H}_3$ as in \eqref{eq:G-H-decomposition}, and bound the summands by exactly the same argument.
Only the first factor is treated differently: in this case the function
\begin{equation*}
  \tilde{\mc{F}_{1}}[\psi_{1,b}] (\sigma):= \int_{0}^{\sigma}\frac{\rho}{\sigma} \mc{F}_{1}[\psi_{1,b}] (\rho) \, \frac{\dd \rho}{\rho} 
\end{equation*}
satisfies
\begin{equation*}
  \|\tilde{\mc{F}_{1}}[\psi_{1,b}]\|_{\gamma_{\dd \sigma/\sigma}} \lesssim \|\mc{F}_{1}[\psi_{1,b}]\|_{\gamma_{\dd\sigma/\sigma}}
\end{equation*}
by Corollary \ref{cor:gamma-convolution}, and
\begin{equation*}
  \|\sigma\partial_{\sigma}\tilde{\mc{F}_{1}}[\psi_{1,b}]\|_{\gamma_{\dd\sigma/\sigma}} = \|\mc{F}_{1}[\psi_{1,b}] - \tilde{\mc{F}_{1}}[\psi_{1,b}]\|_{\gamma_{\dd\sigma/\sigma}}
  \lesssim \|\mc{F}_{1}[\psi_{1,b}]\|_{\gamma_{\dd\sigma/\sigma}}
\end{equation*}
using differentiation under the integral along with the previous estimate.
Thus we decompose $\tilde{\mc{F}_{1}} = \mc{G}_{1} + \mc{H}_{1}$ simply by taking $\mc{H}_{1} = 0$, and the same argument that bounded the first summand of \eqref{eq:lac-zero-split} leads to
\begin{equation*}
  |\mb{M}_{3}| \lesssim \prod_{i=1}^{3} \|F_{i}\|_{\FS_{\Theta_{i}}^{s_{i}} ( T_{i} ) },
\end{equation*}
bounding the second summand of \eqref{eq:lac-zero-split}, establishing the claimed bound \eqref{eq:sizeholder-claim}, and completing the proof.


%% file: main/size-domination.tex
By Corollary \ref{cor:BHT-RN-reduction} and the fact that
\begin{equation*}
  \widetilde{\BHF}_\Pi(f_1,f_2,f_3) = \lim_{\mbb{K} \uparrow \R^3_+} \BHWF_\Pi^{\phi_0, \mbb{K}}(\Emb[f_1], \Emb[f_2], \Emb[f_3]),  
\end{equation*}
we see that we are tasked with proving bounds of the form\footnote{Here we do not make reference to the frequency band $\Theta$, as the precise choice is no longer relevant.}
\begin{equation}\label{eq:emb-bounds}
  \|\Emb[f]\|_{L_{\nu}^{p} \sL_{\mu}^{q} \FS^{s}} \lesssim \|f_i\|_{L^{p}(\R;X_i)}.
\end{equation}
The local sizes $\FS^{s}$ are defined as the sum of four local sizes (Lebesgue, $\gamma$, and the two defect local sizes) so as to make the size-Hölder inequality work.
When applied to embedded functions, it turns out that all of these sizes are controlled by the `outer' $\gamma$ local size $\RS_{out}^2$, so in proving \eqref{eq:emb-bounds} this is the only local size we need to consider.
In this section we prove these statements.

\subsection{The size domination principle}

We make use of a principle that we call \emph{size domination}, which was introduced by the second author in \cite{gU16}.
Since we apply it multiple times to multiple local sizes, it is useful to make an abstract form of the principle into a definition. 

\begin{defn}[Size domination for embedded functions]\label{def:size-domination-embedded}
  Let $X$ be a Banach space, and let $(\R^{3}_{+},\OB, \sigma, \OS_{1})$ and $(\R^{3}_{+},\OB, \sigma, \OS_{2})$ be two outer space structures on $\R^{3}_{+}$ with $\Lin(\Phi;X)$-valued local sizes $\OS_{1}$ and $\OS_{2}$. We say that $\OS_{2}$ \emph{dominates $\OS_{1}$ with respect to the embedding map $\Emb$ and measure $\sigma$}, written $\OS_1\prec^{\Emb,\sigma}\OS_2$, if for all $B \in \OB$ and $V \in \OB^{\cup}$ one has
  \begin{equation*}
    \begin{aligned}[t]
      & \|\1_{\R^3_+ \sm V} \Emb[f] \|_{\OS_{1}(B)} \lesssim \sup_{\substack{B'\in\OB\\ \sigma(B^{'})\lesssim \sigma(B)}}\|\1_{\R^{3}_{+}\sm V} \Emb[f]\|_{\OS_{2}(B')} \qquad \forall f \in \Sch(\R;X)
    \end{aligned}
  \end{equation*}
  with implicit constants independent of $f$, $B$, and $V$.
\end{defn}

This domination has the following consequences for outer Lebesgue quasinorms.

\begin{prop}\label{prop:embedding-domination}
  Let $X$ be a Banach space, and let $(\R^{3}_{+},\TT,\mu,\OS_{1})$ and $(\R^{3}_{+},\TT,\mu,\OS_{2})$ be two outer space structures with $\Lin(\Phi;X)$-valued sizes satisfying $\OS_{1}\prec^{\Emb,\mu}\OS_{2}$. Then for any $q\in(0,\infty]$ it holds that
  \begin{equation*}
    \begin{aligned}[t]
      &\| \Emb[f] \|_{L^{q}_{\mu}\OS_1} \lesssim \|\Emb[f]\|_{L^{q}_{\mu}\OS_2} \qquad \forall f \in \Sch(\R;X)
    \end{aligned}
  \end{equation*}
  and furthermore, if for all $F\in\Bor(\R^{3}_{+};\Lin(\Phi;X))$ and $i\in\{1,2\}$ we have that
  \begin{itemize}
  \item  $\|\1_{T_{1}}F\|_{\OS_{i}(T_{2})}\lesssim \|F\|_{\OS_{i}(T_{1})}$ for any $T_{1},T_{2}\in\TT$ with $T_{1}\subset T_{2}$, and
  \item $\|F\|_{\OS_{i}(T)}\lesssim\|\1_{T}F\|_{\OS_{i}(T)}$  for any $T\in\TT$,
  \end{itemize}
  then the sizes $\sL^{q}_{\mu}\OS_{i}$ on the iterated outer spaces $(\R^{3}_{+},\DD,\nu,\sL^{q}_{\mu}\OS_{1})$ and $(\R^{3}_{+},\DD,\nu,\sL^{q}_{\mu}\OS_{2})$ satisfy
  \begin{equation*}
    \begin{aligned}[t]
      \sL^{q}_{\mu}\OS_{1}\prec^{\Emb,\nu}\sL^{q}_{\mu}\OS_{2}\qquad \forall q\in(0,\infty].
    \end{aligned}
  \end{equation*}
  Thus for any $p\in(0,\infty]$ it holds that
  \begin{equation*}
    \begin{aligned}[t]
      &\| \Emb[f] \|_{L^{p}_{\nu}L^{q}_{\mu}\OS_1} \lesssim \|\Emb[f]\|_{L^{p}_{\nu}L^{q}_{\mu}\OS_2}.
    \end{aligned}
  \end{equation*}
\end{prop}

\begin{proof}
  For the first claim it is sufficient to show that for some $C>0$ we have
  \begin{equation*}
    \mu(\|\Emb[f]\|_{\OS_{1}}>C\lambda)\lesssim     \mu(\|\Emb[f]\|_{\OS_{2}}>\lambda).
  \end{equation*}
  Fix $\lambda$ and let $V_{\lambda}\in\TT^{\cup}$ be such that
  \begin{equation*}
    \|\1_{R^{3}_{+}\setminus V_{\lambda}}\Emb[f]\|_{\OS_{2}}\leq \lambda
    \quad \text{and} \quad
    \mu(V_{\lambda})\lesssim  \mu(\|\Emb[f]\|_{\OS_{2}}>\lambda).
  \end{equation*}
  By our assumptions on $\OS_{1}$ and $\OS_{2}$ it holds that for any $V_{-}\in\TT^{\cup}$,
  \begin{equation*}
    \begin{aligned}[t]
      \|\1_{\R^{3}_{+}\setminus(V_{-}\cup V_{\lambda})}\Emb[f]\|_{\OS_{1}(T)}\lesssim \sup_{\substack{T'\in\TT\\ \mu(T')\lesssim \mu(T)}}\|\1_{\R^{3}_{+}\setminus(V_{-}\cup V_{\lambda})}\Emb[f]\|_{\OS_{2}(T')}\lesssim \lambda,
    \end{aligned}
  \end{equation*}
  which yields our claim after taking the supremum over all $T\in\TT$ and $V_{-}\in\TT^{\cup}$.

  The assumption $\OS_{1}\prec^{\Emb,\mu}\OS_{2}$ is in itself not enough to guarantee that
  \begin{equation*}
    \begin{aligned}[t]
      \|\1_{D\setminus W}\Emb[f]\|_{\OS_{1}}\lesssim \sup_{\substack{D'\in\DD\\ \nu(D')\lesssim\nu(D)}}\|\1_{D' \setminus W}\Emb[f]\|_{\OS_{2}(T')},
    \end{aligned}
  \end{equation*}
  which would imply the second claim $\sL^{q}_{\mu}\OS_{1}\prec^{\Emb,\nu}\sL^{q}_{\mu}\OS_{2}$.
  However, fix $D\in\DD$, let $D'=D_{(x_{D},2s_{D})}$, and notice that for any $T\in\TT$,
  setting $T':=T\cap D'\in \TT$, by the additional assumptions we have 
  \begin{equation*}
    \begin{aligned}[t]
      \|\1_{D\setminus W}\Emb[f]\|_{\OS_1(T)}
      &\lesssim  \|\1_{D\setminus W}\Emb[f]\|_{\OS_1(T')} \\
      &\leq\|\1_{D'\setminus W}\Emb[f]\|_{\OS_1(T')} + \|\1_{D'\setminus (D\cup W)}\Emb[f]\|_{\OS_1(T')}
      \\ & =\|\1_{\R^{3}_{+}\setminus W}\Emb[f]\|_{\OS_1(T')}+\|\1_{\R^{3}_{+}\setminus (D\cup W)}\Emb[f]\|_{\OS_1(T')}
      \\ & \lesssim  \|\1_{\R^{3}_{+}\setminus W}\Emb[f]\|_{\OS_{2}(T')}+\|\1_{\R^{3}_{+}\setminus (D\cup W)}\Emb[f]\|_{\OS_{2}(T')}
      \\ &   \lesssim  \|\1_{D'\setminus W}\Emb[f]\|_{\OS_{1}(T')}+\|\1_{D'\setminus (D\cup W)}\Emb[f]\|_{\OS_{1}(T')},
    \end{aligned}
  \end{equation*}
  where the last inequality holds since $T'\subset D'$.
  Taking the supremum over $T$ yields the conclusion.
\end{proof}

The local sizes defined in Section \ref{sec:sizes} all satisfy the assumptions of the second part of Proposition \ref{prop:embedding-domination}.

\subsection{Domination of the defect and Lebesgue local sizes}

We begin with a technical lemma which controls defect local sizes of truncated functions.

\begin{lem}\label{lem:defect-domination}
  Let $X$ be a Banach space and $F \in C^1(\R^3_+ ; \Lin(\Phi;X))$.
  Then for all trees $T \in \TT$ and all $V\in \TT^\cup$,
  \begin{equation}\label{eq:sizedom-sigma}
    \|\1_{\R^{3}_{+} \setminus V} F\|_{\DS^{\sigma}(T)} \lesssim \|F\|_{\DS^{\sigma}(T)} + \|\1_{\R^{3}_{+} \setminus V}F\|_{\lL^\infty(T)}
  \end{equation}
  and
  \begin{equation}\label{eq:sizedom-zeta}
    \|\1_{\R^{3}_{+}\setminus V} F\|_{\DS^{\zeta}(T)} \lesssim \|F\|_{\DS^{\zeta}(T)} + \|\1_{\R^{3}_{+} \setminus V}F\|_{\lL^\infty(T)}.
  \end{equation}
\end{lem}

\begin{proof}
  We only prove \eqref{eq:sizedom-sigma}, as \eqref{eq:sizedom-zeta} has the same proof.
  By definition we have
  \begin{equation*}
    \begin{aligned}
      &\| \1_{\R^{3}_{+}\setminus V} F\|_{\DS^{\sigma}(T)} \\
      &\qquad =  \sup_{ \phi\in \Phi_{1}} \sup_{g} \Big| \int_{\R^{3}_{+}}\Big\langle (\sigma\DEF_{\sigma} \pi_T^* (\1_{\R^{3}_{+}\setminus V} F))[\varphi](\theta,\zeta,\sigma); g(\theta,\zeta,\sigma) \Big\rangle \, \frac{\dd \sigma}{\sigma} \, \dd \zeta \, \dd \theta\Big|,
    \end{aligned}
  \end{equation*}
  so for fixed $\phi\in \Phi_{1}$ and for $g \in C_c^\infty(\R^3_+ ; X^*)$ normalised in $L^1(\R^2; L^\infty(\R_+; X^*))$ we need to control
  \begin{equation}\label{eq:defect-1-expr}
    \Big| \int_{\overline{\mT}}\Big\langle (\sigma\DEF_{\sigma} \big(\1_{\pi_{T}^{-1}(\R^{3}_{+}\setminus V)} F^{*}[\varphi]\big)(\theta,\zeta,\sigma); g(\theta,\zeta,\sigma) \Big\rangle \, \frac{\dd \sigma}{\sigma} \, \dd \zeta \, \dd \theta\Big|,
  \end{equation}
  where $F^{*} :=\pi_{T}^{*} F $.
  By the product rule we have
  \begin{equation}\label{eq:defect-productrule}
    \DEF_{\sigma} \1_{\pi_{T}^{-1}(\R^{3}_{+}\setminus V)} F^{*} = \big( \partial_\sigma \1_{\pi_{T}^{-1}(\R^{3}_{+}\setminus V)} \big) F^{*} + \1_{\pi_{T}^{-1}(\R^{3}_{+}\setminus V)} \DEF_{\sigma} F^{*}.
  \end{equation}
  Since $V$ is a countable union of trees, there exists a Lipschitz function $\map{\tau}{\R^2}{[0,\infty)}$ such that
  \begin{equation}\label{eq:derivative-1}
    \partial_\sigma \1_{\pi_T^{-1}(\R^{3}_{+}\setminus V)}(\theta,\zeta,\sigma) = \delta(\sigma - \tau(\theta,\zeta))
  \end{equation}
  where $\delta$ is the usual Dirac delta distribution on $\R$ (see Figure \ref{fig:tree_cup} for a sketch of the case $A = T$).
  Substituting this into \eqref{eq:defect-productrule}, we estimate \eqref{eq:defect-1-expr} by the sum of two terms.
  The first of these terms is estimated by  
  \begin{equation*}
    \begin{aligned}
      &\Big| \int_{\overline{\mT}}\Big\langle \sigma \delta(\sigma - \tau(\theta,\zeta)) F^{*}[\varphi](\theta,\zeta,\sigma); g(\theta,\zeta,\sigma) \Big\rangle \, \frac{\dd \sigma}{\sigma} \, \dd \zeta \, \dd \theta\Big| \\
      &=\Big| \int_{\overline{\mT}}\Big\langle F^{*}[\varphi](\theta,\zeta,\tau(\theta,\zeta)); g(\theta,\zeta,\tau(\theta,\zeta)) \Big\rangle \, \dd \zeta \, \dd \theta\Big| \\
      &\leq \|F^{*}[\phi]\|_{L^\infty(\mT \setminus \pi_T^{-1}(V);X)} \|g\|_{L^1(\R^{2}; L^\infty(\R_{+};X^*))} = \|\1_{\R_+^3 \sm V} F\|_{\lL^\infty(T)} .
    \end{aligned}
  \end{equation*}
  The second term is evidently controlled by $\|F\|_{\DS^{\sigma}(T)}$.
\end{proof}

\begin{figure}[h]
  \includegraphics{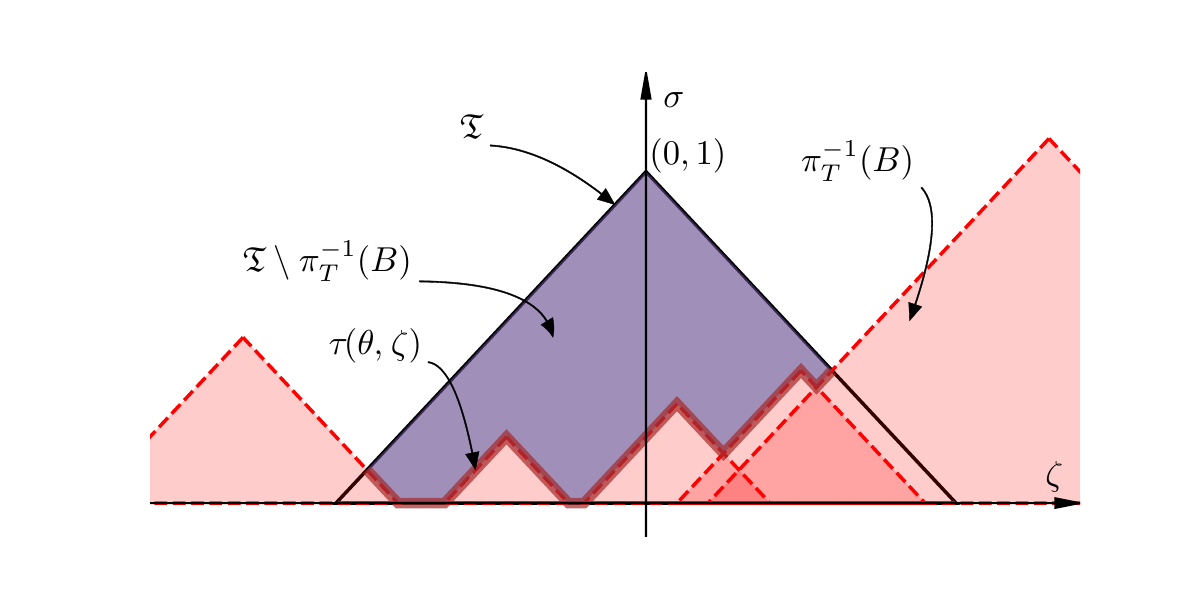}
  \caption{
    A tree $\mT$ and the set $\pi_{T}^{-1}(B)$ with $B\in\TT^{\cup}$ in the $\zeta=0$ plane.   For $\mc{F}=\pi_{T}^{*}\big(\1_{\R^{3}_{+}\setminus B} \Emb[f]\big)$, the terms $\DEF_{\sigma}\mc{F}$ and $\DEF_{\zeta}\mc{F}$ are supported on the boundary of $\mT \sm \pi_T^{-1}(B)$.
  }
  \label{fig:tree_cup}
\end{figure}

\begin{cor}\label{cor:defect-size-dom}
  Let  $X$  be any Banach space and $N\in\N$. Then
  \begin{equation*}
    \DS^{\sigma}_{N}+\DS^{\zeta}_{N}\prec^{\Emb,\mu}\lL^{\infty}_{N}.
  \end{equation*}
\end{cor}

\begin{proof}
  For embedded functions we have
  \begin{equation*}
    \begin{aligned}[t]    
      \|\Emb[f]\|_{\DS_{N}^{\sigma}(T)}&=    \|\Emb[f]\|_{\DS_{N}^{\zeta}(T)}\\ & =
      \sup_{\phi\in\Phi^{N}_{1}} \sup_{(\theta,\zeta)\in\Theta\times B_{1}}\|F^{*}[\varphi](\theta,\zeta,1-|\zeta|)\|_{X}\lesssim \|\Emb[f]\|_{\lL_{N}^{\infty}(T)}.
    \end{aligned}
  \end{equation*}
  Combined with Lemma \ref{lem:defect-domination}, this completes the proof. 
\end{proof}

Next we control the Lebesgue local size $\lL^\infty$ on embedded functions by the outer $\gamma$ local size $\RS_{out}^{2}$.
The proof is done by controlling pointwise values of $\Emb[f]$ by $\gamma$-norms over an appropriate region; this uses a Sobolev embedding argument and the wave packet differential equations \eqref{eq:wpDEs} for embedded functions. 

\begin{lem}\label{lem:overlap-size-dom}
  Let $X$ be a Banach space with finite cotype and $N\in\N$. Then
  \begin{equation*}
    \lL^{\infty}_{N+3} \prec^{\Emb,\mu} \RS_{N,out}^{2}.
  \end{equation*}
\end{lem}

\begin{proof}
  We must show that
  \begin{equation}
    \label{eq:overlap-SD-goal}
    \|\Emb[f][\phi](\xi,x,s)\|_{X} \lesssim \|\1_{\R^{3}\setminus V} \Emb[f]\|_{\RS_{N,out}^{2}}
  \end{equation}
  for all $(\xi,x,s)\in\R^{3}_{+}\setminus V$ and $\phi \in \Phi_1^{N+3}$.
  Fix such $(\xi,x,s)$ and $\phi$ and consider the tree $T\in\TT$ with
  \begin{equation*}
    \begin{aligned}
      \xi_{T} = \xi-2\mf{b}s^{-1},\quad x_{T}=x,\quad s_{T}=C s
    \end{aligned}
  \end{equation*}
  for some sufficiently large $C>1$ (independent of $T$). It holds that
  \begin{equation*}
    \pi_{T}(2\mf{b},0,C^{-1}) = (\xi,x,s) \notin V.
  \end{equation*}
  Write $\Theta =(\theta^{*}_{-},\theta^{*}_{+})$ and set
  \begin{equation*}
    \begin{aligned}
      \Omega := \Big\{(\theta,\zeta,\sigma)\in\mT : \theta>2\mf{b}, |\zeta|<\sigma-C^{-1},
      \sigma>\frac{\theta-\theta^{*}_{-}}{2\mf{b}-\theta^{*}_{-}}C^{-1}\Big\}.
    \end{aligned}
  \end{equation*}
  Then $ \pi_{T}(\theta,\zeta,\sigma)\notin V$ for all $(\theta,\zeta,\sigma)\in \Omega$ (see Figure \ref{fig:tree-sobolev}, and compare it with Figure \ref{fig:tree-freq}).
  Fix $x^{*}\in X^{*}$ and apply the Sobolev embedding $W^{1}_{3}(\Omega) \hookrightarrow C(\overline{\Omega})$ (see \cite[Theorem 4.12, Part II]{AF03}) to the scalar-valued function $\langle\pi_{T}^{*}\Emb[f][\phi] ; x^{*}\rangle$: we find that
  \begin{equation*}
    \begin{aligned}
      &
      \big| \big\langle \pi_{T}^{*}\Emb[f][\phi](2\mf{b},0,C^{-1})  ; x^{*} \big\rangle \big|
      \\
      &
      \lesssim
      \sum_{a_{1}+a_{2}+a_{3}\leq 3}     \Big\|\big\langle \partial_{\theta}^{a_{1}}\partial_{\zeta}^{a_{2}} \partial_{\sigma}^{a_{3}} \pi_{T}^{*}\Emb[f][\phi] ;x^{*}\big\rangle\Big\|_{L^{1}_{\dd \theta \dd \zeta \dd \sigma}(\Omega) }
      \\
      &
      =
      \sum_{a_{1}+a_{2}+a_{3}\leq 3}     \Big\|\big\langle \wpD_{\theta}^{a_{1}}\wpD_{\zeta}^{a_{2}} \wpD_{\sigma}^{a_{3}} \pi_{T}^{*}\Emb[f][\phi] ; x^{*}\big\rangle\Big\|_{L^{1}_{\dd \theta \dd \zeta \dd \sigma}(\Omega) }
      \\
      &
      \lesssim\sup_{\phi\in\Phi_{1}^N}
      \Big\|\big\langle\pi_{T}^{*} \Emb[f][\phi];x^{*}\big\rangle\Big\|_{L^{1}_{\dd \theta \dd \zeta \dd \sigma}(\Omega) }.
    \end{aligned}
  \end{equation*}
  using \eqref{eq:embedded-DEs} and the definition of the wave packet differentials.
  By the $\gamma$-H\"older inequality and multiplier theorem (Proposition \ref{prop:gamma-holder} and Theorem \ref{thm:gamma-multiplier}, which needs finite cotype) we have
  \begin{align*}
    &\big\|\big\langle\pi_{T}^{*} \Emb[f][\phi];x^{*}\big\rangle\big\|_{L^{1}_{\dd \theta \dd \zeta \dd \sigma}(\Omega) } \\
    &\leq
      \|\1_{\Omega} \pi_{T}^{*} \Emb[f][\phi]\|_{L^{2}_{\dd \theta \dd \zeta} (\R^2 ; \gamma_{\dd t / t} (0,1;X)) } \|x^*\|_{X^{*}}  \|\1_{\Omega}\|_{L^2(\Omega)}
    \\
    &
      \lesssim
      \| \1_{\R^3_+ \sm V} \Emb[f]\|_{\RS_{N,out}^{2}(T)} \|x^*\|_{X^{*}} .
  \end{align*}
  Taking the supremum over $x^* \in X^*$ yields \eqref{eq:overlap-SD-goal}.
\end{proof}

\begin{figure}[h]
  \includegraphics{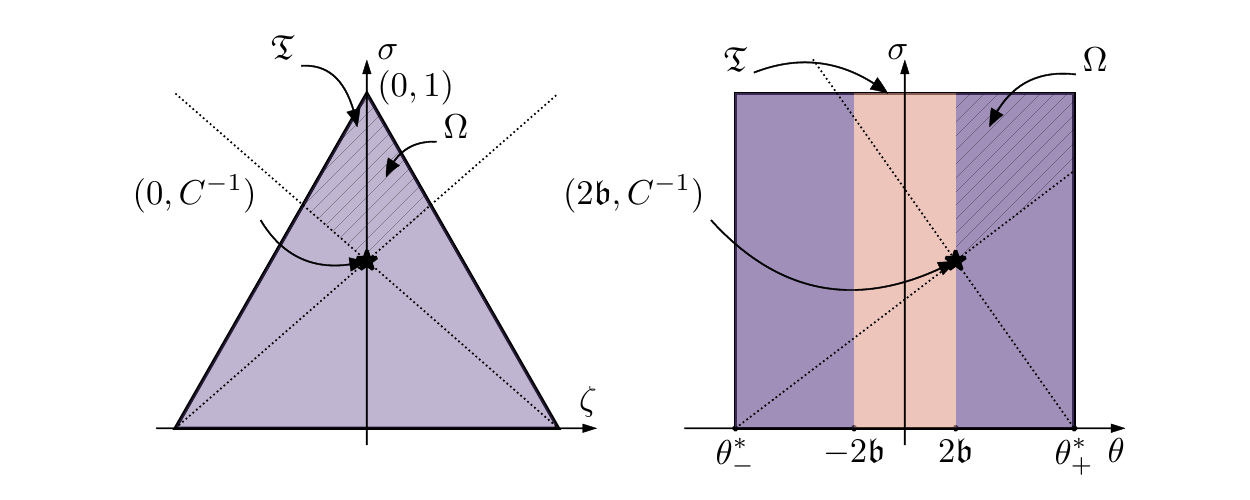}
  \caption{The set $\Omega$ used in the proof of Lemma \ref{lem:overlap-size-dom}}
  \label{fig:tree-sobolev}
\end{figure}

\subsection{Domination of the full \texorpdfstring{$\gamma$}{gamma} local size}

Recall that the full $\gamma$ local size $\RS^2$ is defined with respect to wave packets $\psi \in \Psi_1(\theta)$, where $\theta$ ranges over the whole frequency band $\Theta$, while the outer $\gamma$ local size $\RS_{out}^2$ only sees base frequencies with $|\theta| > 2\mf{b}$.
Although $\RS^2$ is larger, it is dominated by $\RS_{out}^2$ on embedded functions.

\begin{prop}\label{prop:full_R_domination}
  Let $X$ be a Banach space and $N \in \N$. Then
  \begin{equation*}
    \RS_{2N+4}^{2}\prec^{\Emb,\mu} \RS_{N,out}^{2} + \lL^{\infty}_{N}.
  \end{equation*}
\end{prop}

\begin{proof}
  Fix $f\in \Sch(\R;X)$, $T\in\TT$, and $V\in\TT^{\cup}$, and by symmetry assume $T = T_{(0,0,1)}$.
  Let $F^{*}:=\pi_{T}^{*}\Emb[f]$ and $V^{*}:= \pi_{T}^{-1}(V)$, so that $\pi_{T}^{*}(\1_{\R^{3}_{+}\setminus V} F) = \1_{\R^{3}_{+}\setminus V^{*}} F^{*}$, and let $N':=2N+4$.
  Our goal is to show that
  \begin{equation} \label{eq:full_R_domination:local-coordinates} \begin{aligned}[t]
      \int_{\R^{2}} \sup_{\psi \in \Psi^{N'}_{1}(\theta)} & \|\1_{\R^{3}_{+}\setminus V^{*}} F^{*}[\psi] (\theta,\zeta,\sigma) \|_{\gamma_{\dd \sigma / \sigma}}^2 \, \dd \zeta \dd \theta
      \\ & \lesssim  \sup_{\substack{T'\in\TT\\s_{T'}\lesssim s_{T}}}\Big(\|\1_{\R^{3}_{+}\setminus V}\Emb[f]\|_{\RS_{N,out}^{2}(T')}^{2} + \|\1_{\R^{3}_{+}\setminus V}\Emb[f]\|_{\lL^{\infty}_{N}(T')}^{2}\Big).
    \end{aligned}  \end{equation}
  Since $\Psi(\theta) = \Phi$ for all $\theta\in \Theta^{out}$, it suffices to control
  \begin{equation} \label{eq:full_R_domination:local-coordinates:in}
    \int_{\Theta^{in}\times B_{1}} \sup_{\psi \in \Psi^{N'}_{1}(\theta)}  \|\1_{\R^{3}_{+}\setminus V^{*}} F^{*}[\psi] (\theta,\zeta,\sigma) \|_{\gamma_{\dd \sigma / \sigma}}^2 \, \dd \zeta \dd \theta.
  \end{equation}
  
  First we control a bounded range of scales $\sigma$ using the $\lL^{\infty}_{N}$ term, which simplifies comparing the behavior of $\1_{\R^{3}_{+}\setminus V^{*}}F^{*}(\theta,\zeta,\sigma)$ for different frequencies $\theta$.
  For a large constant $C > 1$ to be determined and for $\zeta \in B_{1}$ let
  \begin{equation*}
    \sigma^{\downarrow} = \sigma^{\downarrow}(\zeta) := C \inf\{\sigma \colon \exists\theta\in\Theta_{in} \text{ such that } (\theta,\zeta,\sigma)\notin V^{*}\}.
  \end{equation*}
  If $C$ is sufficiently large, then for all $\sigma>\sigma^{\downarrow}$ it holds that
  \begin{equation}\label{eq:V-condn}
    (\theta',\zeta,\sigma) \notin V^{*}  \qquad \forall\theta'\in B_{5\mf{b}}.
  \end{equation}
  The argument to see this fact is to that appearing in the proof of Lemma \ref{lem:overlap-size-dom} (see Figure \ref{fig:tree-sobolev}).
  The cases where $\sigma^{\downarrow}(\zeta)=+\infty$ are of no importance.
  Split \eqref{eq:full_R_domination:local-coordinates:in} into the ranges $\sigma \in [C^{-1}\sigma^{\downarrow}, \sigma^{\downarrow}]$ and $\sigma \in [\sigma^{\downarrow}, 1)$,
  and to control the first term fix $\theta \in \Theta^{in}$, $\zeta \in B_{1}$, and $\psi \in \Psi^{N'}_{1}(\theta)$ and use Corollary \ref{cor:W11-Haar} to write
  \begin{equation*} \begin{aligned}[t]
      &\|\1_{\R^{3}_{+}\setminus V^{*}} F^{*}[\psi] (\theta,\zeta,\sigma)\, \1_{[C^{-1}\sigma^{\downarrow}, \sigma^{\downarrow}]}(\sigma) \|_{\gamma_{\dd \sigma / \sigma}}
      \\ &\lesssim_C \begin{aligned}[t]
        &\|\1_{\R^{3}_{+}\setminus V^{*}} F^{*}[\psi] (\theta,\zeta,\sigma)\|_{L^\infty(\R_+;X)}
        \\ & + \|\1_{\R^{3}_{+}\setminus V^{*}}  F^{*}[(2\pi i \theta - \partial_z)(z\psi(z))] (\theta,\zeta,\sigma)\|_{L^\infty(\R_+;X)} ,
      \end{aligned}
    \end{aligned} \end{equation*}
  so that the  $[C^{-1}\sigma^{\downarrow}, \sigma^{\downarrow}]$ term is dominated by $\|\1_{\R^{3}_{+}\setminus V} \Emb[F] \|_{\lL^{\infty}_{N}(T)}^{2}$. 

  We are left with controlling
  \begin{equation} \label{eq:full_R_domination:local-coordinates:good-scales}
    \int_{\Theta^{in}\times B_{1}}  \sup_{\psi \in \Psi^{N'}_{1}(\theta)}  \|\1_{\R^{3}_{+}\setminus V^{*}} F^{*}[\psi] (\theta,\zeta,\sigma)\,\1_{[\sigma^{\downarrow},1)}(\sigma) \|_{\gamma_{\dd \sigma / \sigma}}^2 \, \dd \zeta \dd \theta.
  \end{equation}
  Notice that
  \begin{equation*}
    F^*[\psi](\theta,\zeta,\sigma) = f \ast \Dil_{\sigma} \Upsilon(\zeta),
  \end{equation*}
  where $\Upsilon(z) := (\Mod_{-\theta} \overline{\phi})(-z)$ has Fourier transform vanishing at $0$.
  We would really like $\FT{\Upsilon}$ to vanish on a large open neighbourhood of $0$, so to achieve something like this we perform a continuous Whitney decomposition of $\FT{\Upsilon}$.

  Fix $\upsilon\in\Phi$ such that $\FT{\upsilon}$ is non-negative, symmetric, equal to $1$ on $B_{\mf{b}/3}$, and vanishing outside $B_{\mf{b}/2}$. Let
  \begin{equation*}
    W^{+}_{\sigma}(\FT{z}) := C_{\upsilon}^{-1} \int_{4\mf{b}}^{5\mf{b}}  \int_{\sigma}^{\infty} \FT{\upsilon}(\tau \FT{z}-\theta') \, \frac{\dd \tau}{\tau} \, \dd \theta',
    \qquad W^{-}_{\sigma}(\FT{z}) :=W^{+}_{\sigma}(-\FT{z})
  \end{equation*}
  where the constant
  \begin{equation*}
    C_{\upsilon}:=\int_{4\mf{b}}^{5\mf{b}} \int_{0}^{\infty} \FT{\upsilon}(\tau-\theta) \, \frac{\dd\tau}{\tau} \, \dd \theta'
  \end{equation*}
  is bounded and strictly positive, uniformly in $\upsilon$.
  For $\sigma > 0$ the functions $W^{\pm}_{\sigma}$ are smooth on $\pm(0,\infty)$, and support considerations and change of variables provide us with the properties
  \begin{equation*}
    W^{\pm}_{0}=\1_{\pm (0,+\infty)}, \quad
    W^{\pm}_{\sigma}(\FT{z})= W^{\pm}_{1}(\sigma \FT{z}), \quad
    W^{\pm}_{\sigma}(\FT{z}) = \begin{cases}
      1  & \qquad\text{if } 0< \pm \FT{z}< \,\frac{7\mf{b}}{2\sigma} \\
      0  & \qquad\text{if }  \pm \FT{z}> \,\frac{11\mf{b}}{2\sigma}. \end{cases}
  \end{equation*}
  Setting
  \begin{equation*}
    W_{\sigma}^{0}(\FT{z}):= \begin{cases}
      W^{+}_{\sigma}(\FT{z})+ W^{-}_{\sigma}(\FT{z}) & \text{if } \FT{z}\neq 0 \\
      1 & \text{if } \FT{z}= 0,\end{cases}
  \end{equation*}
  we also have that
  \begin{equation*}
    W_{1}^{0}\in C^{\infty}_{c}(B_{6\mf{b}}) \qquad \text{and} \qquad W_{\sigma}^{0}(\FT{z})=W_{1}^{0}(\sigma\FT{z}).
  \end{equation*}

  Let $\FT{\upsilon}_{\theta',\tau}(\tau\FT{z}-\theta'):=\tau \FT{\Upsilon}(\FT{z}) \FT{\upsilon}(\tau\FT{z}-\theta')$ and use the identities above to write
  \begin{equation*}
    \FT{\Upsilon}(\FT{z}) = \FT{\Upsilon}(\FT{z}) W^0_1(\FT{z}) = \int_{4\mf{b}<|\theta'|<5\mf{b}} \int_{1}^{\infty} \tau^{-1} \FT{\upsilon}_{\theta',\tau}(\tau \FT{z}-\theta') \, \frac{\dd \tau}{\tau} \, \dd \theta'.
  \end{equation*}
  We have that $\|\upsilon_{\theta',\tau}\|_{N'-1}\lesssim1$ for $\tau>1$.
  To see this, first note that $\|\FT{\Upsilon}\|_{C^{N'}}\lesssim 1$.
  Since $\FT{\Upsilon}(0)=0$ it holds that $\FT{\Upsilon}(\FT{z})=\FT{z}\FT{\tilde{\Upsilon}}(\FT{z})$ with $\|\FT{\tilde{\Upsilon}}\|_{C^{N'-1}}\lesssim\|\FT{\Upsilon}\|_{C^{N'}}\lesssim1$.
  Finally
  \begin{equation*}
    \FT{\upsilon}_{\theta',\tau}(\FT{z}) = (\FT{z}+\theta') \FT{\tilde{\Upsilon}}\big(\frac{\FT{z}+\theta'}{\tau}\big) \FT{\upsilon}(\FT{z})
  \end{equation*}
  and this shows our claim since we assumed $\tau>1$.
  By rescaling we have
  \begin{equation*} \begin{aligned}[t]
      \Dil_{\sigma}\Upsilon(z)= \int_{4\mf{b}<|\theta'|<5\mf{b}}\int_{\sigma}^{\infty} \frac{\sigma}{\tau} \Dil_{\tau}\Mod_{\theta'}\upsilon_{\theta',\tau/\sigma}(z) \, \frac{\dd \tau}{\tau} \, \dd \theta'.
    \end{aligned} \end{equation*}

  To control $f*\Dil_{\sigma}\Upsilon(\zeta)$ by $F^{*}(\theta',\zeta,\tau)$ with $(\theta',\zeta,\tau)\in\mT^{out}$ we must restrict the above representation to avoid using wavepackets $\Dil_{\tau}\Mod_{\theta'}\upsilon_{\theta',\tau/\sigma}$ with $\tau>1-|\zeta|$. Thus we refine our decomposition by writing
  \begin{equation*} \begin{aligned}[t]
      \Dil_{\sigma}\Upsilon(z)= \begin{aligned}[t]
        &\int_{4\mf{b}<|\theta'|<5\mf{b}}\int_{\sigma}^{1-|\zeta|} \frac{\sigma}{\tau} \Dil_{\tau}\Mod_{\theta'}\upsilon_{\theta',\tau/\sigma}(z) \, \frac{\dd \tau}{\tau} \, \dd \theta'
        \\ &+ \int_{B_{8\mf{b}}}\frac{\sigma}{1-|\zeta|} \Dil_{1-|\zeta|}\upsilon^{\uparrow}_{\theta',(1-|\zeta|)/\sigma}(z) \, \frac{\dd \tau}{\tau} \, \dd \theta',
      \end{aligned}
    \end{aligned} \end{equation*}
  where
  \begin{equation*} \begin{aligned}[t]
      (\upsilon^{\uparrow}_{\theta',(1-|\zeta|)/\sigma})^{\wedge} \big((1-|\zeta|)\FT{z}-\theta'\big):=\tilde{C}_{\upsilon}^{-1}\frac{1-|\zeta|}{\sigma} \FT{\Upsilon}(\sigma\FT{z}) \, W_{1-|\zeta|}^{0}(\FT{z}) \, \FT{\upsilon}\big((1-|\zeta|)\FT{z}-\theta'\big)
    \end{aligned} \end{equation*}
  with an appropriate $\tilde{C}_{\upsilon}>0$. It follows along the same lines as for $\upsilon_{\theta',\tau}$ that
  $\|\upsilon^{\uparrow}_{\theta',(1-|\zeta|)/\sigma}\|_{N'}\lesssim 1$. Let us justify the decomposition above:
  by support considerations we have that
  \begin{equation*}
    \int_{B_{8\mf{b}}} \FT{\upsilon}(\FT{z}-\theta') \, \dd \theta' =\int_{B_{8\mf{b}}} \FT{\upsilon}(-\theta') \dd \theta' =: \tilde{C}_{\upsilon} \qquad \text{on } \FT{z} \in B_{7\mf{b}}
  \end{equation*}
  and so
  \begin{equation*} \begin{aligned}[t]
      W_{1-|\zeta|}^{0}(\FT{z}) =\tilde{C}_{\upsilon}^{-1} W_{1-|\zeta|}^{0}(\FT{z})  \int_{B_{8\mf{b}}}  \FT{\upsilon}\big((1-|\zeta|)\FT{z}-\theta'\big) \, \dd \theta'.
    \end{aligned} \end{equation*}
  The proof of the decomposition claim then follows from  
  \begin{equation*}
    \Big(\Dil_{\sigma}\Upsilon(\cdot)-  \int_{4\mf{b}<|\theta'|<5\mf{b}}\int_{\sigma}^{1-|\zeta|} \frac{\sigma}{\tau} \Dil_{\tau}\Mod_{\theta'}\upsilon_{\theta',\tau/\sigma}(\cdot) \frac{\dd \tau}{\tau} \dd \theta'\Big)^{\wedge}
    =
    \FT{\Upsilon} W^{0}_{1-|\zeta|}.
  \end{equation*}

  In summary, we have obtained for all $\sigma\in[\sigma^{\downarrow},1-|\zeta|)$ that
  \begin{equation*}
    \1_{\R^{3}_{+}\setminus V^{*}}F^{*}[\psi](\theta,\zeta,\sigma) = \begin{aligned}[t] &  \int_{4\mf{b}<|\theta'|<5\mf{b}}\int_{\sigma}^{1-|\zeta|} \frac{\sigma}{\tau}     \1_{\R^{3}_{+}\setminus V^{*}}F^{*}[\upsilon_{\theta',\tau/\sigma}](\theta',\zeta,\tau) \frac{\dd \tau}{\tau} \dd \theta'
      \\ & + \int_{B_{8\mf{b}}}\frac{\sigma}{(1-|\zeta|)}      \1_{\R^{3}_{+}\setminus V^{*}}F^{*}[\upsilon^{\uparrow}_{\theta',(1-|\zeta|)/\sigma}](\theta',\zeta,1-|\zeta|)  \dd \theta' \end{aligned}
  \end{equation*}
  with $\|\upsilon^{\uparrow}_{\theta',(1-|\zeta|)/\sigma}\|_{N'-1} + \|\upsilon_{\theta',\tau/\sigma}\|_{N'-1}\lesssim 1$ uniformly in all the parameters, including $\zeta$ and $\theta$ on which these wave packets implicitly depend.
  Applying Lemma \ref{lem:uniformly-bounded-wave-packets}, we may replace the wave packets $\upsilon_{\theta',\tau/\sigma}$ and $\upsilon^{\uparrow}_{\theta',(1-|\zeta|)/\sigma}$ with a fixed (but arbitrary) $\phi \in \Phi^{N}_{1}$, independent of all these parameters.
  This leads us to
  \begin{equation*}
    \|\1_{\R^{3}_{+}\setminus V^{*}} F^{*}[\varphi] (\theta,\zeta,\sigma)\,\1_{[\sigma^{\downarrow},1)}(\sigma) \|_{\gamma_{\dd \sigma / \sigma}}^{2} \lesssim \sup_{\phi \in \Phi_{1}^{N}} \big(\mb{A}(\phi) + \mb{B}(\phi)\big),
  \end{equation*}
  where
  \begin{equation*} \begin{aligned}[t]
      \mb{A}(\phi) & :=   \int_{\Theta^{out}} \bigg\|\int_{\sigma}^{(1-|\zeta|)} \frac{\sigma}{\tau}\; \1_{\R^{3}_{+}\setminus V^{*}}  F^{*}  [\phi](\theta',\zeta,\tau) \1_{[\sigma^{\downarrow},1-|\zeta|)}(\tau) \, \frac{\dd \tau}{\tau}  \bigg\|_{\gamma_{\dd \sigma / \sigma}(\sigma^{\downarrow}, 1 - |\zeta|)}^2\dd \theta' \\
      & \lesssim \int_{\Theta_{out}} \|\1_{\R^{3}_{+}\setminus V^{*}} F^{*}[\phi] (\theta',\zeta,\tau) \|_{\gamma_{\dd \tau / \tau}(\sigma^{\downarrow}, 1)}^2 \, \dd\theta'
    \end{aligned} \end{equation*}
  by Corollary \ref{cor:gamma-convolution}, and 
  \begin{equation*} \begin{aligned}[t]
      \mb{B}(\phi) := &\int_{\Theta} \big\| \frac{\sigma}{1-|\zeta|}  \1_{\R^{3}_{+}\setminus V^{*}}  F^{*}[\phi] (\theta',\zeta,1-|\zeta|) \1_{[\sigma^{\downarrow},1-|\zeta|)}(\sigma) \big\|_{\gamma_{\dd\sigma/\sigma}}^{2}\, \dd \theta' \\
      &= \int_{\Theta}\big\|\frac{\sigma}{1-|\zeta|} \1_{[\sigma^{\downarrow},1-|\zeta|)}(\sigma) \big\|_{L^{2}_{\dd \sigma / \sigma}(\R_{+})}^2\; \big\|\1_{\R^{3}_{+}\setminus V^{*}} F^{*}[\phi] (\theta',\zeta,1-|\zeta|)\big\|_{X}^{2} \, \dd \theta'.
    \end{aligned} \end{equation*}
  Keeping \eqref{eq:V-condn} in mind, these estimates can be used to bound $\eqref{eq:full_R_domination:local-coordinates:good-scales}$ as required, with $\mb{A}(\phi)$ yielding the $\RS_{N, out}^{2}$ term and $\mb{B}(\phi)$ yielding the $\lL_N^\infty$ term.
\end{proof} 

\subsection{John--Nirenberg property for the \texorpdfstring{$\gamma$}{gamma} local size}

The proof of the local size-H\"older estimate leads us to consideration of the local sizes $\RS^s$ with integrability parameter $s \in (1,\infty)$.
In fact, when applied to embedded functions, these local sizes enjoy a John--Nirenberg-type inequality, allowing the comparison of any two integrability parameters.
More precisely, we have the following.

\begin{prop}\label{prop:john_nirenberg}
  Let $X$ be a Banach space and let $s\in[1,\infty)$. Then for all $N \in \N$,
  \begin{equation*}
    \RS^{s}_{N+2}\prec^{\Emb ,\mu}\RS^{ 2}_{N} +\lL^{\infty}_{N}.
  \end{equation*}
\end{prop}

The proof relies on an induction-on-scales argument, whose key ingredient we isolate as a Lemma.

\begin{lem}[John--Nirenberg-type induction on scales]\label{lem:john_nirenberg_induction_step}
  Let $X$ be a Banach space and $s\in[1,\infty)$, and define the local size
  \begin{equation*}
    \|F\|_{\RS^{(2,s)}_{N}(T)}=\Big(\int_{B_{1}}\Big(\int_{\Theta}\sup_{\psi\in\Psi^{N}_{1}(\theta)} \big\|\pi_{T}^{*}F[\psi](\theta,\zeta,\sigma)\big\|^{2}_{\gamma_{\dd \sigma/\sigma}(\R_{+};X)}\dd \theta\Big)^{s/2}\dd \zeta \Big)^{1/s}.
  \end{equation*}
  Then for any $\epsilon_{0}>0$ there exists a constant $C_{r,\epsilon_{0}}>1$ such that for any $T\in\TT$ and any $V\in\TT^{\cup}$ one has
  \begin{equation}
    \label{eq:john_nirenberg_induction_step}
    \begin{aligned}
      \|\1_{\R^{3}_{+}\setminus V}\Emb[f]\|_{\RS^{(2,s)}_{N+1}(T)}
      & \leq C_{r,\epsilon_{0}}\big(\|\1_{\R^{3}_{+}\setminus V}\Emb[f]\|_{\RS^{2}_{N}(T)}+\|\1_{\R^{3}_{+}\setminus V}\Emb[f]\|_{\lL^{\infty}_{N}(T)}\big)
      \\
      & +\epsilon_{0} \sup_{\substack{T'\in\TT\\ s_{T'}<\epsilon_{0}s_{T}}}\|\1_{\R^{3}_{+}\setminus V} \Emb[f]\|_{\RS^{(2,s)}_{N+1}(T')} \qquad \forall f\in\Sch(\R;X).
    \end{aligned}
  \end{equation}
\end{lem}

\begin{proof}[Proof of Proposition \ref{prop:john_nirenberg}, assuming Lemma \ref{lem:john_nirenberg_induction_step}]
  We will show that
  \begin{equation}
    \label{eq:john_nirenberg_amalg}
    \begin{aligned}
      \|\1_{\R^{3}_{+}\setminus V}\Emb[f]\|_{\RS^{(2,s)}_{N+1}(T)}
      & \lesssim_{r}\sup_{\substack{T'\in\TT\\s_{T'}\leq s_{T}}} \|\1_{\R^{3}_{+}\setminus V}\Emb[f]\|_{\RS^{2}_{N}(T)}+\|\1_{\R^{3}_{+}\setminus V}\Emb[f]\|_{\lL^{\infty}_{N}(T)} 
    \end{aligned}
  \end{equation}
  and 
  \begin{equation}
    \label{eq:john_nirenberg_unamalg}
    \begin{aligned}
      \|\1_{\R^{3}_{+}\setminus V}\Emb[f]\|_{\RS^{s}_{N+1}(T)}
      & \lesssim_{r}\|\1_{\R^{3}_{+}\setminus V}\Emb[f]\|_{\RS^{(2,s)}_{N+1}(T)}+\|\1_{\R^{3}_{+}\setminus V}\Emb[f]\|_{\lL^{\infty}_{N}(T)}
    \end{aligned}
  \end{equation}
  for all $f\in\Sch(\R;X)$

  The statement \eqref{eq:john_nirenberg_amalg} follows by an induction-on-scales argument. We fix $f\in\Sch(\R;X)$ and $V\in\TT^{\cup}$ and first suppose that $V\in\TT^{\cup}$ is such that $\R^{3}_{+}\setminus V\subset \R^{2}\times(s_{0},\infty)$ for some $s_{0}>0$. We claim that if there exists a (large) constant $\tilde{C}_{r}$ such that if
  \begin{equation}\label{eq:john_nirenberg_claim}
    \begin{aligned}
      \|\1_{\R^{3}_{+}\setminus V}\Emb[f]\|_{\RS^{(2,s)}_{N+1}(T)}
      & \leq \tilde{C}_{r}\sup_{\substack{T'\in\TT\\s_{T'}\leq s_{T}}}\Big(\|\1_{\R^{3}_{+}\setminus V}\Emb[f]\|_{\RS^{2}_{N}(T')}+\|\1_{\R^{3}_{+}\setminus V}\Emb[f]\|_{\lL^{\infty}_{N}(T')}\Big)
    \end{aligned}
  \end{equation}
  holds for all $T$ with $s_{T}<s_{*}$ for some $s_{*}>0$, then \eqref{eq:john_nirenberg_claim} also holds with \emph{the same constant $\tilde{C}_{r}$} for any $s_{T}<2s_{*}$. Since \eqref{eq:john_nirenberg_claim} trivially holds for all $T$ with $s_{T}<s_{0}$, the claim implies that \eqref{eq:john_nirenberg_claim} holds for all $T$ without restriction on $s_{T}$. The result then extends by approximation to arbitrary $V\in\TT^{\cup}$, proving \eqref{eq:john_nirenberg_amalg}.

  Now we prove the claim, using Lemma \ref{lem:john_nirenberg_induction_step}: if $T\in\TT$ with $s_{T}<2s_{*}$, then
  \begin{equation*}
    \begin{aligned}
      \|\1_{\R^{3}_{+}\setminus V}\Emb[f]\|_{\RS^{(2,s)}_{N+1}(T)}
      \leq& C_{r,\epsilon_{0}}\big(\|\1_{\R^{3}_{+}\setminus V}\Emb[f]\|_{\RS^{2}_{N}(T)}
      +
      \|\1_{\R^{3}_{+}\setminus V}\Emb[f]\|_{\lL^{\infty}_{N}(T)}\big)
      \\
      &+\epsilon_{0} \sup_{\substack{T'\in\TT\\ s_{T'}<\epsilon_{0}s_{T}}}\|\1_{\R^{3}_{+}\setminus V} \Emb[f]\|_{\RS^{(2,s)}_{N+1}(T')}
      \\
      \leq&   C_{r,\epsilon_{0}}\big(\|\1_{\R^{3}_{+}\setminus V}\Emb[f]\|_{\RS^{2}_{N}(T)}+\|\1_{\R^{3}_{+}\setminus V}\Emb[f]\|_{\lL^{\infty}_{N}(T)}\big)
      \\
      &+ \epsilon_{0}\tilde{C}_{r} \sup_{\substack{T'\in\TT\\s_{T'}\leq s_{T}}}\big( \|\1_{\R^{3}_{+}\setminus V} \Emb[f]\|_{\RS^{2}_{N}(T')} +
      \|\1_{\R^{3}_{+}\setminus V}\Emb[f]\|_{\lL^{\infty}_{N}(T')}\big).
    \end{aligned}
  \end{equation*}
  In the last step we used the induction hypothesis, which is valid as long as $\epsilon_{0}<1/2$. This implies our claim provided that $\tilde{C}_{r}$ is large enough (independently of $s_*$).
  
  Now we prove the bound \eqref{eq:john_nirenberg_unamalg}. Fix $V\in\TT^{\cup}$ and $T\in\TT$ and, by symmetry, assume that $T=T_{(0,0,1)}$. Let $F^{*}=\pi_{T}^{*}\Emb[f]$ and $V^{*}:= \pi_{T}^{-1}(V)$, so that $\pi_{T}^{*}(\1_{\R^{3}_{+}\setminus V} F) = \1_{\R^{3}_{+}\setminus V^{*}} F^{*}$.
  Define
  \begin{equation*}
    \sigma^{\downarrow}(\theta,\zeta):= 2\Big(1+\frac{4|\Theta|}{\dist(\theta;\R\setminus\Theta)}\Big)\inf\big\{\sigma>0 \colon  (\theta,\zeta,\sigma)\notin V^{*}\big\}
  \end{equation*}
  so that $(\theta',\zeta,\sigma')\notin V^{*}$ for all $\theta'\in\Theta$ and any $\sigma'>\sigma^{\downarrow}(\theta,\zeta)/2$.
  Then for $\sigma>\sigma^{\downarrow}(\theta,\zeta)$ it holds that
  \begin{equation*}
    \begin{aligned}
      F^{*}[\phi](\theta,\zeta,\sigma)
      &
      =
      \int_{\R}f(z) \, \bar{\Tr_{\zeta} \Dil_{\sigma}\Mod_{\theta} \phi(z)} \, \dd z
      \\
      &
      =
      \int_{\R}f(z) \, \bar{\Tr_{\zeta} \Dil_{\sigma/2}\Mod_{\theta'}  \Mod_{\theta/2-\theta'} \Dil_{2}\phi(z)} \, \dd z
      \\
      &
      =
      F^{*}[\Mod_{\theta/2-\theta'} \Dil_{2}\phi](\theta',\zeta,\sigma/2)
    \end{aligned}
  \end{equation*}
  with $(\theta',\zeta,\sigma/2)\notin V^{*}$. If $\phi\in\Phi$ then, as long as
  \begin{equation*}
    \theta'\in (\theta/2-\mf{b}/2,\theta/2+\mf{b}/2),
  \end{equation*}
  it holds that $\theta'\in\Theta$, so
  \begin{equation*}
    \spt(\Tr_{\theta/2-\theta'} \Dil_{1/2}\FT{\phi}) \subset B_{\mf{b}/2}(\theta/2-\theta') \subset B_{\mf{b}}.
  \end{equation*}
  Thus $\Mod_{\theta/2-\theta'} \Dil_{2}\phi \in\Psi(\theta')$ if $\phi\in\Psi(\theta)$.
  Note that we also have
  \begin{equation*}
    \|\Mod_{\theta/2-\theta'} \Dil_{2}\phi\|_{N+1}\lesssim \|\phi\|_{N+1}.
  \end{equation*}

  Now estimate the left-hand-side of \eqref{eq:john_nirenberg_unamalg} by
  \begin{equation*}
    \begin{aligned}
      &\int_{\Theta} \int_{B_{1}} \sup_{\psi\in\Psi^{N+1}_{1}(\theta)} \|\1_{\sigma<\sigma^{\downarrow}(\theta,\zeta)}\;\1_{\R^{3}_{+}\setminus V^{*}} F^{*}[\psi](\theta,\zeta,\sigma)\|_{\gamma_{\dd \sigma/\sigma}}^{s} \, \dd \zeta \, \dd \theta
      \\
      &
      \qquad+ 
      \int_{\Theta} \int_{B_{1}} \sup_{\psi\in\Psi^{N+2}_{1}(\theta)} \|\1_{\sigma^{\downarrow}(\theta,\zeta)<\sigma}\;\1_{\R^{3}_{+}\setminus V^{*}} F^{*}[\psi](\theta,\zeta,\sigma)\|_{\gamma_{\dd \sigma/\sigma}}^{s} \, \dd \zeta \, \dd \theta.
    \end{aligned}
  \end{equation*}
  By Corollary \ref{cor:W11-Haar} and the fact that
  \begin{equation*}
    \1_{\sigma<\sigma^{\downarrow}(\theta,\zeta)}\;\1_{\R^{3}_{+}\setminus V^{*}}(\theta,\zeta,\sigma)\neq 0\implies \frac{\dist(\theta;\R\setminus\Theta)}{10|\Theta|}<\frac{\sigma}{\sigma^{\downarrow}(\theta,\zeta)}<1,
  \end{equation*}
  the first summand is bounded by
  \begin{equation*}
    \|\1_{\R^{3}\setminus V}\Emb[f]\|_{\lL^{\infty}_{N}(T)}^{s} \int_{\Theta} 1+\log\big(1+\dist(\theta;\R\setminus\Theta)^{-1}\big)^{s}\, \dd \theta.
  \end{equation*}
  It remains to bound the summand with $\sigma > c(\theta) \sigma^{\downarrow}(\theta,\zeta)$.
  By the discussion above and a change of variable,
  \begin{equation*}
    \begin{aligned}
      &\int_{\Theta} \int_{B_{1}} \sup_{\psi\in\Psi^{N+1}_{1}(\theta)} \|\1_{\sigma^{\downarrow}(\theta,\zeta)<\sigma}\;\1_{\R^{3}_{+}\setminus V^{*}} F^{*}[\psi](\theta,\zeta,\sigma)\|_{\gamma_{\dd \sigma/\sigma}}^{s} \dd \zeta \dd \theta
      \\
      &
      \lesssim
      \begin{aligned}[t]
        \int_{\Theta} \int_{B_{1}} \sup_{\psi\in\Psi^{N+1}_{1}(\theta)} \Big(&\int_{\theta/2+\mf{b}/2}^{\theta/2+\mf{b}/2} \big\|\1_{\sigma^{\downarrow}(\theta,\zeta)<\sigma}\;\1_{\R^{3}_{+}\setminus V^{*}}(\theta',\zeta,\sigma/2) \\
        &  \times F^{*}[\Mod_{\theta/2C^{*}-\theta'} \Dil_{2}\psi](\theta',\zeta,\sigma/2)\big\|_{\gamma_{\dd \sigma/\sigma}}  \dd \theta'\Big)^{s}\dd \zeta  \dd \theta
      \end{aligned}
      \\
      &\lesssim \|\1_{\R^{3}_{+}\setminus V}\Emb[f]\|_{\RS^{(2,s)}_{N+1}(T)}^{s}.
    \end{aligned}
  \end{equation*}
  This concludes the proof. 
\end{proof}

\begin{proof}[Proof of Lemma \ref{lem:john_nirenberg_induction_step}]
  As in the previous proof, fix $f\in \Sch(\R;X)$ and $V\in\TT^{\cup}$, by symmetry assume that $T=T_{(0,0,1)}$, and let $F^{*}=\pi_{T}^{*}\Emb[f]$ and $V^{*}:= \pi_{T}^{-1}(V)$ so that $\pi_{T}^{*}(\1_{\R^{3}_{+}\setminus V} F) = \1_{\R^{3}_{+}\setminus V^{*}} F^{*}$. Let
  \begin{equation*}
    \mf{F}_{s}(\zeta) := \Big| \int_{\Theta}\sup_{\psi\in\Psi^{N}_{1}(\theta)}\big\|\1_{\sigma<s} \,(\1_{\R^{3}_{+}\setminus V^{*}} F^{*})[\psi] (\theta,\zeta,\sigma)\big\|_{\gamma_{\dd \sigma/\sigma}}^{2}\dd \theta \Big|^{1/2},
  \end{equation*}
  so that
  \begin{equation*}
    \|\1_{\R^{3}_{+}\setminus V}\Emb[f]\|_{\RS^{2}_{N}(T)}=\Big( \int_{B_{1}}\mf{F}_{1}(\zeta)^{2}\dd \zeta\Big)^{1/2}.
  \end{equation*}
  For a fixed small $0<\epsilon\ll\epsilon_{0}$ to be determined let
  \begin{equation*}
    \begin{aligned}
      &
      A:=\Big\{\zeta\in\R\colon M(\mf{F}_{1})(\zeta)> \epsilon^{-1} \|\1_{\R^{3}_{+}\setminus V}\Emb[f]\|_{\RS^{2}_{N}(T)}\Big\}.
    \end{aligned}
  \end{equation*}
  Since $A\subset B_{2}$ and $M\mf{F}_{1}$ is lower semi-continuous, we can represent $A$ as an at-most-countable union of disjoint open sub-intervals of $B_{2}$,
  \begin{equation*}
    A=\bigcup_{n} B_{s_{n}}(x_{n}).
  \end{equation*}
  We decompose
  \begin{equation}\label{eq:JN-3-terms}
    \|\1_{\R^{3}_{+}\setminus V}\Emb[f]\|_{\RS^{(2,s)}_{N+1}(T)}^{s}\lesssim\begin{aligned}[t]
      &\int_{B_{1}\setminus A} \mf{F}_{1}(\zeta)^{s} \, \dd \zeta
      \\ &\qquad + \sum_{n} \|\1_{B_{s_{n}}(x_{n})}(y) \1_{t<s_{n}} \1_{\R^{3}_{+}\setminus V}\Emb[f]\|_{\RS^{(2,s)}_{N+1}(T)}^{s}
      \\ &\qquad + \sum_{n} \big\| \1_{B_{s_{n}}(x_{n})}(y) \1_{t>s_{n}} \1_{\R^{3}_{+}\setminus V}\Emb[f]\big\|_{\RS^{(2,s)}_{N+1}(T)}^{s}
    \end{aligned}
  \end{equation}
  and bound the three terms separately.
  Since $M\mf{F}_{1}\geq\mf{F}_{1}$, the first term satisfies 
  \begin{equation*}
    \int_{B_{1}\setminus A} \mf{F}_{1}(\zeta)^{s} \dd \zeta\lesssim\epsilon^{-r} \|\1_{\R^{3}_{+}\setminus V}\Emb[f]\|_{\RS^{2}_{N}(T)}^{s}.
  \end{equation*}
  For the second, let $T_{n}=T_{\Theta,(0,x_{n},2s_{n})}$ and estimate
  \begin{equation*}
    \begin{aligned}
      \sum_n \|\1_{B_{s_{n}}(x_{n})}(y) \1_{t<s_{n}} \1_{\R^{3}_{+}\setminus V}\Emb[f]\|_{\RS^{(2,s)}_{N+1}(T)}^{s}
      &\leq 2 \sum_n s_{n}\|\1_{\R^{3}_{+}\setminus V}\Emb[f]\|_{\RS^{(2,s)}_{N+1}(T_{n})}^{s} \\
      &\lesssim \epsilon^{2} \sup_{\substack{T'\in\TT_{\Theta}\\s_{T'}<2\epsilon^{2}}}\|\1_{\R^{3}_{+}\setminus V}\Emb[f]\|_{\RS^{(2,s)}_{N+1}(T')}^{s}
    \end{aligned}
  \end{equation*}
  using $L^2$-boundedness of the maximal operator.

  It remains to control the third term of \eqref{eq:JN-3-terms}.
  We exploit the lower scale bound to control the oscillation of $\1_{\R^{3}_{+}\setminus V^{*}}F^{*}$ on $B_{s_{n}}(x_{n})$.
  Let
  \begin{equation*}
    \sigma^{\downarrow}(\theta,n):= 4 \inf\{\sigma>s_{n} \colon \exists\zeta\in B_{s_{n}}(x_{n}) \text{ such that } (\theta,\zeta,\sigma)\notin V^{*}\}.
  \end{equation*}
  We split each of the summands into the regimes $1/4<\sigma/\sigma^{\downarrow}(\zeta,n)<1$ and $\sigma^{\downarrow}(\zeta,n)<\sigma<1$.
  For the scales close to $\sigma^{\downarrow}(\theta,n)$ we estimate for each $n$, using Corollary \ref{cor:W11-Haar},
  \begin{equation*}  \begin{aligned}[t]
      & \big\|\1_{1/4<\sigma/\sigma^{\downarrow}(\zeta,n)<1} \,(\1_{\R^{3}_{+}\setminus V^{*}} F^{*})[\psi]\big\|_{\gamma_{\dd \sigma/\sigma}}
      \\ & \lesssim \big\|\1_{1/4<\sigma/\sigma^{\downarrow}(\zeta,n)<1} \,(\1_{\R^{3}_{+}\setminus V^{*}} F^{*})[\psi]\big\|_{L^\infty} \\
      & \qquad + \big\|\1_{1/4<\sigma/\sigma^{\downarrow}(\zeta,n)<1} \,(\1_{\R^{3}_{+}\setminus V^{*}} \sigma \partial_{\sigma}F^{*})[\psi] \big\|_{L^\infty},
    \end{aligned} \end{equation*}
  leading to a bound by $\|\1_{\R^{3}_{+}\setminus V} F\|_{\lL^{\infty}_{N}(T)}$ after integrating and summing in $n$, using disjointness of the sets $B_{s_n}(x_n)$.
  We still need to deal with the scales  $\sigma^{\downarrow}(\zeta,n)<\sigma<1$.
  Note that (omitting the variable $\theta$ for clarity)
  \begin{equation*} \begin{aligned}[t]
      F^{*}[\psi](\zeta,\sigma) &= \fint_{B_{s_{n}}(x_{n})} F^{*}[\psi](z,\sigma) \, \dd z + \fint_{B_{s_{n}}(x_{n})} \int_{z}^{\zeta}\partial_{y}F^{*}[\psi](y,\sigma)  \, \dd y \, \dd z
      \\ & = \fint_{B_{s_{n}}(x_{n})} \Big( F^{*}[\psi](z,\sigma) + \Big(\frac{z-x_{n}}{s_{n}}+\sgn(\zeta-z)\Big)\,\frac{s_{n}}{\sigma}\, F^{*}[\tilde{\psi}](z,\sigma) \Big) \, \dd z
    \end{aligned} \end{equation*}
  where $\tilde{\psi}=(2\pi i \theta-\partial)\psi\in\Psi_1(\theta)$.
  For all $\sigma>\sigma^{\downarrow}(\theta,n)$ one has 
  \begin{equation*}
    (\theta,z,\sigma) \notin V^{*} \qquad \forall z\in B_{s_{n}}(x_{n})
  \end{equation*}
  by an argument similar to the one in the proof of Lemma \ref{lem:overlap-size-dom} (see Figure \ref{fig:tree-sobolev}).
  The cases where $\sigma^{\downarrow}(\zeta,n)=\infty$ are of no importance.
  It follows that for all $\sigma>\sigma^{\downarrow}(\theta,n)$ (again omitting $\theta$ for clarity)
  \begin{equation*} \begin{aligned}[t]
      &\1_{\R^{3}_{+}\setminus V^{*}}F^{*}[\psi](\zeta,\sigma)
      \\ & = \fint_{B_{s_{n}}(x_{n})} \Big( \1_{\R^{3}_{+}\setminus V^{*}} F^{*}[\psi](z,\sigma) \\
      &\hspace{3cm}+ \Big(\frac{z-x_{n}}{s_{n}}+\sgn(\zeta-z)\Big) \frac{s_{n}}{\sigma} \1_{\R^{3}_{+}\setminus V^{*}}F^{*}[\tilde{\psi}](z,\sigma) \Big)\dd z.
    \end{aligned} \end{equation*}
  It follows that for any $\zeta\in B_{s_{n}}(x_{n})$,
  \begin{equation*} \begin{aligned}[t]
      &\Big(\int_{\Theta}\sup_{\psi\in\Psi^{N+1}_{1}(\theta)} \big\|\1_{\sigma^{\downarrow}(\zeta,n)<\sigma} \1_{\R^{3}_{+}\setminus V^{*}} F^{*}[\psi] (\theta,\zeta,\sigma)\big\|_{\gamma_{\dd \sigma/\sigma}}^{2} \, \dd \theta\Big)^{1/2}
      \\ & \lesssim \Big(\int_{\Theta} \sup_{\psi\in\Psi^{N}_{1}(\theta)}\Big(\fint_{B_{s_{n}}(x_{n})} \big\|\1_{\sigma^{\downarrow}(\zeta,n)<\sigma}  \1_{\R^{3}_{+}\setminus V^{*}} F^{*} [\psi] (\theta,z,\sigma)\big\|_{\gamma_{\dd \sigma/\sigma}}\dd z\Big)^{2}\dd \theta\Big)^{1/2}
      \\ & \lesssim \fint_{B_{s_{n}}(x_{n})}\mf{F}_{1}(z) \, \dd z\lesssim \epsilon^{-1}\|\1_{\R^{3}_{+}\setminus V}\Emb[f]\|_{\RS^{2}_{N}(T)}
    \end{aligned} \end{equation*}
  where the last bound follows by the construction of the sets $B_{s_n}(x_n)$. Summing in $n$ and integrating in $\zeta$ bounds the third term of \eqref{eq:JN-3-terms}, concluding the proof. 
\end{proof}

\subsection{A summary of the size domination results}\label{sec:size-domination-summary}

The results of the previous sections can be summarised concisely by the diagram

\begin{equation*}
  \xymatrix{
    \FS^{s} \ar[r] \ar[d] \ar[dr] & \RS^{s} \ar[r] \ar[d] & \RS^{2} \ar[d] \ar[dl] \\
    \DS \ar[r] & \lL^\infty \ar[r] & \RS^{2}_{out}
  }
\end{equation*}
in which each depicted local size is dominated, in the sense of Definition \ref{def:size-domination-embedded}, by the sum of the local sizes to which it points (possibly with a decrease in the order of regularity $N$, which is not shown in the diagram).
Following the arrows shows that, in the end, we have
\begin{equation*}
  \FS^{s} \prec^{\Emb, \mu} \RS^{2}_{out}.
\end{equation*}
Proposition \ref{prop:embedding-domination} then applies and yields the result below, which tells us that for our purposes it suffices to prove embedding bounds with respect to the local size $\RS_{out}^{2}$. 

\begin{thm}\label{thm:size-dom-master}
  Let $X$ be a Banach space with finite cotype.
  Then for any $p,q \in (0,\infty]$, $s \in (1,\infty)$, and $N \in \N$, there exists $N' > N$ such that
  \begin{equation*}
    \|\Emb[f]\|_{L^p_\mu \FS_{N'}^{s}} \lesssim \|\Emb[f]\|_{L^p_\mu \RS_{N,out}^2}
    \quad
    \text{and}
    \quad
    \|\Emb[f]\|_{L^p_\nu \sL^q_\mu \FS_{N'}^{s}} \lesssim \|\Emb[f]\|_{L^p_\nu \sL^q_\mu \RS_{N,out}^2} 
  \end{equation*}
  for all $f \in \Sch(\R;X)$.
\end{thm}


%% file: main/embeddings-noniter.tex
In this section we fix a parameter $\mf{b} > 0$, a bounded interval $\Theta$ containing $B_{1}\supsetneq B_{2\mf{b}}$, and a large $N > 1$.
All estimates implicitly depend on these choices, and we will not refer to $\Theta$ or $N$ in the notation.

\subsection{Statement of the bounds, and the \texorpdfstring{$L^\infty$}{L\^infty} endpoint}

To prove Theorem \ref{thm:intro-BHT}, we need to prove embedding bounds of the form \eqref{eq:emb-bounds}.
By Theorem \ref{thm:size-dom-master}, it suffices to prove these with respect to the local size $\RS_{out}^2$.
In fact, the estimates
\begin{equation}\label{eqn:non-iterated-bd}
  \| \Emb[f] \|_{L^p_\mu \RS_{out}^{2}} \lesssim \|f\|_{L^p(\R;X)} \qquad \forall f \in \Sch(\R;X)
\end{equation}
into non-iterated outer Lebesgue spaces would suffice, but we can only prove this for $p > r$ when $X$ is $r$-intermediate UMD.
This is good enough to prove Theorem \ref{thm:intro-BHT} with the additional restriction that $p_i > r_i$ for each $i$, not for the full range.
With additional work (done in Section \ref{sec:embeddings-iter}) the estimates \eqref{eqn:non-iterated-bd} for $p > r$ will be `localised' to prove the full range of estimates \eqref{eq:emb-bounds} that we need.

The main result of this section is the following `non-iterated' embedding theorem.

\begin{thm}\label{thm:non-iter-embedding}
  Fix $r \in [2,\infty)$ and suppose $X$ is $r$-intermediate UMD.
  Then for all $p \in (r,\infty]$, the bound \eqref{eqn:non-iterated-bd} holds,
  and thus by Theorem \ref{thm:size-dom-master} we also have
  \begin{equation*}
    \big\| \Emb[f]  \big\|_{L^{p}_{\mu} \FS^{s}} \lesssim \|f\|_{L^p(\RR;X)} \qquad \forall f \in \Sch(\R;X)
  \end{equation*}
  for all $s \in (1,\infty)$.
\end{thm}

By Marcinkiewicz interpolation for outer Lebesgue spaces, Theorem \ref{thm:non-iter-embedding} follows from the following endpoint bounds.

\begin{prop}\label{prop:non-iter-Linfty}
  Suppose $X$ is UMD. Then
  \begin{equation*}
    \| \Emb[f] \|_{L^\infty_\mu \RS_{out}^{2}}\lesssim \|f\|_{L^\infty(\R;X)} \qquad \forall f \in \Sch(\R;X).
  \end{equation*}
\end{prop}

\begin{prop}\label{prop:non-iter-weak}
  Fix $r \in [2,\infty)$ and suppose $X$ is $r$-intermediate UMD.
  Then for all $p \in (r,\infty]$, 
  \begin{equation*}
    \| \Emb[f] \|_{L^{p,\infty}_\mu \RS_{out}^{2} } \lesssim \|f\|_{L^p(\R;X)} \qquad \forall f \in \Sch(\R;X).
  \end{equation*}
\end{prop}

Theorem \ref{thm:non-iter-embedding} is essentially already known in the case $r = 2$, i.e. when $X$ is isomorphic to a Hilbert space.
We state this as a separate theorem.

\begin{thm}\label{thm:hilbert-embeddings}
  If $H$ is a Hilbert space,
  then for all $p \in (2,\infty]$, 
  \begin{equation*}
    \|\Emb[f]\|_{L^p_\mu \RS_{out}^{2}} \lesssim \|f\|_{L^p(\R;H)} \qquad \forall f \in \Sch(\R;H).
  \end{equation*}
\end{thm}

\begin{proof}
  Taking into account minor notational differences, the bounds
  \begin{equation}\label{eq:Do-Thiele-embedding}
    \|\Emb[f][\phi]\|_{L^{p}_{\mu}\lL^2_{out}} \lesssim \|f\|_{L^p(\R;H)}\qquad p\in(2,\infty]
  \end{equation}
  were established in \cite{DT15} for all $\phi\in\Phi^{N}_{1}$ with $N$ sufficiently large in the one-dimensional case $H = \CC$.
  The same proof, which only relies on orthogonality arguments, applies to general Hilbert spaces.
  The result follows using Proposition \ref{prop:sup-wavepackets} and that $\RS_{out}^{2}$ is equivalent to $\lL^2_{out}$  for Hilbert spaces.
\end{proof}

Let us return to Theorem \ref{thm:non-iter-embedding}.
First we prove the $L^\infty$ endpoint, which amounts to estimates on a single tree.

\begin{proof}[Proof of Proposition \ref{prop:non-iter-Linfty}]
  By Proposition \ref{prop:sup-wavepackets} it is sufficient to prove
  \begin{equation*}
    \| \Emb[f][\phi] \|_{L^\infty_\mu \RS_{out}^{2}(T)}\lesssim \|f\|_{L^\infty(\R;X)} 
  \end{equation*}
  for any $\phi\in\Phi_{1}$, $f \in \Sch(\R;X)$ and $T \in \TT$. Let us fix such $\phi$, $f$ and $T$,
  and by symmetry we may assume that $T = T_{(0,0,1)}$ (see Remark \ref{rmk:invariances}).
  Decompose $f$ into a local part and a tail part, $f = f_{loc} + f_{tail}$, where $f_{loc} = f\1_{B_{2}}$.
  For the local part we have
  \begin{equation*}
    \|\Emb[f_{loc}][\phi]\|_{\RS^{2}_{out}(T)} =  \Big( \int_{\Theta \sm B_{2\mf{b}}} \int_\R \|\pi_T^* \Emb[f_{loc}][\varphi](\theta,\zeta,t)\|_{\gamma_{\dd t / t} (\R_+;X)}^2 \, \dd\zeta \, \dd\theta \Big)^{1/2}.
  \end{equation*}
  For $\theta \in \Theta \sm B_{2\mf{b}}$ we need to estimate
  \begin{align*}
    &\int_\R \|\pi_T^* \Emb[f_{loc}][\varphi](\theta,\zeta,t)\|_{\gamma_{\dd t / t}(\R_+;X)}^2 \, \dd\zeta \\
    &\qquad = \int_\R \|  \1_{\overline{\mT}}(\theta,\zeta,t)  \langle f_{loc}; \Tr_\zeta \Dil_{t} \psi \rangle \|_{\gamma_{\dd t / t}(\R_+;X)}^2 \, \dd\zeta,
  \end{align*}
  where $\psi = \Mod_{-\theta} \varphi$.
  Since $\theta \notin B_{2\mf{b}}$ and $\varphi$ has Fourier support in $B_{\mf{b}}$, $\psi$ has Fourier transform vanishing in an open neighbourhood of the origin.
  Thus by Theorem \ref{thm:WP-LP}, and using that $f_{loc}$ is supported on $B_{2}$, we have
  \begin{equation}\label{eq:single-tree-near-lac}
    \|\Emb[f_{loc}]\|_{\RS^{2}(T)} \lesssim \| f_{loc}\|_{L^2(\R;X)} \lesssim \|f\|_{L^\infty(\R;X)}.
  \end{equation}

  For the tail part, having fixed $\theta$ and $\varphi$ as before, we are faced with estimating
  \begin{equation*}
    \int_\R \big\|  \1_{\overline{\mT}}(\theta,\zeta,t)  \langle f_{tail}; \Tr_\zeta \Dil_{t} \psi \rangle \big\|_{\gamma_{\dd t/t}}^2 \, \dd\zeta 
    \lesssim \int_{B_1} \big\|   \langle f_{tail}; \Tr_\zeta \Dil_{t} \psi \rangle \big\|_{\gamma_{\dd t/t}(0,1;X)}^2 \, \dd\zeta.
  \end{equation*}
  By Lemma \ref{lem:gamma-tail-estimate} we have
  \begin{equation}\label{eq:farpart-goal}
    \big\| \langle f_{tail}; \Tr_\zeta \Dil_{t} \psi \rangle \big\|_{\gamma_{\dd t / t}(0,1;X)} \lesssim \|f\|_{L^\infty(\R;X)}
  \end{equation}
  for each $\zeta \in B_1$, and integrating in $\zeta$ and $\theta$ completes the proof.
\end{proof}

The remainder of the section is devoted to proving the weak endpoint, Proposition \ref{prop:non-iter-weak}. 

\subsection{Tree orthogonality for intermediate spaces}

In time-frequency analysis one has the following heuristic: if $\phi \in \Sch(\R)$, and we consider a sequence of points $(\eta_i,y_i,t_i) \in \R^3_+$ that are sufficiently separated, then the wave packets $\Lambda_{(\eta_i,y_i,t_i)}\phi$ are essentially orthogonal.
If $f \in \Sch(\R;H)$ takes values in a Hilbert space, this essential orthogonality can be exploited to control weighted $\ell^2$-sums of the coefficients $\langle f, \Lambda_{(\eta_i,y_i,t_i)} \phi \rangle$ by $\|f\|_{L^2(\R;H)}$.
More generally, one can control $\ell^2$-sums of square functions over disjoint regions $E_i$ of trees $T_i$.
This is one of the main techniques used in the proof of Theorem \ref{thm:hilbert-embeddings}.

This orthogonality is lost when working with general Banach spaces.
However, by working with intermediate UMD spaces $X = [Y,H]_\theta$, we can use some orthogonality from $H$ to strengthen the UMD-derived estimates on $Y$ (which hold only on a single tree), resulting in the following theorem.

\begin{thm}\label{thm:lac-type}
  Let $X$ be $r$-intermediate UMD for some fixed $r \in [2,\infty)$ and let $\mc{T}$ be a finite collection of trees with distinguished subsets $E_{T}\subset T$ for each $T \in \mc{T}$.
  Then for all $p > r$ and all $f \in \Sch(\R;X)$,
  \begin{equation}\label{eq:tree-orthogonality}
    \Big( \sum_{T \in \mc{T}} \| \1_{E_T} \Emb[f]\|_{\RS^{2}_{out}(T) }^2 s_T \Big)^{1/2} \lesssim \|f\|_{L^p(\RR;X)} \Big(\sum_{T \in \mc{T}} s_T \Big)^{\frac{1}{2} - \frac{1}{p}} \Big\| \sum_{T \in \mc{T}} \1_{E_T} \Big\|_{L_{\mu}^\infty(\lL^\infty + \lL^{1}_{in})}^{1/p}.
  \end{equation}
\end{thm}

\begin{proof}
  By Proposition \ref{prop:sup-wavepackets} it is sufficient to show 
  \begin{equation*}
    \begin{aligned}
      &\Big( \sum_{T \in \mc{T}} \| \1_{E_T} \Emb[f][\phi]\|_{\RS^{2}_{out}(T) }^2 s_T \Big)^{1/2} \\
      &\qquad \lesssim \|f\|_{L^p(\RR;X)} \Big(\sum_{T \in \mc{T}} s_T \Big)^{\frac{1}{2} - \frac{1}{p}} \Big\| \sum_{T \in \mc{T}} \1_{E_T} \Big\|_{L_{\mu}^\infty(\lL^\infty + \lL^{1}_{in})}^{1/p}
    \end{aligned}
  \end{equation*}
  for any fixed $\phi\in\Phi^{N}_{1}$.
  First fix a UMD space $Y$ and a tree $T$.
  By Proposition \ref{prop:non-iter-Linfty} we have
  \begin{equation*}
    \|\Emb[f]\|_{\RS^{2}_{out}(T)} \lesssim \|f\|_{L^\infty(\RR;Y)}
  \end{equation*}
  and thus
  \begin{equation*}
    \Big( \sum_{T \in \mc{T}} \| \1_{E_T} \Emb[f][\phi]\|_{\RS^{2}_{out}(T) }^2 s_T \Big)^{1/2} \lesssim \|f\|_{L^\infty(\RR;Y)} \Big(\sum_{T \in \mc{T}} s_T \Big)^{1/2}.
  \end{equation*}
  
  Now fix a Hilbert space $H$.
  By the equivalence of $\gamma$- and $L^2$-norms for Hilbert spaces (Proposition \ref{prop:gamma-l2-hilb}), for each tree $T$ we have
  \begin{equation*}
    \| \1_{E_T} \Emb[f][\phi] \|_{\RS^{2}(T) } \simeq  \big\| \1_{T^{out} \cap E_T} \Emb[f][\phi]\big\|_{L^2(\R_+^3 ; H)}.
  \end{equation*}

  By Proposition \ref{prop:lebesgue-size-holder}, using that $p > 2$,
  \begin{align*}
    &\int_{\R^3_+} \Big( \sum_{T \in \mc{T}} \1_{E_T}(\eta,y,t) \Big) \|\Emb[f][\phi] (\eta,y,t)\|_H^2 \, \dd \eta \, \dd y \, \dd t \\
    &\lesssim \big\| \|\Emb[f][\phi]\|_H \big\|_{L^p_\mu (\lL_{out}^{2} + \lL^{\infty})}^2 \Big\| \sum_{T \in \mc{T}} \1_{E_T} \Big\|_{L^{\frac{p}{p-2}}_\mu( \lL^{\infty} + \lL^{1}_{in})} \\
    &\lesssim \big\| \|\Emb[f][\phi] \|_H \big\|_{L^p_\mu( \lL_{out}^{2} + \lL^{\infty})}^2 \mu(\mc{T})^{1 - \frac{2}{p}} \Big\| \sum_{T \in \mc{T}} \1_{E_T} \Big\|_{L^{\infty}_\mu( \lL^{\infty} + \lL^{1}_{in})} \\
    &\lesssim \big\| \Emb[f][\phi]\big\|_{L^p_\mu \RS^{2}_{out}}^2
      \Big( \sum_{T \in \mc{T}} s_T \Big)^{1 - \frac{2}{p}} \Big\| \sum_{T \in \mc{T}} \1_{E_T} \Big\|_{L^{\infty}_\mu( \lL^{\infty} + \lL^{1}_{in})}.
  \end{align*}
  Applying Theorem \ref{thm:hilbert-embeddings} yields
  \begin{equation}\label{eq:tile-orth-HS}
    \sum_{T \in \mc{T}} \| \1_{E_T} \Emb[f][\phi]\|_{\RS^{2}_{out}(T) }^2 s_T \lesssim \|f\|_{L^p(\R;H)}^2  \Big( \sum_{T \in \mc{T}} s_T \Big)^{1 - \frac{2}{p}} \Big\| \sum_{T \in \mc{T}} \1_{E_T} \Big\|_{L^{\infty}_\mu( \lL^{\infty} + \lL^{1}_{in})},
  \end{equation}
  giving the desired result for Hilbert spaces.

  Now consider $X = [Y,H]_{2/r}$, where $Y$ is UMD and $H$ is a Hilbert space.
  Let $\mathring{L}^\infty(\R;Y)$ denote the closure of the Schwartz functions $\Sch(\R;Y)$ in $L^\infty(\R;Y)$.
  Then using the relations
  \begin{equation*}
    [\ell^2(\mc{T};L^2(\R^2;\gamma(\R_+;Y))), \ell^2(\mc{T};L^2(\R^2;\gamma(\R_+;H)))]_{2/r} \cong \ell^2(\mc{T};L^2(\R^2;\gamma(\R_+;X)))
  \end{equation*}
  (see \cite[Theorem 2.2.6]{HNVW16} and \cite[Theorem 9.1.25]{HNVW17}) and
  \begin{equation}
    [\mathring{L}^\infty(\R;Y), L^q(\R;H)]_{2/r} \cong L^{qr/2}(\R; X)
  \end{equation}
  for all $q > 2$ (see \cite[\textsection 1.18.4, Remark 3]{hT78}), we interpolate between the above estimates to get the desired result for all $p > r$.
\end{proof}

\subsubsection{Technical remarks on Theorem \ref{thm:lac-type}}

The estimate \eqref{thm:lac-type} and its proof are inspired by the notion of Fourier tile-type introduced by Hyt\"onen and Lacey in \cite{HL13}, but there are some fundamental differences.
Using the notation of \cite[Definition 5.1]{HL13}, a Banach space $X$ is said to have \emph{Fourier tile-type $q$} if for every $\alpha \in (0,1)$,
\begin{equation}\label{eq:fourier-tile-type}
  \begin{aligned}
    &\Big( \sum_{\mathbb{T} \in \mathscr{T}} \Big\| \sum_{P \in \mathbb{T}} \langle f ; \phi_P \rangle \phi_P \Big\|_{L^q(\R;X)}^q \Big)^{1/q} \\
    &\qquad \lesssim_{\alpha} \|f\|_{L^q(\R;X)} + \Big( \|f\|_{L^\infty(\R;X)} \Big[ \sum_{\mathbb{T} \in \mathscr{T}} |I_{\mathbb{T}}| \Big]^{1/q} \Big)^{1-\alpha} \|f\|_{L^q(\R;X)}^\alpha
  \end{aligned}
\end{equation}
whenever $f \in L^q(\R;X) \cap L^\infty(\R;X)$ and $\mathscr{T}$ is a finite disjoint collection of finite trees with a certain disjointness property.
They work on a discrete model space of tiles rather than $\R^3_+$, and their trees are subsets of tiles (which correspond to subsets of our trees).
The functions $\phi_P$ here are $L^2$-normalised wave packets.

The first obvious difference between \eqref{eq:fourier-tile-type} and \eqref{thm:lac-type} is that ours is an $L^p \to L^2$ estimate, while \eqref{eq:fourier-tile-type} is a range of $L^q \cap L^\infty \to L^q$ estimates with an auxiliary parameter $\alpha$.
Hyt\"onen and Lacey work with sub-indicator functions $|f| \leq \1_E$ where $E \subset \R$ is a bounded measurable set, so the space $L^q \cap L^\infty$ is natural.
However, our estimates must only be in terms of $\|f\|_{L^p(\R:X)}$, so \eqref{eq:fourier-tile-type} does not directly help us.
A second difference is the form of the left-hand sides of the two estimates.
The functions being measured in \eqref{eq:fourier-tile-type} are the `tree projections'
\begin{equation*}
  \sum_{P \in \TT} \langle f ; \phi_P \rangle \phi_P,
\end{equation*}
which can be thought of as derived from $\Emb[f]$.
On the other hand, \eqref{eq:tree-orthogonality} measures $\Emb[f]$ itself.

Another difference is the method of proof.
Like Hyt\"onen and Lacey, we argue by interpolation, based on an estimate for Hilbert spaces.
Their fundamental Hilbert space estimate \cite[Proposition 6.1]{HL13} takes the form
\begin{equation*}
  \begin{aligned}
    &\big( \sum_{P \in \mathbb{P}} |\langle f; \phi_P \rangle|^2 \big)^{1/2} \\
    &\qquad \lesssim \|f\|_{L^2(\R;H)} + \bigg( \sup_{P \in \mathbb{P}} \frac{|\langle f, \phi_P \rangle|}{|I_P|^{1/2}} \big[ \sum_{\mathbb{T} \in \mathscr{T}} |I_{\mathbb{T}}| \big]^{1/2} \bigg)^{1/3} \|f\|_{L^2(\R;H)}^{2/3};
  \end{aligned}
\end{equation*}
its proof is based on the orthogonality heuristic mentioned before Theorem \ref{thm:lac-type}.
The wave packet coefficients $\langle f, \phi_P \rangle$ can be controlled since $f$ is a sub-indicator function, but we do not have this luxury.
Instead, our Hilbert space estimate \eqref{eq:tile-orth-HS} is an $L^p \to L^2$ estimate that follows from the embedding bounds of Theorem \ref{thm:hilbert-embeddings}.
This theorem is proven by Marcinkiewicz interpolation for outer Lebesgue spaces; the weak endpoint is proven by a tree selection argument depending on the function $f$ being analysed, and one then uses orthogonality arguments to prove the desired bounds.
For the weak endpoint in the Banach case (Proposition \ref{prop:non-iter-weak}) we also use a tree selection argument depending on $f$, but in place of orthogonality arguments we use \eqref{eq:tree-orthogonality}.
To prove \eqref{eq:tree-orthogonality} by complex interpolation, we need to see it as boundedness of a linear operator.
While such an operator is allowed to depend on the data $\mc{T}$ and $(E_T)_{T \in \mc{T}}$, but to be well-defined and linear, it can't depend on $f$.
By proving the fundamental Hilbert space estimate \eqref{eq:tile-orth-HS} as a consequence of Theorem \ref{thm:hilbert-embeddings} we manage to embed an $f$-dependent tree selection argument where it shouldn't be allowed.
This illustrates the strength of Theorem \ref{thm:hilbert-embeddings}.

The conclusion of Theorem \ref{thm:lac-type} could be considered as a Banach space property that all $r$-intermediate UMD spaces have, as is done by Hyt\"onen and Lacey with Fourier tile-type, but we will not go so far as to make such a definition here.

\subsection{Tree selection and the weak endpoint}

The proof of the weak endpoint estimate (Proposition \ref{prop:non-iter-weak}) uses a tree selection argument, which we place separately as a lemma.

\begin{lem}\label{lem:tree-selection}
  Let $X$ be a Banach space with finite cotype, $A\subset \R^{3}_{+}$ a compact set, and suppose $F\in L^{\infty}_{\mu}\RS^{2}_{out}$ is supported in $A$.
  Then for all $\lambda>0$ there exists a finite collection of trees $\mc{T}$ with distinguished subsets $E_T\subset T\cap A$ and a set $V\in\TT^{\cup}$, satisfying the properties 
  \begin{equation}\label{eq:tree-sel-conditions}
    \begin{aligned}
      &\| \1_{\R^{3}_{+} \sm V}F\|_{\RS^{2}_{out}} \lesssim \lambda,
      \\
      & V\supset \bigcup_{T\in\mc{T}} T \qquad \mu(V)\lesssim \mu\Big(\bigcup_{T\in\mc{T}}T\Big),
      \\
      &\| F \1_{E_T} \|_{\RS^{2}_{out}(T)} \gtrsim \lambda \qquad \forall T \in \mc{T},
      \\
      &\Big\|\sum_{T \in \mc{T}} \1_{E_{T}}  \Big\|_{L^{\infty}_{\mu} ( \lL^{\infty} + \lL^1_{in} )} \lesssim 1.
    \end{aligned}
  \end{equation}
\end{lem}

\begin{proof}
  First we replace $\RS^{2}_{out}$ with an equivalent monotonic size: define
  \begin{equation*}
    \|F\|_{\tilde{\RS}^{2}_{out}(T)} := \sup_{V \in \TT^\cup} \|\1_{\R^3_+ \sm V} F\|_{\RS_{out}^{2}(T)}.
  \end{equation*}
  Then by the $\gamma$-multiplier theorem (Theorem \ref{thm:gamma-multiplier}, which requires finite cotype) $\tilde{\RS}^{2}_{out}$ is equivalent to $\RS_{out}^{2}$, while by definition we also have the monotonicity property
  \begin{equation*}
    \|\1_{V_{1}} F\|_{\tilde{\RS}^{2}_{out}(T)} \leq \|\1_{V_{2}} F\|_{\tilde{\RS}^{2}_{out}(T)}
  \end{equation*}
  for all $V_{1}\subset V_{2} \subset \R^{3}_{+}$ and all $T \in \TT$.
  In the rest of the proof we abuse notation and write $\RS^{2}_{out}$ in place of $\tilde{\RS}^{2}_{out}$.

  Define local sizes $\RS^{2}_{out\pm}$ by
  \begin{equation*}
    \| F \|_{\RS^{2}_{out\pm}(T)} := \| \1_{T^{\pm}} F \|_{\RS^{2}_{out}(T)}
  \end{equation*}
  where $T^{\pm}=\pi_{T}(\mT_{\Theta}\cap \{\pm \theta\geq0\})$. It suffices to prove the lemma with $\RS^{2}_{out+}$ in place of $\RS^{2}_{out}$.
  A symmetric proof handles $\RS^{2}_{out-}$, and we can simply add the results together.

  By homogeneity we assume that $\|F\|_{L^{\infty}_{\mu}\RS^{2}_{out}}\leq1$. Without loss of generality we assume that $A=\bar{B_{S}}^{2}\times [S^{-1};S]$ for some $S\gg 1$. 
  Define
  \begin{align*}
    \R^{3,latt}_{+} &:= \big\{ (\eta,y,t) \colon  \eta t, yt^{-1}\in\Z, t \in2^{\Z}   \big\} \\
    \TT^{latt}&:=\big\{ T_{(\eta,y,t),\Theta} \colon (\eta,y,t)\in \R^{3,latt}_{+}\big\}.
  \end{align*}
  and given any $T\in\TT$ let
  \begin{equation*}
    \begin{aligned}
      \mc{V}(T) &:= \Big\{T_{(\xi,x,s)}\in\TT^{latt} \colon  2^{3}s_{T}\leq s<2^{5}s_{T},\; |x-x_{T}|\leq s_{T},\; |\xi-\xi_{T}|<2^{10}s_{T}^{-1}\Big\}, \\
      V(T) &:= \bigcup_{T^{*}\in\mc{V}(T)}T^{*}\\
    \end{aligned}
  \end{equation*}
  so that $\mu(V(T))\lesssim \mu(T)$ and $\bigcup_{\epsilon\in\{-1,0,1\}}T_{(\xi_{T}+\epsilon s_{T}^{-1},x_{T},s_{T})}\subset  V(T)$.
  Define $\mc{A}^{latt,\lambda}$ to be the family of trees  in $\TT^{latt}$ given by 
  \begin{equation*}
    \mc{A}^{latt,\lambda}:=\bigcup\Big\{\mc{V}(T)\colon T\in\TT, S^{-1}s_{T}<2S\lambda^{-2},\; |x_{T}|<3S\lambda^{-2},\; |\xi_{T}|<4S^{3}\Big\}.
  \end{equation*}
  The collection $\mc{A}^{latt,\lambda}$ is finite, since the tops of trees in $\mc{A}^{latt,\lambda}$ are a compact subset of the discrete set $\R^{3,latt}_{+}$. For any $T\in\TT$ it holds that
  \begin{equation*}
    \begin{aligned}[t]
      \|F\|_{\RS^{2}_{out}(T)} \leq \Big(\frac{S+s_{T}}{s_{T}}\Big)^{1/2}\|F\|_{\RS_{out}^{2}(T_{(\xi_{T},0,2S+s_{T})})}
      \leq \Big(\frac{2S}{s_{T}}\Big)^{1/2}\|F\|_{\RS^{2}_{out}(T_{(\xi_{T},0,2S)})},
    \end{aligned}
  \end{equation*}
  so for any $\lambda\in(0,1)$ we have $\|F\|_{\RS^{2}_{out}(T)}>\lambda$ only if $\mc{V}(T)\subset \mc{A}^{latt,\lambda}$.

  Given a collection of trees $\mc{X} \subset  \TT$ such that $s_{T}<2^{8}S$ for all $T\in\mc{X}$ and $\{(\xi,x,s)\in\R^{3}_{+}\colon T_{(\xi,x,s)}\in\mc{X}\}$ is pre-compact, we say that $T\in\mc{X}$ is \emph{quasi-maximal} if
  \begin{equation*}
    T'\in\mc{X}\implies \xi_{T'}\leq \xi_{T}+\frac{1}{2^{10}S}.
  \end{equation*}
  Any non-empty collection $\mc{X}$ with the properties above has at least one quasi-maximal element.

  We proceed to choose trees iteratively.
  Start with the collection $\mc{T}_{0}=\emptyset$ and let $A_{0}:=A$.
  Suppose that at step $j$ we have a collection $\mc{T}_j \subset \TT$ and a subset $A_j \subset A$.
  At step $j+1$ let
  \begin{equation*}
    \mc{X}_{j+1} := \{ T\in\TT : \| \1_{A_{j}} F \|_{\RS^{2}_{out+}(T)}>\lambda \}.
  \end{equation*}
  If $\mc{X}_{j+1}$ is empty, terminate the iteration, otherwise choose a maximal element $T_{j+1}$ of $\mc{X}_{j+1}$.
  Let $\mc{T}_{j+1} = \mc{T} \cup \{T_{j+1} \}$, let $E_{j+1}:= T_{j+1}^{out+}\cap A_{j}$,
  and let 
  \begin{equation*}
    A_{j+1}:=A_{j}\setminus V(T_{j}).
  \end{equation*}
  This process terminates after finitely many steps because $T\in\mc{X}_{j+1}$ only if
  \begin{equation*}
    \mc{V}(T)\cap \big(\mc{A}^{latt,\lambda}\setminus \bigcup_{j'=0}^{j}\mc{V}(T_{j'})\big) \neq \emptyset,
  \end{equation*}
  so $\mc{A}^{latt,\lambda}\setminus \bigcup_{j=0}^{j}\mc{V}(T_{j'})$ is a strictly decreasing subcollection of the finite collection $\mc{A}^{latt,\lambda}$.
  Thus this procedure yields a collection of trees $\mc{T}=\bigcup_{j}\mc{T}_{j}$ and pairwise disjoint distinguished subsets $E_T \subset T \in \mc{T}$.
  
  The first three conditions of \eqref{eq:tree-sel-conditions} hold by construction.
  Pairwise disjointness of the sets $(E_T)_{T \in \TT}$ guarantees that $\big\| \sum_{T \in \mc{T}} \1_{E_{T}} \big\|_{L^{\infty}_{\mu}\lL^{\infty}} \leq 1$.
  It remains to show that
  \begin{equation}\label{eq:treesel-last-goal}
    \begin{aligned}
      &\Big\| \sum_{T \in \mc{T}} \1_{E_{T}} \Big\|_{L^{\infty}_{\mu} \lL^1_{in}(T')} 
      = \int_{B_{1}} \int_{B_{2\mf{b}}} \int_{0}^{1} \Big| \pi_{T'}^{*} \big(
      \sum_{T \in \mc{T}} \1_{E_{T}} \big) \Big| (\theta,\zeta,\sigma) \, \frac{\dd \sigma}{\sigma} \, \dd \theta \, \dd \zeta
      \lesssim 1
    \end{aligned}
  \end{equation}
  for any tree $T' \in\TT$.
  Without loss of generality we may suppose $T'=T_{(0,0,1)}$. Fix $\zeta\in B_{1}$ and $\theta\in B_{2\mf{b}}$ and suppose that the integrand above doesn't vanish identically in $\sigma$. Let
  \begin{equation*}
    \begin{aligned}
      &\tau_{+}(\theta,\zeta) := \sup\Big\{\sigma \colon (\theta,\zeta,\sigma)\in\pi^{-1}_{T'}\big( \bigcup_{T \in \mc{T}} E_{T}
      \big)\Big\}
      \\
      & \tau_{-}(\theta,\zeta) := \inf\Big\{\sigma \colon (\theta,\zeta,\sigma)\in\pi^{-1}_{T'}\big( \bigcup_{T \in \mc{T}}E_{T}
      \big)\Big\}.
    \end{aligned}
  \end{equation*}
  Recall that we are implicitly working with respect to a frequency band $\Theta$, which we write as $\Theta=(\theta^{*}_{-},\theta^{*}_{+})$.
  We claim that
  \begin{equation}\label{eq:treesel-contra-claim}
    \tau_{-}(\theta,\zeta) \geq \frac{1}{4}\frac{2\mf{b}-\theta}{\theta^{*}_{+}-\theta}\tau_{+}(\theta,\zeta).
  \end{equation}
  This would allow us to conclude that 
  \begin{align*}
    &   \int_{B_{1}} \int_{B_{2\mf{b}}} \int_{0}^{1} \Big| \pi_{T'}^{*} \big( \sum_{T \in \mc{T}} \1_{E_{j}} \big) \Big| (\theta,\zeta,\sigma) \, \frac{\dd \sigma}{\sigma} \, \dd \theta \, \dd \zeta \lesssim \int_{B_{1}} \int_{ B_{2\mf{b}}} \int_{\tau_{-}(\theta,\zeta)}^{\tau_{+}(\theta,\zeta)} \, \frac{\dd \sigma}{\sigma} \, \dd \theta \, \dd \zeta
    \\
    & \qquad
      \lesssim
      \int_{B_{1}} \int_{ B_{2\mf{b}}}  \log (2\mf{b}-\theta)- \log(4(\theta^{*}_{+}-\theta))  \dd \theta \, \dd \zeta \lesssim 1,
  \end{align*}
  which would prove \eqref{eq:treesel-last-goal}.
  To prove the claimed lower bound \eqref{eq:treesel-contra-claim}, argue by contradiction and suppose
  \begin{equation*}
    \tau_{-}(\theta,\zeta) < \frac{1}{4}\frac{2\mf{b}-\theta}{\theta^{*}_{+}-\theta}\tau_{+}(\theta,\zeta)<\frac{2\mf{b}-\theta}{\theta^{*}_{+}-\theta}\Big(\frac{1}{\tau_{+}(\theta,\zeta)} + \frac{1}{S}\Big)^{-1}
  \end{equation*}
  so that we can choose a $\sigma$ such that
  \begin{equation}\label{eq:contradiction-goal}
    (\theta,\zeta,\sigma) \in \pi_{T'}^{-1} \big( \bigcup_{T \in \mc{T}} E_T \big),
  \end{equation}
  with $\tau_{-}(\theta,\zeta)\leq\sigma<\frac{2\mf{b}-\theta}{\theta^{*}_{+}-\theta}\Big(\frac{1}{\tau_{+}(\theta,\zeta)} + \frac{1}{S}\Big)^{-1}$.
  Fix a tree $T_{0} \in \mc{T}$ such that $E_{T_{0}}$ intersects arbitrarily small neighbourhoods of $(\theta,\zeta,\tau_{+}(\theta,\zeta))$ (see Figure \ref{fig:dabbing-tree}). Then
  $-\frac{\theta^{*}_{+}-\theta}{\tau_{+}(\theta,\zeta)}\leq \xi_{T_0}$ and
  \begin{equation*}
    (\theta,\zeta)\times (0,1) \cap \pi^{-1}_{T'}(T_{0}^{out+}) \subset (\theta,\zeta) \times \Big(\frac{2\mf{b}-\theta}{\theta^{*}_{+}-\theta}\tau_{+}(\theta,\zeta),\tau_{+}(\theta,\zeta)\Big),
  \end{equation*}
  and, in particular, $(\theta,\zeta,\sigma) \notin \pi_{T'}^{-1}(T_{0}^{out+})$.
  Since $(\theta,\zeta,\sigma)\in\pi_{T'}^{-1}(T_{0})$ we have $(\theta,\zeta,\sigma)\notin \pi_{T'}^{-1}(E_{T_{\alpha}})$ for any $T_{\alpha} \in \mc{T}$ selected after $T_0$ in the selection procedure.
  On the other hand suppose that $T_{\beta} \in \mc{T}$ was selected before $T_0$  and 
  $(\theta,\zeta,\sigma)\in\pi_{T'}^{-1}(E_{T_\beta})\subset \pi_{T'}^{-1}(T_{\beta}^{out+})$. It would then hold that
  \begin{equation*}
    \begin{aligned}
      \xi_{T_\beta}
      \leq
      -\frac{2\mf{b}-\theta}{\sigma}
      <
      \xi_{T_{0}}-\Big(\frac{2\mf{b}-\theta}{\sigma}-\frac{\theta^{*}_{+}-\theta}{\tau_{+}(\theta,\zeta)}\Big)
      <
      \xi_{T_{0}}-\frac{\theta^{*}_{+}-\theta}{S}<\xi_{T_{0}}-\frac{1}{2S}
    \end{aligned}
  \end{equation*}
  by definition of $\sigma$, which contradicts the maximality condition on the construction since $T_{\beta}$ was selected before $T_0$: by monotonicity of $\RS_{out}$ (recall that we redefined this size to force this monotonicity) $T_{0}$ could have been selected earlier, but it was not.
  Thus we cannot have \eqref{eq:contradiction-goal}, since the point $(\theta,\zeta,\sigma)$ can belong neither to $E_{T_{0}}$ nor to any $E_{T}$ for any $T\in \mc{T}$ selected before or after $T_{0}$. The proof is complete.
\end{proof}

\begin{figure}[h]
  \includegraphics{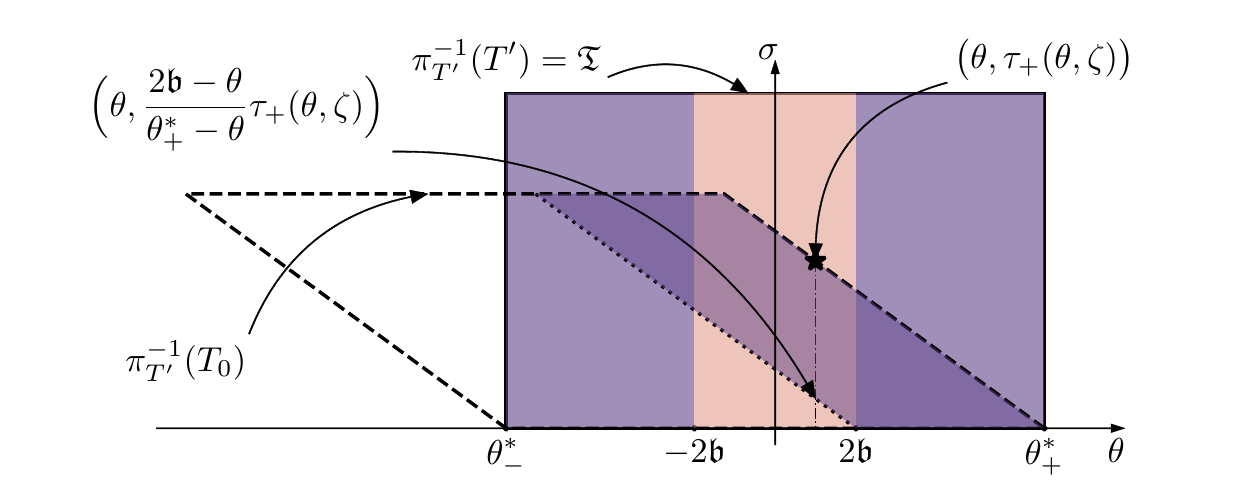}
  \caption{
    The trees $T'$ and $T_0$ appearing in the proof of Lemma \ref{lem:tree-selection}.
  }
  \label{fig:dabbing-tree}
\end{figure}

We are now prepared to prove the claimed weak endpoint.  

\begin{proof}[Proof of Proposition \ref{prop:non-iter-weak}]
  We need to show for all $\lambda > 0$ that there exists a set $V_{\lambda}\subset \R^{3}_{+}$ such that
  \begin{equation}\label{eq:wk-endpoint-goal-2}
    \begin{aligned}
      &
      \lambda^{p} \mu(V_{\lambda}) \lesssim \|f\|_{L^{p}(\R;X)}^{p}
      \\
      &
      \|\1_{\R^3_+\sm V_{\lambda}} \Emb[f]\|_{\RS^{2}_{out}} \lesssim \lambda.
    \end{aligned}
  \end{equation}
  Fix $\lambda > 0$.
  Since $f\in\Sch(\R;X)$ there exists a compact set $\mathbb{K} \subset \R^{3}_{+}$ such that 
  \begin{equation}\label{eq:G-approx}
    \| \1_{\mathbb{K}}\Emb[f]\|_{\RS^{2}_{out}} < \lambda.
  \end{equation}

  Consider the set $V_{\lambda}\in\TT^{\cup}$  and the collection $\mc{T}_{\lambda}$ of trees together with the distinguished subsets $E_{T}\subset T$ for each $T\in\mc{T}_{\lambda}$ given by Lemma \ref{lem:tree-selection} applied to $\1_{\mbb{K}} \Emb[f]$.
  We have that
  \begin{equation*}
    \|\1_{\mathbb{K} \sm V_{\lambda}} \Emb[f]\|_{\RS^{2}_{out}} \lesssim \lambda,
  \end{equation*}
  and combining this with \eqref{eq:G-approx} yields the second condition of \eqref{eq:wk-endpoint-goal-2}. 
  To prove the first condition, according to Lemma \ref{lem:tree-selection}, we reduce to showing $\lambda^{p} \sum_{T\in\mc{T}_{\lambda}}s_{T}\lesssim \|f\|_{L^{p}(\R;X)}^{p}$. Write 
  \begin{align*}
    \lambda^{2}  \sum_{T\in\mc{T}} s_ {T}
    &\lesssim   \sum_{T\in\mc{T}}\|\1_{E_{T}} \Emb[f]\|_{\RS^{2}_{out}} ^{2} s_ {T} \\
    &\lesssim \| f \|_{L^{p}(\R;X)}^{2} \big(\sum_{T\in\mc{T}_{\lambda}} s_{T} \big)^{1 -\frac{2}{p}} \big\| \sum_{T\in\mc{T}} \1_{E_{T}} \big\|_{L^{\infty}_{\mu}( \lL^{\infty} + \lL^{1}_{in})}^{1/p}
  \end{align*}
  using Theorem \ref{thm:lac-type} (valid as $X$ is $r$-intermediate UMD and $p > r$). Since we have that
  $\big\| \sum_{T\in\mc{T}} \1_{E_{T}} \big\|_{L^{\infty}_{\mu}( \lL^{\infty} + \lL^{1}_{in})}  \lesssim 1$
  it follows that
  \begin{equation*}
    \mu(A_{\lambda})\lesssim \sum_{T\in\mc{T}}s_{T}\lesssim \lambda^{-p}\| f \|_{L^{p}(\R;X)}^{p}
  \end{equation*}
  as required.

\end{proof}


%% file: main/embeddings-iter.tex
As in the previous section, we fix a parameter $\mf{b} > 0$, a bounded interval $\Theta$ containing $B_{1}\supsetneq B_{2\mf{b}}$, and some sufficiently large $N > 1$. Everything below depends on these choices, and we do not mention them in the notation.

In this section we prove the main technical result of the article, which eventually leads to our bounds on the bilinear Hilbert transform.

\begin{thm}\label{thm:RS-iterated-embeddings}
  Fix $r \in [2,\infty)$ and let $X$ be an $r$-intermediate UMD space.
  Then for all $p \in (1,\infty)$ and $q \in (\min(p,r)^\prime(r-1),\infty]$ we have 
  \begin{equation*}
    \big\| \Emb[f]  \big\|_{L_\nu^p \sL_\mu^q \RS_{out}^{2}} \lesssim \|f\|_{L^p(\RR;X)} \qquad \forall f \in \Sch(\R;X),
  \end{equation*}
  and thus by Theorem \ref{thm:size-dom-master} we also have
  \begin{equation*}
    \big\| \Emb[f]  \big\|_{L_\nu^p \sL_\mu^q \FS^s} \lesssim \|f\|_{L^p(\RR;X)} \qquad \forall f \in \Sch(\R;X)
  \end{equation*}
  for all $s \in (1,\infty)$.
\end{thm}

As explained at the start of Section \ref{sec:embeddings-noniter}, we obtain embedding bounds into iterated outer Lebesgue spaces for a larger range of $p \in (1,\infty)$ than we get for the non-iterated spaces.
Theorem \ref{thm:RS-iterated-embeddings} is proven by localisation to strips: we use Theorem \ref{thm:non-iter-embedding} to provide refined information on the quasinorms $\|\1_{D \sm W} \Emb[f] \|_{L_\mu^q \RS_{out}^{2}}$ when $D$ is a strip and $W$ is a countable union of strips.
This takes a considerable amount of technical work, particularly in the tail estimates.
A similar argument was implicitly used to prove \cite[Theorem 5.3]{AU20-Walsh}, the discrete version of this result, but in that result the tails are not present and life is simpler.

\subsection{Localisation lemmas}
We have three localisation lemmas, corresponding to the three `endpoints' needed in the proof of Theorem \ref{thm:RS-iterated-embeddings}.

\begin{lem}[First localisation lemma]\label{lem:localised-single-tree}
  Let $X$ be a UMD space.
  Then for all  $W\in\DD^{\cup}$, $T\in\TT$, and $M \in \NN$, we have
  \begin{equation}\label{eq:K-s-endpoint}
    \big\| \1_{\RR^3_+ \setminus W} \Emb[f]  \big\|_{\RS_{out}^{2}(T)} \lesssim_M \hspace{-0.7em}\sup_{(\eta,y,t)\in T_{(\xi_T, x_T, 2s_T)}\setminus W} \int_{\R} \| f(z) \|_{X} \,t^{-1} \Big\langle\frac{z-y}{t} \Big\rangle^{-M}  \dd z
  \end{equation}
  for all $f \in \Sch(\R;X)$.
\end{lem}

\begin{lem}[Second localisation lemma]\label{lem:localised-Lq-tail-lemma}
  Let $X$ be an $r$-intermediate UMD space, and suppose $q \in (r,\infty)$.
  Then for all $D \in \DD$,  $W\in\DD^{\cup}$, and all sufficiently large $M \in \NN$, we have
  \begin{equation}\label{eq:localised-Lq-tail-lemma} \begin{aligned}[t]
      &\big\| \1_{D \setminus W} \Emb[f]  \big\|_{L^{q}_{\mu} \RS_{out}^{2}} \lesssim_{M,q} |D|^{\frac{1}{q}}\hspace{-0.7em} \sup_{(\eta,y,t)\in D \setminus W} \hspace{-0.2em}\Big( \int_{\R} \| f(z) \|_{X}^{q} \; t^{-1} \Big\langle\frac{z-y}{t} \Big\rangle^{-M} \dd z \Big)^{\frac{1}{q}}
    \end{aligned} \end{equation}
  for all $f \in \Sch(\R;X)$.
\end{lem}

\begin{lem}[Third localisation lemma]\label{lem:localised-Lpq-in-out-lemma}
  Let $X$ be an $r$-intermediate UMD space, and suppose $p \in (1,r]$ and $q \in (p'(r-1), \infty)$.
  Then for all $D \in \DD$, $W\in\DD^{\cup}$, and all sufficiently large $M \in \NN$,
  \begin{equation}\label{eq:localised-Lpq-in-out-lemma} \begin{aligned}[t]
      & \big\| \1_{D \setminus W} \Emb[f]  \big\|_{L^{q}_{\mu} \RS_{out}^{2}} \lesssim_{M,q} |D|^{\frac{1}{q}}\hspace{-0.7em} \sup_{(\eta,y,t)\in D_{(x_D,2s_D)} \setminus W} \hspace{-0.2em} \Big( \int_{\R} \| f(z) \|_{X}^{p} \; t^{-1} \Big\langle\frac{z-y}{t} \Big\rangle^{-M} \dd z \Big)^{\frac{1}{p}}
    \end{aligned}  \end{equation}
  for all $f \in \Sch(\R;X)$.
\end{lem}

\begin{proof}[Proof of Lemma \ref{lem:localised-single-tree}]
  Without loss of generality assume $M$ is large.
  By homogeneity we may assume that
  \begin{equation}\label{eq:LL1-hom}
    \sup_{(\eta,y,t)\in T(\xi_T,x_T,2s_T)\setminus W} \int_{\R} \| f(z) \|_{X} t^{-1} \Big\langle\frac{z-y}{t} \Big\rangle^{-M} \dd z =1,
  \end{equation}
  and without loss of generality we may assume that $T=T_{(0,0,1)}$ (see Remark \ref{rmk:invariances}).
  It then suffices to show that
  \begin{equation*}
    \big\| \1_{\RR^3_+ \setminus W} \Emb[f]  \big\|_{\RS_{out}^{2}(T)}
    \lesssim 1.
  \end{equation*}
  
  Split $f$ into a local part and a tail part, $f = f_{loc} + f_{tail}$ with $f_{loc} = \1_{B_{2}} f$, so that by Lemma \ref{lem:gamma-tail-estimate} we have
  \begin{equation*}
    \big\| \Emb[f_{tail}]  \big\|_{\RS^{2}_{out}(T)} \lesssim \int_\RR \|f(z)\|_X \langle z\rangle^{-M} \, \dd z \lesssim 1.
  \end{equation*}
  It remains to control the contribution of the local part.

  We may assume that $\int_{B_{2}}\| f(z) \|_{X}\dd z\lesssim 1$, as if this does not hold, then by \eqref{eq:LL1-hom} we have $T\subset W$ and there is nothing to prove.
  Fix $C > 0$ and write
  \begin{equation*}
    \big\{ x\in B_2  \colon  Mf(x)>C \big\} = \bigcup_{n} B_{s_{n}}(x_{n})
  \end{equation*}
  as an at-most-countable disjoint union of open balls, where $M$ is the Hardy--Littlewood maximal operator.
  For $C$ sufficiently large we have $D_{(x_{n},3s_{n})}\subset W$ for each $n$, for otherwise \eqref{eq:LL1-hom} is contradicted.
  Fix such a large $C$.
  By disjointness of the balls, we have $\sum_{n}s_{n}\lesssim 1$.

  We perform a Calder\'on--Zygmund decomposition of $f_{loc}$: for each $n$ define
  \begin{equation*}
    b_n := \1_{B_{s_{n}}(x_{n})} \Big( f_{loc}  -  \frac{1}{|B_{s_{n}}(x_{n})|} \int_{B_{s_{n}}(x_{n})}   f_{loc}(z) \; \dd z \Big),
  \end{equation*}
  so that
  \begin{align*}
    & \spt b_{n} \subset B_{s_{n}}(x_{n}),
      \qquad
      \int_{B_{s_{n}}(x_{n})} b_{n}(z) \, \dd  z =0,
      \qquad  \| b_{n} \|_{L^{1}(\R)} \lesssim s_{n},
  \end{align*}
  and write
  \begin{equation*}
    f_{loc} =: \sum_n b_n + g.
  \end{equation*}
  The `good part' $g$ then satisfies $\|g\|_{L^\infty(\R;X)} \lesssim 1$, so by Proposition \ref{prop:non-iter-Linfty} we have
  \begin{equation*}
    \big\| \Emb[g ]  \big\|_{\RS^{2}_{out}(T)} \lesssim 1,
  \end{equation*}
  and we need only show
  \begin{equation}\label{eq:LL1-bad-pt}
    \big\| \1_{\RR^3_+ \setminus W} \Emb[b]  \big\|_{\RS^{2}_{out}(T)} \lesssim 1
  \end{equation}
  where $b := \sum_n b_n$ is the `bad part'.
  An application of Lemma \ref{lem:uniformly-bounded-wave-packets} implies
  \begin{equation*}
    \begin{aligned}
      &\big\| \1_{\RR^3_+ \setminus W} \Emb[b]  \big\|_{\RS^{2}_{out}(T)}^{2} \\
      &\qquad \lesssim \sup_{\phi \in \Phi_1}  \int_{\Theta \sm B_{2\mf{b}}} \int_{B_1} \big\|\pi_{T}^{*} \bigl( \1_{\R^{3}_{+}\setminus W}  \Emb[b] \bigr)[\phi](\theta,\zeta,\sigma) \big\|_{\gamma_{\dd \sigma / \sigma} ( \R_+ ;X )}^2 \, \dd \zeta \, \dd \theta.
    \end{aligned}
  \end{equation*}
  Fix $\phi \in \Phi_1$, $\theta \in \Theta \sm B_{2\mf{b}}$, and $\zeta \in B_1$.
  By Corollary \ref{cor:W11-Haar} we estimate the $\gamma$-norm by an integral in scale: 
  \begin{align*}
    & 
      \big\|\pi_{T}^{*} \bigl( \1_{\R^{3}_{+}\setminus W}  \Emb[b] \bigr)[\phi](\theta,\zeta,\sigma) \big\|_{\gamma_{\dd \sigma / \sigma} (\R_{+} ; X )}
    \\
    & \qquad
      \lesssim
      \int_{\tau_{-}(\zeta)}^{\tau_{+}(\zeta)}\| \mc{F}(\theta,\zeta,\sigma) \|_{X} + \|\sigma\partial_{\sigma}\mc{F}(\theta,\zeta,\sigma)\|_{X} \frac{\dd \sigma}{\sigma},
  \end{align*}
  where $\tau_{-}(\zeta)=\inf\big\{\tau\in\R_{+}\colon (\theta,\zeta,\sigma)\in\R^{3}_{+}\setminus\pi_{T}^{-1}(W)\big\}$ and  $\tau_{+}(\zeta)=1-|\zeta|$, and $\mc{F}(\theta,\zeta,t) = \sum_{n} \mc{F}_{n}(\theta,\zeta,t)$ with $\mc{F}_n(\theta,\zeta,\sigma) = \langle b_n; \Tr_{\zeta} \Dil_{\sigma} \Mod_{\theta} \phi \rangle$.
  Use integration by parts to write
  \begin{align*}
    \mc{F}_{n}(\theta,\zeta,\sigma)
    &
      = \int_{\R} \bigg(\int_{x_n - s_n}^z  b_n(z') \, \dd z' \bigg)
      \frac{\dd}{\dd z} \bar{\Tr_{\zeta}  \Dil_{\sigma} \Mod_{\theta} \phi (z)}\, \dd z
    \\
    &
      = \sigma^{-1}  \int_{\R} \bigg(\int_{x_n - s_n}^z  b_n(z') \, \dd z' \bigg) \bar{ \Tr_{\zeta}  \Dil_{\sigma} \Mod_{\theta}(2\pi i\theta\phi+\phi')(z)} \, \dd z.
  \end{align*}
  Since $b_{n}$ has mean zero the inner integral vanishes outside $B_{s_{n}}(x_{n})$ and thus
  \begin{align*}
    \| \mc{F}_{n}(\theta,\zeta,\sigma) \|_{X} &\lesssim  \int_{\R} \sigma^{-2} s_{n} \| b_{n} \|_{L^{1}} \1_{B_{s_{n}}(x_{n})}(z) \Big\langle\frac{z-\zeta}{\sigma}\Big\rangle^{-5}  \, \dd z 
    \\
                                              &
                                                \lesssim   \int_{\R} \sigma^{-2} s_{n}^{2} \1_{B_{s_{n}}(x_{n})}(z) \Big\langle\frac{z-\zeta}{\sigma}\Big\rangle^{-5} \dd z.
  \end{align*}
  A similar estimate holds for the derivative in scale:
  \begin{align*}
    &\| \sigma \partial_{\sigma} \mc{F}_{n}(\theta,\zeta,\sigma) \|_{X}
    \\
    &
      \qquad
      =
      \Big\|  \int_{\R} \bigg(\int_{x_n - s_n}^z  b_n(z') \, \dd z' \bigg) \sigma\frac{\dd }{\dd \sigma}\bar{\Tr_{\zeta}  \Dil_{\sigma} \Mod_{\theta} (2\pi i\theta\phi+ \phi')}(z)\Big) \, \dd z
      \Big\|_{X}
    \\
    &
      \qquad
      \lesssim
      \int_{\R} \sigma^{-2} s_{n}^{2} \1_{B_{s_{n}}(x_{n})}(z) \Big\langle\frac{z-\zeta}{\sigma}\Big\rangle^{-5}  \, \dd z. 
  \end{align*}
  Combining these estimates, we obtain
  \begin{equation}\label{eq:gamma-est-with-int}
    \begin{aligned}
      &  \big\|\pi_{T}^{*} \bigl( \1_{\R^{3}_{+}\setminus W}  \Emb[b] \bigr)[\phi](\theta,\zeta,\sigma) \big\|_{\gamma_{\dd \sigma / \sigma} ( \R_{+};X )}
      \\
      &
      \qquad
      \lesssim
      \int_{\tau_{-}(\zeta)}^{\tau_{+}(\zeta)}\int_{\R}\sum_{n}\sigma^{-2} s_{n}^{2} \1_{B_{s_{n}}(x_{n})}(z) \Big\langle\frac{z-\zeta}{\sigma}\Big\rangle^{-5} \dd z \frac{\dd \sigma}{\sigma}.
    \end{aligned}
  \end{equation}
  For any fixed $\sigma\in\big(\tau_{-}(\zeta),\tau_{+}(\zeta)\big)$, in view of the pairwise disjointness of  the balls $B_{s_{n}}(x_{n})$, we have the estimate
  \begin{align*}
    &\int_{\R} \sum_{n : s_{n}\leq\sigma}\sigma^{-2} s_{n}^{2} \1_{B_{s_{n}}(x_{n})}(z) \Big\langle\frac{z-\zeta}{\sigma}\Big\rangle^{-5} \dd z
    \\
    &
      \qquad
      \leq \Big\| \sum_{n : s_{n}\leq\sigma}\sigma^{-2} s_{n}^{2} \1_{B_{s_{n}}(x_{n})} \Big\|_{L^{\infty}(\R)} \Big\| \Big\langle\frac{z-\zeta}{\sigma}\Big\rangle^{-5}  \Big\|_{L^{1}_{\dd z}(\R)} \lesssim \sigma.
  \end{align*}
  If $\sigma<s_{n}$ notice that $\tau_{-}(\zeta)<\sigma$, so by construction $|z-\zeta|>s_{n}$ whenever $z \in B_{s_n}(x_n)$. Thus
  \begin{equation*}
    \Big\langle\frac{z-\zeta}{\sigma}\Big\rangle^{-5} \lesssim     \Big\langle\frac{z-\zeta}{\sigma}\Big\rangle^{-2} \Big(\frac{z-\zeta}{\sigma}  \Big)^{-3}  \lesssim   \Big(\frac{\sigma}{s_{n}}  \Big)^{3}
  \end{equation*}
  and
  \begin{align*}
    &\int_{\R} \sum_{n : s_{n}>\sigma}\sigma^{-2} s_{n}^{2} \1_{B_{s_{n}}(x_{n})}(z) \Big\langle\frac{z-\zeta}{\sigma}\Big\rangle^{-5} \dd z
    \\
    &
      \qquad
      \leq \Big\| \sum_{n : s_{n}>\sigma}\sigma^{1} s_{n}^{-1} \1_{B_{s_{n}}(x_{n})} \Big\|_{L^{\infty}(\R) } \Big\| \Big\langle\frac{z-\zeta}{\sigma}\Big\rangle^{-2}  \Big\|_{L^{1}_{\dd z}(\R)} \lesssim \sigma.
  \end{align*}
  Summing up these estimates and plugging the result back into \eqref{eq:gamma-est-with-int} gives us  
  \begin{equation*}
    \big\|\pi_{T}^{*} \bigl( \1_{\R^{3}_{+}\setminus W}  \Emb[b] \bigr)[\phi](\theta,\zeta,\sigma) \big\|_{\gamma_{\dd \sigma / \sigma} ( \R_{+} ; X )} 
    \lesssim  \int_{\tau_{-}(\zeta)}^{\tau_{+}(\zeta)} \sigma \, \frac{\dd \sigma}{\sigma} \lesssim 1.
  \end{equation*}
  Integrating the square over $\zeta \in B_1$ and $\theta \in \Theta \sm B_{2\mf{b}}$ yields \eqref{eq:LL1-bad-pt}, completing the proof.
\end{proof}

\begin{proof}[Proof of Lemma \ref{lem:localised-Lq-tail-lemma}]

  As in the previous proof, we are free to assume that $M$ is huge, by homogeneity we assume
  \begin{equation*}
    \sup_{(\eta,y,t)\in D\setminus W} \int_{\R} \| f(z) \|_{X}^{q}\; t^{-1} \Big\langle\frac{z-y}{t} \Big\rangle^{-M} \dd z =1,
  \end{equation*}
  and without loss of generality we assume $D=D_{(0,1)}$ (see Remark \ref{rmk:invariances}).

  Decompose $f$ into dyadic annuli, i.e. write $f = \sum_{k=0}^\infty f_k$, where $f_0 = \1_{B_{2}} f$ and $f_k = \1_{B_{2^{k+1}} \setminus B_{2^{k}}} f$ for all $k \geq 1$.
  By the first localisation lemma (Lemma \ref{lem:localised-single-tree}), for each $k$ we have 
  \begin{align*}
    &\| \1_{D \sm W} \Emb[f_k] \|_{L_\mu^\infty \RS_{out}^2}
    \\
    &
      \qquad
      \lesssim \sup_{T\subset D} \sup_{(\eta,y,t) \in T \sm W} \Big( \int_\R \|f_k(z)\|_X^{q} \; t^{-1} \Big\langle \frac{z-y}{t} \Big\rangle^{-2M} \, \dd z \Big)^{\frac{1}{q}}
      \lesssim   2^{-k\frac{M}{q}}
  \end{align*}
  where we used that $\langle (z-y)/t \rangle \gtrsim 2^{k}$ on the support of $f_{k}$.
  On the other hand, for any  $r<r_0 < q$, by Theorem \ref{thm:non-iter-embedding} (here we use that $X$ is $r$-intermediate UMD) we have
  \begin{align*}
    \| \1_{D \sm W} \Emb[f_k] \|_{L^{r_0}_{\mu} \RS_{out}^{2}}
    \lesssim \|f_k\|_{L^{r_0}(\R;X)}.
  \end{align*}
  Assuming $D \not\subset W$ (for otherwise there is nothing to prove) it holds that
  \begin{equation*}
    \|f_k\|_{\sL^{r_0}(\R ;X)} \leq \| f_k \|_{\sL^{q}(B_{2^{k+1}};X)} \lesssim  1,
  \end{equation*}
  so
  \begin{equation*}
    \| \1_{D \sm K_\lambda} \Emb[f_k] \|_{L_\mu^{r_0} \RS_{out}^{2}} \lesssim 2^{\frac{k}{r_0}}.
  \end{equation*}
  By logarithmic convexity of the outer Lebesgue quasinorms we find
  \begin{align*}
    \| \1_{D \sm K_\lambda} \Emb[f_k] \|_{L_\mu^{q} \RS_{out}^{2}}
    \lesssim  2^{\frac{k}{q}} 2^{-kM(1-\frac{r_{0}}{q})}  
    = 2^{ \frac{k}{r_{+}} (  1- M(q-r_{0}) )}.
  \end{align*}
  By quasi-subadditivity we have for some $C\geq 1$ that 
  \begin{align*}
    \big\| \1_{D \sm K_\lambda} \Emb[f] \big\|_{L_{\mu}^{q} \RS_{out}^{2}}
    &\lesssim \sum_{k=0}^\infty C^k \big\| \1_{D \sm K_\lambda} \Emb[f_k] \big\|_{L_{\mu}^{q} \RS^{2}_{out}} \\
    &\lesssim  \sum_{k=0}^\infty C^k 2^{ \frac{k}{q} ( 1- M(q-r_{0}) )}\lesssim 1
  \end{align*}
  provided $M$ is taken to be sufficiently large.
\end{proof}

\begin{proof}[Proof of Lemma \ref{lem:localised-Lpq-in-out-lemma}]
  We will assume that $\spt f\subset B_{2^{L}}$ for some $L>1$ and show
  \begin{equation}\label{eq:localised-Lq-tail-lemma-Lgrowth}
    \begin{aligned}
      &\big\| \1_{D \setminus W} \Emb[f]  \big\|_{L^{q}_{\mu}\RS_{out}^{2}} \\
      & \qquad
      \lesssim_M C_{tail}^{L} \sup_{(\eta,y,t)\in D(x_T,2s_T) \setminus W} \Big( \int_{\R} \| f(z) \|_{X}^{p} \; t^{-1} \Big\langle\frac{z-y}{t} \Big\rangle^{-M} \dd z \Big)^{\frac{1}{p}}
    \end{aligned}
  \end{equation}
  for some constant $C_{tail}>1$. This will be enough to complete the proof; to see this, write $f= \sum_{j=-1}^\infty f_j$ with
  \begin{align*}
    f_{-1} &:= f\1_{B_1} \\
    f_j &:= f \1_{B_{2^{j+1}}\setminus B_{2^{j}}} \qquad (j \in \N),
  \end{align*}
  and observe that for any $q_1>q$ we have for some $C_{tr} \geq 1$,
  \begin{align*}
    &\big\| \1_{D \setminus W} \Emb[f]  \big\|_{L^{q_1}_{\mu}\RS_{out}^{2}}
    \\
    &\lesssim
      \sum_{j=-1}^{\infty} C_{tr}^{j}
      \big\| \1_{D \setminus W} \Emb[f_{j}]  \big\|_{L^{q}_{\mu}\RS_{out}^{2}}^{\frac{q}{q_1}}
      \big\| \1_{D \setminus W} \Emb[f_{j}]  \big\|_{L^{\infty}_{\mu}\RS_{out}^{2}}^{1-\frac{q}{q_1}} 
    \\
    &
      \begin{aligned}
        \lesssim\sup_{(\eta,y,t)\in D(x_D,2s_D) \setminus W} \Big( \int_{\R} \| f(z) \|_{X}^{p} \; &t^{-1} \Big\langle\frac{z-y}{t} \Big\rangle^{-M} \dd z \Big)^{\frac{1}{p} \frac{q}{q_1} }
        \\
        &\sum_{j=-1}^{\infty} C_{tr}^{j} C_{tail}^{j\frac{q}{q_1} }
        \Big(  \int_{\R} \| f_{j}(z) \|_{X} \; \langle z \rangle^{-\tilde{M}} \dd z  \Big)^{1-\frac{q}{q_1}}
      \end{aligned}
    \\
    &
      \begin{aligned}
        \lesssim\sup_{(\eta,y,t)\in D(x_D,2s_D) \setminus W} \Big( \int_{\R} \| f(z) \|_{X}^{p} \; &t^{-1} \Big\langle\frac{z-y}{t} \Big\rangle^{-M} \dd z \Big)^{\frac{1}{p} }
        \\
        &\sum_{j=-1}^{\infty} C_{tr}^{j} C_{tail}^{j\frac{q}{q_1} }  2^{-j\frac{\tilde{M}}{2}\big( 1-\frac{q}{q_1} \big)},
      \end{aligned}
  \end{align*}
  using the first localisation lemma (Lemma \ref{lem:localised-single-tree}) with some sufficiently large $\tilde{M} > M$. The final sum in $j$ converges and yields the required estimate provided that $\tilde{M}$ is large enough.
  Since $q > p'(r-1)$ is an open condition on $q$, we are free to start with a slightly smaller $q$ so that $q_1$ is the `goal' exponent. 

  Now for the proof of \eqref{eq:localised-Lq-tail-lemma-Lgrowth}, explicitly tracking the dependence on $L$.
  Once more, we may assume that $D=D_{(0,1)}$ (see Remark \ref{rmk:invariances}) and 
  \begin{equation*}
    \sup_{(\eta,y,t)\in D_{(0,2)}\setminus W} \int_{\R} \| f(z) \|_{X}^{p} \; t^{-1} \Big\langle\frac{z-y}{t} \Big\rangle^{-M} \dd z =1.
  \end{equation*}
  Write $f = \sum_{k=-1}^\infty f_k$, where 
  \begin{equation*}
    f_{-1} := f \1_{\{M_p f \leq C\}},
    \qquad
    f_k := f \1_{\{C^{k+1} < M_p f \leq C^{k+2} \}} \qquad \forall k\in\N  
  \end{equation*}
  for some large parameter $C > 1$ to be determined later.
  Decompose the super-level sets of $M_p$ (intersected with $B_{2^L}$) as disjoint unions of open balls as follows:
  \begin{equation*}
    B_{2^{L}}\cap\{x \in \RR : M_pf (x) > C^k \} = \bigcup_n B_{s_{k,n}}(x_{k,n}).
  \end{equation*}
  For convenience we write $B_{k,n} := B_{s_{k,n}}(x_{k,n})$ for all $k,n$.
  Fix  $r_+ \in (r,q)$ (noting that $q > r$ follows from $q > p'(r-1)$ and $p \leq r$, which we have assumed).
  Since $\|f_{-1} \|_{L^\infty(\R;X)}\lesssim 1$, we get
  \begin{align*}
    &\int_{B_{2^{L}}} \| f_{-1}(z) \|_{X}^{r_{+}} \; t^{-1} \Big\langle\frac{z-y}{t} \Big\rangle^{-M} \dd z
      \lesssim 1.
  \end{align*}
  For $k\in\N$ we have $L^\infty$ and support bounds 
  \begin{equation}\label{eq:fk-bounds}
    \begin{aligned}
      & \| f_{k} \|_{L^{\infty}(\RR;X)}\lesssim C^{k} 
      \\
      &|\spt f_k| \lesssim \sum_n |B_{k,n}| \lesssim \|f\|_{L^p(\RR;X)}^{p} C^{-kp} \lesssim  2^{ML}C^{-kp}.
    \end{aligned}
  \end{equation}
  The penultimate bound is a consequence of the $L^p \to L^{p,\infty}$ boundedness of $M_{p}$, and the bound $\|f\|_{L^p(\RR;X)}^{p}\lesssim 2^{ML}$ follows from $(0,0,1)\notin W$ and
  \begin{equation*}
    \int_{B_{2^{L}}} \| f(z) \|_{X}^{p} \, \dd z \lesssim     \int_{B_{2^{L}}} \| f(z) \|_{X}^{p}\langle z\rangle ^{-M} 2^{LM} \, \dd z \lesssim 2^{LM}.
  \end{equation*}
  The balls $B_{k,n}$ are nested, and in particular for each $k \geq 0$ and each $n$ there exists a unique $m(n)$ such that $B_{k,n} \subset B_{1,m(n)}$.
  Furthermore for each $m$ we have
  \begin{equation}\label{eq:nested-ball-density}
    \sum_{n : m(n) = m} s_{k,n} \lesssim C^{-kp} s_{1,m}.
  \end{equation}
  This yields the $L^1$-bound
  \begin{equation}\label{eq:fk-L1-bound}
    \|f_k\|_{L^1(B_{1,m})} \lesssim C^{k(1-p)} s_{1,m}.
  \end{equation}
  for each $m$.
  
  Now let's smash some estimates.
  For all $k \geq -1$, by Theorem \ref{thm:non-iter-embedding} (using that $X$ is $r$-intermediate UMD) and \eqref{eq:fk-bounds},
  \begin{equation*}
    \big\| \1_{D \sm W} \Emb[f_k] \big\|_{L_{\mu}^{r_+} \RS_{out}^{2}} \lesssim \| f_k \|_{L_\mu^{r_+}(\R;X)}  \lesssim  2^{ML/r_+} C^{k(1-\frac{p}{r_{+}})}.
  \end{equation*}
  Lemma \ref{lem:localised-single-tree}, on the other hand, gives forth 
  \begin{align*}
    &\big\| \1_{D \setminus W} \Emb[f_{k}]  \big\|_{L_\mu^{\infty} \RS_{out}^{2}} \lesssim_M \sup_{(\eta,y,t)\in D(0,2)\setminus W} \int_{\R} \| f_{k}(z) \|_{X} t^{-1} \Big\langle\frac{z-y}{t} \Big\rangle^{-M} \dd z.
  \end{align*}
  For each $(\eta,y,t) \in D(0,2) \sm W$ we have
  \begin{align*}
    &\int_{\R} \| f_{k}(z) \|_{X} t^{-1} \Big\langle\frac{z-y}{t} \Big\rangle^{-M} \dd z \\
    &= \sum_m \int_\R \| f_{k}(z) \1_{B_{1,m}}(z) \|_{X} t^{-1} \Big\langle\frac{z-y}{t} \Big\rangle^{-M} \dd z \\
    &= \Big( \sum_{m : t \geq s_{1,m}} + \sum_{m : t < s_{1,m}} \Big) \int_\R \| f_{k}(z) \1_{B_{1,m}}(z) \|_{X} t^{-1} \Big\langle\frac{z-y}{t} \Big\rangle^{-M} \dd z .
  \end{align*}
  
  The first summand above is treated as follows: we have 
  \begin{equation}\label{eq:constant-bracket}
    \Big\langle \frac{z-y}{t} \Big\rangle \simeq \Big\langle \frac{z'-y}{t} \Big\rangle
  \end{equation}
  uniformly in $z,z' \in B_{1,m}$ (with constant independent of $z,z',m$), so
  \begin{align*}
    &\sum_{m : t \geq s_{1,m}} \int_\R \| f_{k}(z) \1_{B_{1,m}}(z) \|_{X} t^{-1} \Big\langle\frac{z-y}{t} \Big\rangle^{-M} \dd z \\
    &\simeq \sum_{m : t \geq s_{1,m}} \| f_{k}(z) \1_{B_{1,m}}(z) \|_{\sL^1(B_{1,m};X)} \int_\R \1_{B_{1,m}}(z)  t^{-1} \Big\langle\frac{z-y}{t} \Big\rangle^{-M} \dd z \\
    &\lesssim C^{k(1-p)},
  \end{align*}
  using the $L^1$-bound \eqref{eq:fk-L1-bound} and disjointness of the balls $B_{1,m}$.

  Now we treat the second summand, in which $t < s_{1,m}$.
  If $C$ is sufficiently large, then $D_{x_{1,m},2s_{1,m}} \subset W$, so for $z \in B_{1,m}$ we have that $|y-z| > s_{1,m} > t$, and so
  \begin{equation*}
    \Big\langle \frac{y-z}{t} \Big\rangle \gtrsim \frac{s_{1,m}}{t}. 
  \end{equation*}
  We proceed to estimate
  \begin{align*}
    &\sum_{m : t < s_{1,m}} \int_\R \| f_{k}(z) \1_{B_{1,m}}(z) \|_{X} t^{-1} \Big\langle\frac{z-y}{t} \Big\rangle^{-M} \, \dd z \\
    &\lesssim   \sum_{m : t < s_{1,m}} \int_\R \| f_{k}(z) \1_{B_{1,m}}(z) \|_{X} \frac{t}{s_{1,m}^2} \Big\langle\frac{z-y}{t} \Big\rangle^{-M+2} \, \dd z \\
    &= \sum_{\ell \in \NN} 2^{-\ell} \sum_{m : s_{1,m} \simeq 2^\ell t} \int_\R  \| f_{k}(z) \1_{B_{1,m}}(z) \|_{X} \frac{1}{s_{1,m}} \Big\langle\frac{z-y}{t} \Big\rangle^{-M+2} \, \dd z \\
    &\leq \sum_{\ell \in \NN} 2^{-\ell} \sum_{m : s_{1,m} \simeq 2^\ell t} \int_\R  \| f_{k}(z) \1_{B_{1,m}}(z) \|_{X} \frac{1}{s_{1,m}} \Big\langle\frac{z-y}{s_{1,m}} \Big\rangle^{-M+2} \, \dd z.
  \end{align*}
  Once more using \eqref{eq:constant-bracket} (substituting $s_{1,m}$ for $t$) and the $L^1$-bound \eqref{eq:fk-L1-bound}, the expression above is controlled by
  \begin{align*}
    & \sum_{\ell \in \NN} 2^{-\ell} \sum_{m : s_{1,m} \simeq 2^\ell t} \int_\R  \| f_{k}(z) \|_{\sL^1(B_{1,m};X)} \1_{B_{1,m}}(z) \frac{1}{s_{1,m}} \Big\langle\frac{z-y}{s_{1,m}} \Big\rangle^{-M+2} \dd z \\
    &\lesssim \sum_{\ell \in \NN} 2^{-\ell} \sum_{m : s_{1,m} \simeq 2^\ell t} \int_\R  \| f_{k}(z) \|_{\sL^1(B_{1,m};X)} \1_{B_{1,m}}(z) \frac{1}{2^\ell t} \Big\langle\frac{z-y}{2^\ell t} \Big\rangle^{-M+2} \dd z \\
    &\lesssim C^{k(1-p)}\sum_{\ell \in \NN} 2^{-\ell}  \int_\R \frac{1}{2^\ell t} \Big\langle\frac{z-y}{2^\ell t} \Big\rangle^{-M+2} \dd z 
      \lesssim C^{k(1-p)}.
  \end{align*}

  Thus we have obtained the two bounds
  \begin{align*}
    \big\| \1_{D \sm W} \Emb[f_k] \big\|_{L_\mu^{r_+} \RS_{out}^{2}} &\lesssim 2^{ML/r_+} C^{k(1-\frac{p}{r_{+}})} \\
    \big\| \1_{D \setminus W} \Emb[f_{k}]  \big\|_{L_\mu^{\infty} \RS_{out}^{2}} &\lesssim C^{k(1-p)}
  \end{align*}
  By interpolation (using that $r < r_+ < q$) we find that
  \begin{align*}
    \big\| \1_{D \setminus W} \Emb[f_{k}]  \big\|_{L_\mu^{q} \RS_{out}^{2}} \lesssim C_{tail}^L C^{k \alpha}
  \end{align*}
  for some constant $C_{tail}$, where
  \begin{align*}
    \alpha := (p-1)\Big(\frac{r_{+}}{q}-1\Big)+\Big(1-\frac{p}{r_{+}}\Big)\frac{r_{+}}{q} < 0
  \end{align*}
  by our assumptions on $p$ and $q$.
  By quasi-subadditivity we obtain 
  \begin{equation*}
    \big\| \1_{D \sm W} \Emb[f] \big\|_{L_\mu^q \RS_{out}^{2}}
    \lesssim \sum_{k=-1}^{\infty} C_{tr}^{k}\big\| \1_{D \sm W} \Emb[f_{k}] \big\|_{L_\mu^q \RS_{out}^{2}}
    \lesssim C_{tail}^L \sum_{k=-1}^\infty C_{tr}^k   C^{k\alpha} \lesssim C_{tail}^L 
  \end{equation*}
  as long as $C$ is large enough.
  This establishes \eqref{eq:localised-Lpq-in-out-lemma} (with normalised supremum), completing the proof.
\end{proof}

\subsection{Proof of the embedding bounds}

\begin{proof}[Proof of Theorem \ref{thm:RS-iterated-embeddings}]
  First we prove the bound
  \begin{equation}\label{eq:iter-weak-endpt}
    \|\Emb[f]\|_{L_\nu^{p,\infty} \sL_\mu^q \RS_{out}^{2}} \lesssim \|f\|_{L^p(\R;X)}
  \end{equation}
  for $p \in (1,\infty)$ and $q = \infty$.
  Fix $\lambda > 0$ and represent the super-level set
  \begin{equation*}
    \big\{x \in \RR : Mf(x) > \lambda\big\} = \bigcup_{n} B_{s_n}(x_n)
  \end{equation*}
  as a disjoint union of balls, and then define
  \begin{equation*}
    K_\lambda := \bigcup_{n} D_{(x_n,s_n)}.
  \end{equation*}
  Since $M$ is of weak type $(1,1)$ we have
  \begin{equation*}
    \nu(K_\lambda) \lesssim \sum_{n} s_n \lesssim \lambda^{-p} \| f \|_{L^p(\RR;X)},
  \end{equation*}
  so it remains to show that 
  \begin{equation}\label{eq:iter-infty-1}
    \big\| \1_{\RR^3_+ \setminus K_\lambda} \Emb[f]  \big\|_{\sL^{\infty}_{\mu} \RS_{out}^{2}} \lesssim \lambda.
  \end{equation}
  For each $T \in \TT$, by Lemma \ref{lem:localised-single-tree} we have
  \begin{align*}
    \big\| \1_{\RR^3_+ \setminus K_\lambda} \Emb[f]  \big\|_{L^{\infty}_{\mu} \RS^{2}_{out}(T)} &\lesssim \sup_{(\eta,y,t) \in \R^{3}_{+} \sm K_\lambda} \int_\R \|f(z)\|_X t^{-1} \Big\langle \frac{z-y}{t} \Big\rangle^{-2} \, \dd z \\
                                                                                                &\lesssim \sup_{y \notin \cup_n B_{s_n}(x_n)} Mf(y) \leq \lambda,
  \end{align*}
  which establishes \eqref{eq:iter-infty-1} and in turn \eqref{eq:iter-weak-endpt} with $q = \infty$.

  Next we prove \eqref{eq:iter-weak-endpt} for $p > r$ and $q > r$.
  Fix $r_{+}\in (r,\min(p,q) )$ and  $\lambda > 0$, represent
  \begin{equation*}
    \big\{x \in \RR : M_{r_+} f (x) > \lambda\big\} = \bigcup_{n} B_{s_n}(x_n)
  \end{equation*}
  as a disjoint union of balls, and as before define
  \begin{equation*}
    K_\lambda := \bigcup_{n} D_{(x_n,s_n)}.
  \end{equation*}
  Since $M_{r_+}$ is bounded on $L^{q}$ we  have
  \begin{equation*}
    \nu(K_\lambda) \lesssim \sum_{n} s_n \lesssim \lambda^{-q} \| f \|_{L^q(\RR;X)},
  \end{equation*}
  and as in the previous case we deduce the estimate
  \begin{equation*}
    \big\| \1_{\RR^3_+ \setminus K_\lambda} \Emb[f]  \big\|_{\sL^{q}_{\mu} \RS^{2}_{out}} \lesssim \lambda
  \end{equation*}
  from the second localisation lemma (Lemma \ref{lem:localised-Lq-tail-lemma}). 

  Finally can we show \eqref{eq:iter-weak-endpt} for $p \in (1, r]$ and $q > p'(r-1)$.
  The proof is the same as in the previous two cases, except this time we use the super-level set of $M_p$ and the third localisation lemma (Lemma \ref{lem:localised-Lpq-in-out-lemma}).
  The result then follows by Marcinkiewicz interpolation for outer Lebesgue spaces. 
\end{proof}


%% file: main/bhf.tex
We have filled a considerable number of pages doing time-frequency analysis on $\R^3_+$ and bounding embedding maps into outer Lebesgue spaces.
It is time to consolidate this theory to prove estimates for the bilinear Hilbert transform $\BHT_{\Pi}$ and the dual trilinear form $\BHF_{\Pi}$, both defined with respect to a general bounded trilinear form $\map{\Pi}{X_1 \times X_2 \times X_3}{\CC}$ on the product of intermediate UMD spaces.
As in Section \ref{sec:size-holder}, we set $\alpha = (1,1,-2)$ and $\beta = (-1,1,0)$, we fix $\mf{b} = 2^{-4}$ and $\Theta = B_2$, and we define $\Theta_i:= \alpha_i \Theta + \beta_i$.

\subsection{\texorpdfstring{$L^p$}{Lp} bounds for \texorpdfstring{$\BHF_\Pi$}{BHF\_Pi}}

First we prove our main theorem.

\begin{proof}[Proof of Theorem \ref{thm:intro-BHT}]
  The assumption \eqref{eq:intro-BHT-exponents} on the exponents $(p_i)_{i=1}^3$ implies that there exists a Hölder triple $(q_i)_{i=1}^3$ such that
  \begin{equation*}
    q_i > \min(p_i,r_i)'(r_i - 1)
  \end{equation*}
  for each $i$.
  Fix $\phi_0 \in \Sch(\R;\C)$ with Fourier support in $B_{\mf{b}}$ such that
  \begin{equation*}
    \widetilde{\BHF}_{\Pi}(f_1,f_2,f_3) = \lim_{\mathbb{K} \uparrow \R^3_+} \BHWF_{\Pi}^{\phi_0,\mathbb{K}} (\Emb[f_1], \Emb[f_2], \Emb[f_3]),
  \end{equation*}
  where $\mathbb{K}$ runs over all compact sets $V_+ \sm V_-$ with $V_\pm \in \TT_\Theta^\cup$.
  Then we can estimate
  \begin{equation*}
    \begin{aligned}
      |\BHWF_{\Pi}^{\phi_0,\mathbb{K}} (\Emb[f_1], \Emb[f_2], \Emb[f_3])|
      \lesssim \prod_{i=1}^3 \|\Emb[f_i]\|_{L_\nu^{p_i} \sL_{\mu_i}^{q_i} \FS_{\Theta_i}^{3}}     \lesssim \prod_{i=1}^3 \|f_i\|_{L^{p_i}(\R;X_i)},
    \end{aligned}
  \end{equation*}
  where the second inequality uses the embedding bounds of Theorem \ref{thm:intro-embeddings}, justified by the assumptions on the Banach spaces $X_i$ and the exponents $p_i,q_i$.
  The proof is complete.
\end{proof}

\subsection{Sparse domination of \texorpdfstring{$\BHF_{\Pi}$}{BHF\_Pi}}\label{sec:sparse-dom}
The $L^p$-boundedness we just proved can be improved to a sparse domination result, quantifying localisation properties of $\BHF_{\Pi}$ that are strictly stronger than $L^p$-boundedness.
Although sparse domination is primarily of interest for its connection with sharp weighted estimates, our goal is to use it to establish estimates for the bilinear operator $\BHT_{\Pi}$ in the quasi-Banach range, i.e. $L^p$-bounds for exponents $p < 1$.
We begin by recalling one of the many equivalent definitions of a sparse family of strips.\footnote{Strips correspond to intervals in $\R$, and the equivalent formulation in terms of intervals is certainly more familiar. We use the language of strips to emphasise the connection with outer spaces on $\R^3_+$.}
First, given a strip $D\in\DD$ we define the `top half'
\begin{equation*}
  D^{\uparrow}:= D\cap \{(\eta,y,t)\colon t>s_{D}/2\}.
\end{equation*}

\begin{defn}[Sparse strip collection]
  The sparsity constant of a collection of strips $\mc{G}\subset \DD$ is given by
  \begin{equation}
    \label{eq:collection-sparsity}
    \|\mc{G}\|_{sparse} := \Big\|\sum_{D\in\mc{G}}\1_{D^{\uparrow}}\Big\|_{L^{\infty}_{\nu}\sL_{\mu}^{\infty}\lL^{1}}.
  \end{equation}
\end{defn}

When the collection $\mc{G}$ is nested, i.e. when
\begin{equation*}
  D,D'\in\mc{D} \implies D\subseteq D' \text{ or } D'\subseteq D,
\end{equation*}
we have that
\begin{equation*}
  \|\mc{G}\|_{sparse} = C_{\Theta} \sup_{D\in\mc{G}}  \sum_{\substack{D'\in\mc{G}\\D'\subset D}} |s_{D}|.
\end{equation*}

Next we prove an abstract `outer' sparse domination of $\BHWF$, relying only on the outer structures on $\R^{3}_+$, and not requiring any boundedness of embedding maps (and thus no assumptions on the Banach spaces $X_i$).

\begin{prop}\label{prop:outer-sparse-domination}
  For any $p_{1}^{*},p_{2}^{*},p_{3}^{*}\in[1,\infty]$ such that $\sum_{i=1}^{3} (p_{i}^{*})^{-1}>1$, any Hölder triple $q_1, q_2, q_3 \in [1,\infty]$, any compact $\mathbb{K} \in \R^3_+$, and any $F_i \in L^{p_{i}^{*}}_{\mu}\sL^{q_{i}}_{\nu} \FS^{s_{i}}$, there exists a finite nested collection $\mc{G} \subset \DD$ with $\|\mc{G}\|_{sparse}\lesssim_{\Theta} 1$ such that
  \begin{equation} \label{eq:outer-sparse-domination}
    \big|\BHWF_{\Pi}^{\phi_0,\mathbb{K}} (F_{1},F_{2},F_{3})\big|\lesssim \sum_{D\in\mc{G}}\prod_{i=1}^{3} \|\1_{D}F_{i}\|_{L^{p_{i}^{*}}_{\mu}\sL^{q_{i}}_{\nu} \FS^{s_{i}}}.
  \end{equation}
\end{prop}

\begin{proof}
  Start with a strip $D_{0}\in\DD$ containing $\mathbb{K}$ and set $\mc{G}_{0}=\{D_{0}\}$. By the definition of the outer Lebesgue quasinorms, for $C$ sufficiently large, for every $D\in\DD$ and $i \in \{1,2,3\}$ there exists a countable union of strips $W^{(i)}\in\DD^{\cup}$ contained in $D$ such that
  \begin{equation*}
    \begin{aligned}[t]
      \|\1_{D\setminus W^{(i)}}F_{i}\|_{L^{\infty}_{\nu}\sL^{q_{i}}_{\mu}\FS^{s_{i}}}&\leq C   \|\1_{D}F_{i}\|_{L^{p_{i}^{*}}_{\nu}\sL^{q_{i}}_{\mu}\FS^{s_{i}}}, 
      \\
      \nu(W^{(i)})&\leq C^{-p_{i}^{*}} \nu(D).
    \end{aligned}
  \end{equation*}
  Now define $\mc{G}^{(i)}(D)$ to be a finite collection of strips such that 
  \begin{equation*}
    \begin{aligned}[t]
      W^{(i)}=\bigcup_{D'\in\mc{G}^{(i)}(D)}D' \quad \text{and} \quad \sum_{D'\in\mc{G}^{(i)}(D)} s_{D'}\leq 2C^{-p_{i}^{*}} \nu(D).
    \end{aligned}
  \end{equation*}
  For $n \geq 1$, define $\mc{G}_n$ inductively by
  \begin{equation*}
    \mc{G}_{n+1}=\bigcup_{i=1}^{3}\bigcup_{D\in\mc{G}_{n}}\mc{G}^{(i)}(D)
  \end{equation*}
  and let $\mc{G}=\bigcup_{n\in\N}\mc{G}_{n}$. Since $\mathbb{K}$ is compact and the measures of $D'\in\mc{G}(D)$ decrease geometrically, $\mc{G}_{n}$ is eventually empty. $\mc{G}$ is finite and nested by construction, and satisfies $\|\mc{G}\|_{sparse} \lesssim_{\Theta} 1$ by the measure bound above as long as $C$ is sufficiently large.
  By telescoping, for any Hölder triple $p_{1},p_{2},p_{3} \in[1,\infty]$  we have 
  \begin{equation}
    \begin{aligned}[t]
      \big|\BHWF_{\Pi}^{\phi_0,\mathbb{K}} (F_{1},F_{2},F_{3})\big| &\leq \sum_{n\in\N}\sum_{D\in\mc{G}_{n}} \big|\BHWF_{\Pi}^{\phi_0,(D\setminus\cup \mc{G}_{n+1}) \cap\mathbb{K}} (F_{1},F_{2},F_{3})\big| \\
      &\lesssim\sum_{n \in \N} \sum_{D\in\mc{G}_n} \prod_{i=1}^{3} \|\1_{D \setminus \cup \mc{G}_{n+1}}F_{i}\|_{L^{p_{i}}_{\nu}\sL^{q_{i}}_{\mu}\FS^{s_{i}}}
    \end{aligned}
  \end{equation}
  by Corollary \ref{cor:BHT-RN-reduction}.
  If $p_{i}\geq p_{i}^{*}$ then by log-convexity of outer Lebesgue quasinorms
  \begin{equation*}
    \begin{aligned}[t]
      \|\1_{D\setminus\cup \mc{G}_{n+1}}F_{i}\|_{L^{p_{i}}_{\nu}\sL^{q_{i}}_{\mu}\FS^{s_{i}}}&\lesssim    \|\1_{D\setminus\cup \mc{G}_{n+1}}F_{i}\|_{L^{p_{i}^{*}}_{\nu}\sL^{q_{i}}_{\mu}\FS^{s_{i}}}^{p_{i}^{*}/p_{i}}     \|\1_{D\setminus\cup \mc{G}_{n+1}}F_{i}\|_{L^{\infty}_{\nu}\sL^{q_{i}}_{\mu}\FS^{s_{i}}}^{1-p_{i}^{*}/p_{i}}
      \\
      &\lesssim    \|\1_{D\setminus\cup \mc{G}_{n+1}}F_{i}\|_{L^{p_{i}^{*}}_{\nu}\sL^{q_{i}}_{\mu}\FS^{s_{i}}}.
    \end{aligned}
  \end{equation*}
  By the assumption on $(p_{i}^{*})_{i=1}^3$ we can find a H\"older triple $(p_i)_{i=1}^3$ with $p_i \geq p_i^{*}$ for all $i$, so we conclude that
  \begin{equation*}
    \begin{aligned}
      \big|\BHWF_{\Pi}^{\phi_0,\mathbb{K}} (F_{1},F_{2},F_{3})\big|
      &\lesssim \sum_{n \in \N} \sum_{D\in\mc{G}_n} \prod_{i=1}^{3} \|\1_{D \setminus \cup \mc{G}_{n+1}}F_{i}\|_{L^{p_{i}^{*}}_{\nu}\sL^{q_{i}}_{\mu}\FS^{s_{i}}} \\
      &\lesssim \sum_{D\in\mc{G}} \prod_{i=1}^{3} \|\1_{D}F_{i}\|_{L^{p_{i}^{*}}_{\nu}\sL^{q_{i}}_{\mu}\FS^{s_{i}}}  
    \end{aligned}
  \end{equation*}
  and the proof is done.
\end{proof}

Extending this outer sparse domination to a true sparse domination result for $\BHF_\Pi$ relies on a localisation property of the embedding map.

\begin{lem}[Fourth localisation lemma]\label{lem:4th-loc}
  Let $X$ be a Banach space, let $p,q,s\in[1,\infty]$ and $N\in\N$, and suppose that the embedding bound
  \begin{equation*}
    \|\Emb[f]\|_{L^{p}_{\nu}L_{\upsilon}^{q}\FS^{s}_{N}} \lesssim \|f\|_{L^{p}(\R;X)}\qquad \forall f\in\Sch(\R;X)
  \end{equation*}
  holds. Then for all $N'\in\N$ and all strips $D \in \DD$ we have
  \begin{equation*}
    \|\1_{D}\Emb[f]\|_{L_{\nu}^{p}L_{\mu}^{q}\FS^{s}_{N+N'+1}} \lesssim_{N,N'} \|f(z) \langle (z-x_{D})/|I_{D}| \rangle^{-N'} \|_{L_{\dd z}^{p}(\R;X)}.
  \end{equation*}
\end{lem}

\begin{proof}
  By symmetry we may assume that $D=D_{(0,1)}$. For all $n\in\N$, $T\in\TT$, $(\theta,\zeta,\sigma)\in\mT$, and $\phi\in\Phi_{1}^{N+N'+1}$ there exists $\tilde{\phi}_{n}\in\Phi$ with $\|\tilde{\phi}_{n}\|_{N}\lesssim 2^{-n(N'+1)}$ such that
  \begin{equation*}
    \pi_{T}^{*}\big(\1_{D}\Emb[\1_{\R\setminus B_{2^{n}}}f]\big)[\phi](\theta,\zeta,\sigma)
    = \pi_{T}^{*}\big(\1_{D}\Emb[\1_{\R\setminus B_{2^{n}}}f]\big)[\tilde{\phi}_{n}](\theta,\zeta,\sigma).
  \end{equation*}
  By the definition of $\FS^{s}$ and by the assumed embedding bounds we find that  
  \begin{equation*}
    \|\1_{D}\Emb[\1_{\R\setminus B_{2^{n}}}f]\|_{L_{\nu}^{p}L_{\mu}^{q}\FS^{s}_{N+N'+1}} \lesssim_{N,N'} 2^{-N'-1}\|f\|_{L^{p}(\R;X)}.
  \end{equation*}
  Summing the above estimates applied to the functions $\1_{B_{1}}f$ and $2^{nN'} \1_{B_{2^{n+1}}\setminus B_{2^{n}}}f$ over $n\in\N$ completes the proof.
\end{proof}

We are now ready to prove sparse domination of $\BHF_\Pi$.

\begin{thm}\label{thm:BHF-sparse-dom}
  Let $(X_{i})_{i=1}^{3}$ be $r_i$-intermediate UMD spaces for some $r_i \in [2,\infty)$, and let $\map{\Pi}{X_0 \times X_1 \times X_2}{\CC}$ be a bounded trilinear form.
  For all H\"older triples $(p_{i}^{*})_{i=1}^{3}$ in $(1,\infty)$ satisfying
  \begin{equation}\label{eq:SD-BHT-exponents}
    \sum_{i=1}^3 \frac{1}{\min(p_{i}^{*},r_i)'(r_i - 1)} > 1
  \end{equation}
  and all $f_i \in L^{p_{i}^{*}}(\R;X_i)$, if $N' \in \N$ is sufficiently large we have the sparse domination
  \begin{equation*}
    \begin{aligned}
      &\big|\BHF_{\Pi} (f_{1},f_{2},f_{3})\big| \\
      &\qquad \lesssim \sup_{\substack{\mc{G} \subset \DD \\ \|\mc{G}\|_{sparse} \lesssim_{\Theta} 1}} \sum_{D\in\mc{G}} \prod_{i=1}^{3} \Big(\frac{1}{|I_{D}|}\int_{\R}\|f_{i}(z)\|_{X_i}^{p_{i}^{*}} \Big\langle \frac{z-x_{D}}{|I_{D}|}\Big\rangle^{-N'} \dd z\Big)^{\frac{1}{p_{i}^{*}}}|I_{D}|
    \end{aligned}
  \end{equation*}
  where the supremum is over all finite nested collections $\mc{G} \subset \DD$.
\end{thm}

\begin{proof}
  As $\BHF_{\Pi}$ is a linear combination of the H\"older form (which satisfies the requisite sparse domination) and $\tilde{\BHF}_{\Pi}$, it suffices to prove the sparse domination for $\tilde{\BHF}_{\Pi}$.
  As in the proof of Theorem \ref{thm:intro-BHT}, we can estimate
  \begin{equation*}
    \begin{aligned}
      |\tilde{\BHF}_{\Pi}(f_1, f_2, f_3)|
      &\lesssim \lim_{\mathbb{K} \uparrow \R^3_+} \BHWF_{\Pi}^{\phi_0, \mathbb{K}}(\Emb[f_1], \Emb[f_2], \Emb[f_3]) \\
      &\lesssim \lim_{\mathbb{K} \uparrow \R^3_+} \sum_{D \in \mc{G}(\mathbb{K}, f_1, f_2, f_3)} \prod_{i=1}^3 \|\1_{D} \Emb[f_i]\|_{L_{\mu}^{p_{i}^{*}} \sL_{\mu}^{q_i} \FS^{s_i}(D)}
    \end{aligned}
  \end{equation*}
  for appropriately chosen exponents, relying on the assumption \eqref{eq:SD-BHT-exponents}.
  The collections $\mc{G}(\mathbb{K}, f_1, f_2, f_3)$ are finite and nested with controlled sparsity constant, so
  \begin{equation*}
    \begin{aligned}
      |\tilde{\BHF}_{\Pi}(f_1, f_2, f_3)| \lesssim \sup_{\substack{\mc{G} \subset D \\ \|\mc{G}\|_{sparse} \lesssim_{\Theta} 1}} \sum_{D \in \mc{G}} \prod_{i=1}^3 \|\1_{D} \Emb[f_i]\|_{L_{\mu}^{p_{i}^{*}} \sL_{\mu}^{q_i} \FS^{s_i}(D)}.
    \end{aligned}
  \end{equation*}
  Applying Lemma \ref{lem:4th-loc} completes the proof. 
\end{proof}

\begin{rmk}
  The sparse domination proven in Theorem \ref{thm:BHF-sparse-dom} is stated with respect to the smoothed averages
  \begin{equation*}
    \Big(\frac{1}{|I_{D}|}\int_{\R}\|f_{i}(z)\|_{X_i}^{p_{i}^{*}} \Big\langle \frac{z-x_{D}}{|I_{D}|}\Big\rangle^{-N'} \dd z \Big)^{\frac{1}{p_{i}^{*}}}
  \end{equation*}
  rather than the more familiar averages
  \begin{equation*}
    \langle \|f_i\|_{X_i} \rangle_{p_{i}^{*}, I_D} := \Big( \frac{1}{|I_D|} \int_{I_D} \|f_i(z)\|_{X_i}^{p_{i}^{*}} \, \dd z \Big)^{\frac{1}{p_{i}^{*}}}.
  \end{equation*}
  However, it seems to be well-known that for $\map{f_i}{\R}{\C}$ we have the equivalence
  \begin{equation*}
    \begin{aligned}
      &\sup_{\substack{\mc{G} \subset \DD \\ \|\mc{G}\|_{sparse} \lesssim_{\Theta} 1}} \sum_{D\in\mc{G}} \prod_{i=1}^{3} \Big(\frac{1}{|I_{D}|}\int_{\R} |f_{i}(z)|^{p_{i}^{*}} \Big\langle \frac{z-x_{D}}{|I_{D}|}\Big\rangle^{-N'} \dd z\Big)^{\frac{1}{p_{i}^{*}}}|I_{D}| \\
      &\simeq \sup_{\substack{\mc{G} \subset \DD \\ \|\mc{G}\|_{sparse} \lesssim_{\Theta} 1}} \sum_{D\in\mc{G}} \prod_{i=1}^{3} \langle |f_i| \rangle_{p_{i}^{*}, I_D} |I_{D}|
    \end{aligned}
  \end{equation*}
  up to a change in implicit constant in the sparsity conditions.
  The direction $\gtrsim$ follows from the pointwise control of an $L^1$-normalised characteristic function by its smoothed version; here we sketch a proof of the reverse direction.
  Use the estimate
  \begin{equation*}
    \Big\langle \frac{z-x_{D}}{|I_{D}|}\Big\rangle^{-N'}\lesssim  \sum_{n\in\N}2^{-nN'}\1_{B_{2^{n}s_{D}}(x_{D})}(z)
  \end{equation*}
  to write
  \begin{equation*}
    \begin{aligned}[t]
      &\sum_{D\in\mc{G}} \prod_{i=1}^{3} \Big(\frac{1}{|I_{D}|}\int_{\R}|f_{i}(z)|^{p_{i}^{*}} \Big\langle \frac{z-x_{D}}{|I_{D}|}\Big\rangle^{-N'} \dd z\Big)^{\frac{1}{p_{i}^{*}}}|I_{D}|
      \\
      &\lesssim \sum_{n_{1},n_{2},n_{3}\in\N}2^{-\sum_{i=1}^{3}n_{i}N'}\sum_{D\in\mc{G}} \prod_{i=1}^{3}s_{D}^{-\sum_{i=1}^{3}1/p_{i}^{*}} \Big(\int_{B_{2^{n_{i}}s_{D}}(x_{D})}|f_{i}(z)|^{p_{i}^{*}} \dd z\Big)^{\frac{1}{p_{i}^{*}}} |I_{D}|.
    \end{aligned}
  \end{equation*}
  By symmetry it suffices to consider the part of the sum where $n_{1}=\max(n_{1},n_{2},n_{3})$:
  \begin{equation*}
    \begin{aligned}[t]
      &\sum_{n_{1}}\sum_{n_{2},n_{3}=0}^{n_{1}}2^{-N'\sum_{i=1}^{3}n_{i}}\sum_{D\in\mc{G}} \prod_{i=1}^{3}s_{D}^{-\sum_{i=1}^{3}1/p_{i}^{*}} \Big(\int_{B_{2^{n_{i}}s_{D}}(x_{D})}|f_{i}(z)|^{p_{i}^{*}} \dd z\Big)^{\frac{1}{p_{i}^{*}}} |I_{D}|
      \\
      &
      \lesssim
      \sum_{n_{1}}2^{-n_{1}(1+N'-\sum_{i=1}^{3}1/p_{i}^{*})}\sum_{D\in\mc{G}} \prod_{i=1}^{3} \Big(\fint_{B_{2^{n_{1}}s_{D}}(x_{D})}|f_{i}(z)|^{p_{i}^{*}} \dd z\Big)^{\frac{1}{p_{i}^{*}}} 2^{n_{1}}|I_{B_{2^{n_{1}}s_{D}}(x_{D})}|.
    \end{aligned}
  \end{equation*}
  The collection $\mc{G}_{n_{1}}:=\{B_{2^{n_{1}}s}(x) \colon B_{s}(x)\in\mc{G}\}$ satisfies $\|\mc{G}_{n_{1}}\|_{sparse}\lesssim 2^{n_{1}}$, and this leads to the proof. 
\end{rmk}

\subsection{Bounds for \texorpdfstring{$\BHT_\Pi$}{BHT\_Pi}, including quasi-Banach exponents}\label{sec:bounds-for-BHT}

Now we consider the bilinear Hilbert transform $\BHT_\Pi$ defined in \eqref{eq:BHT-defn} rather than the dual trilinear form $\BHF_\Pi$.
This lets us consider bounds into quasi-Banach $L^p$-spaces, which is not possible in the trilinear formulation because the spaces $L^p(\R)$ with $p < 1$ have trivial duals.
More precisely, for H\"older triples $(p_i)_{i=1}^3$ with $p_i \in [1,\infty)$ the bound
\begin{equation}\label{eq:BHT-bounds-ref}
  \map{\BHT_{\Pi}}{L^{p_1}(\R; X_1) \times L^{p_2}(\R; X_2)}{L^{p_3}(\R;X_3^*)} 
\end{equation}
is equivalent to
\begin{equation}\label{eq:BHF-bounds-dual}
  \map{\BHF_{\Pi}}{L^{p_1}(\R; X_1) \times L^{p_2}(\R; X_2) \times L^{p_3'}(\R;X_3)}{\CC},
\end{equation}
but when $p_3 < 1$ the bound \eqref{eq:BHT-bounds-ref} may hold while having no equivalent trilinear representation.

\begin{thm}\label{thm:BHT-bounds}
  Let $(X_{i})_{i=1}^3$ be $r_i$-intermediate UMD spaces for some $r_i \in [2,\infty)$, and let $\map{\Pi}{X_0 \times X_1 \times X_2}{\CC}$ be a bounded trilinear form.
  Fix $p_1,p_2 \in (1,\infty)$ and $p_3 \in (0,\infty)$ with $p_1^{-1} + p_2^{-1} = p_3^{-1}$, and suppose that
  \begin{equation*}
    \sum_{i=1}^3 \frac{1}{\min(p_i,r_i)'(r_i - 1)} > 1.
  \end{equation*}
  Then $\BHT_\Pi$ satisfies the bound \eqref{eq:BHT-bounds-ref}.
\end{thm}

\begin{proof}
  The result for $p_3 > 1$ follows from Theorem \ref{eq:intro-Lpbd}.
  For $p_3 \leq 1$ it follows from the sparse domination for $\BHF_\Pi$ established in Theorem \ref{thm:BHF-sparse-dom}: this implies weighted estimates which, by extrapolation, yield the claimed bounds \eqref{eq:BHT-bounds-ref} (plus further weighted estimates that we do not discuss here).
  The details of this weighted multilinear extrapolation are worked out in \cite[Section 4.1]{bN19} (see also \cite{LMMOV20, LMO18}).
\end{proof}

\subsection{Examples}\label{sec:examples}

Let us consider some examples of Banach spaces $(X_i)_{i=1}^3$ and trilinear forms $\Pi$ to which our theorems apply.
In particular, we need each $X_i$ to be $r_i$-intermediate UMD, such that the region of H\"older triples $(p_i)_{i=1}^3$ satisfying \eqref{eq:intro-BHT-exponents} is nonempty; this occurs if and only if
\begin{equation}\label{eq:exp-condn}
  \sum_{i=1}^3 \frac{1}{r_i} > 1.
\end{equation}

\subsubsection{Single intermediate UMD spaces}
Given any Banach space $X$, there is a natural trilinear form
\begin{equation*}
  \map{\Pi}{X \times X^* \times \CC}{\CC}, \qquad \Pi(\mb{x},\mb{x}^*,\lambda) := \lambda \mb{x}^*(\mb{x}).
\end{equation*}
If $X$ is $r$-intermediate UMD then so is $X^*$, so the set of exponents satisfying \eqref{eq:intro-BHT-exponents} is nonempty provided that $r < 4$.
If $X$ is not a Banach lattice, then there are no known bounds for $\BHT_\Pi$ besides our results.\footnote{Excluding of course the results of \cite{DPLMV19-3} discussed in Remark \ref{rmk:DPLMV}.}

\subsubsection{Triples of sequence spaces}

If $(r_i)_{i=1}^3$ are exponents satisfying $\sum_{i=1}^3 r_i^{-1} \geq 1$, then
\begin{equation*}
  \map{\Pi}{\ell^{r_1} \times \ell^{r_2} \times \ell^{r_3}}{\CC}, \qquad \Pi(f,g,h) = \sum_{n \in \N} f(n) g(n) h(n)
\end{equation*}
is bounded.
If $r_i \in [2,\infty)$ for each $i$ then each $\ell^{r_i}$ is $\tilde{r}_i$-intermediate UMD for all $\tilde{r}_i > r_i$, and if in addition $\sum_{i=1}^3 r_i^{-1} > 1$, then Theorem \ref{thm:intro-BHT} applies.
However, bounds for $\BHT_\Pi$ in this case have already appeared in the works of Silva \cite{pS14}, Benea and Muscalu \cite{BM16,BM17}, and Lorist and Nieraeth \cite{LN19, LN20, bN19}, all of which allow for a larger range of sequence spaces (including the non-UMD space $\ell^\infty$).
\subsubsection{Triples of Lebesgue spaces}\label{eg:lebesgue}
Consider a H\"older triple of exponents $(r_i)_{i=1}^3$, so that the spaces $L^{r_i}(\R)$ are $\tilde{r}_i$-intermediate for all $\tilde{r}_i > \max(r_i,r_i')$.
Consider the trilinear form
\begin{equation*}
  \map{\Pi}{L^{r_1}(\R) \times L^{r_2}(\R) \times L^{r_3}(\R)}{\CC}, \qquad \Pi(f,g,h) = \int_{\R} f(x) g(x) h(x) \, \dd x.
\end{equation*}
Condition \eqref{eq:exp-condn} holds if and only if $\max(r_i, r_i') < r_i$ for some $i$, which is impossible, so Theorem \ref{thm:intro-BHT} does not apply to these spaces.
However, bounds for $\BHT_{\Pi}$ in this case do hold as long as $r_1, r_2, r_3 \in (1,\infty]$ \cite{BM16,BM17, LN19, LN20, bN19}.
Our results are far from optimal here.

\subsubsection{Triples of Schatten classes}
The purpose of our results was to prove bounds for $\BHF_\Pi$ when the spaces $X_i$ are \emph{not} Banach lattices, and here we succeed.
Consider for example the `non-commutative Hölder form': for a H\"older triple $(r_i)_{i=1}^3$ satisfying $\sum_{i=1}^3 r_i \geq 1$ as in the sequence space example above, the map
\begin{equation*}
  \map{\Pi}{\mc{C}^{r_1} \times \mc{C}^{r_2} \times \mc{C}^{r_3}}{\CC}, \qquad \Pi(F,G,H) := \tr(FGH)
\end{equation*}
is bounded.
Here $\mc{C}^{r}$ denotes the Schatten class of compact operators on a Hilbert space with approximation numbers in $\ell^r$ (see \cite[Appendix D]{HNVW17}).
Each $\mc{C}^{r_i}$ is $\tilde{r}_i$-intermediate UMD for all $\tilde{r}_i > \max(r_i,r_i')$, and we obtain bounds for $\BHF_\Pi$ exactly when we obtain them with $\mc{C}^r$ replaced by the sequence space $\ell^r$.  
Naturally one can consider more general non-commutative $L^p$ spaces, as described in \cite{PX03}, provided that they satisfy the same containment relations as $\ell^p$ spaces.

\subsection{Sharpness of assumptions}
\label{sec:sharpness}

Example \ref{eg:lebesgue} above shows that for `classical' vector-valued estimates, the results we obtain are quite suboptimal.
In this section we show, somewhat surprisingly, that our results are actually sharp.\footnote{This also shows that the results of \cite{DPLMV19-3} are sharp, in the same sense. The paradox is resolved by understanding what we mean by `sharp'. We stress that our interpretation is not unreasonable.}

Our results require that the intermediate UMD exponents $r_i$ of $X_i$ satisfy \eqref{eq:exp-condn}.
This condition arises naturally in both our proof and that of \cite{DPLMV19-3}, but it turns out to be a natural obstruction, as the following proposition shows.

\begin{thm}\label{thm:counterexample}
  For all $\varepsilon > 0$ there exist $r_i$-intermediate UMD Banach spaces $X_i$, with
  \begin{equation*}
    \sum_{i=1}^3 \frac{1}{r_i} = 1 - \varepsilon,
  \end{equation*}
  and a bounded trilinear form $\map{\Pi}{X_1 \times X_2 \times X_3}{\CC}$ such that $\BHF_{\Pi}$ does not satisfy any $L^p$-estimates of the form \eqref{eq:intro-Lpbd}.
\end{thm}

The argument relies on the following observation of Muscalu, Pipher, Tao, and Thiele.
Consider the tensor product $\BHT \otimes \BHT$ of two (scalar-valued) bilinear Hilbert transforms, acting on Schwartz functions $f,g \in \Sch(\R^2)$ as
\begin{equation*}
  \BHT \otimes \BHT(f,g)(x,y) = \pv \int_{\RR^2} f(x-t_1, y-t_2) g(x+t_1, y+t_2) \, \frac{\dd t_1}{t_1} \, \frac{\dd t_2}{t_2}.
\end{equation*}
It was shown in \cite[Theorem 7.1]{MPTT04} that this operator does not satisfy any $L^p$ estimates.
This result can be extended almost trivially to mixed-norm estimates.

\begin{lem}\label{lem:Bd}
  Let $(p_i)_{i=1}^3$ and $(q_i)_{i=1}^3$ be H\"older triples of exponents.
  Then the mixed-norm estimate
  \begin{equation}\label{eq:mixed-norm-Bd}
    \|\BHT \otimes \BHT(f,g)\|_{L^{p_3'} L^{q_3'}(\RR^2)} \lesssim \|f\|_{L^{p_1}L^{q_1}(\RR^2)}  \|g\|_{L^{p_2}L^{q_2}(\RR^2)} \qquad \forall f,g \in \Sch(\R^2)
  \end{equation}
  does not hold.
\end{lem}

\begin{proof}
  For $N > 0$ consider the functions
  \begin{equation*}
    f_N(x,y) = e^{ixy}\chi_{[-N,N]}(x)\chi_{[-N,N]}(y).
  \end{equation*}
  In the proof of \cite[Theorem 7.1]{MPTT04} it is shown that there exists $C > 0$ such that for $N$ sufficiently large and $x,y \in [-N/1000, N/1000]$,
  \begin{equation*}
    |B_d(f_N,f_N)(x,y)| \geq C\log N + O(1).
  \end{equation*}
  Taking $f = g = f_N$ and assuming the estimate \eqref{eq:mixed-norm-Bd} holds, this implies
  \begin{equation*}
    \log(N) N^{\frac{1}{p_3'} + \frac{1}{q_3'}} \lesssim N^{\frac{1}{p_1} + \frac{1}{q_1}} N^{\frac{1}{p_2} + \frac{1}{q_2}}
  \end{equation*}
  for large $N$.
  Rearranging and using the H\"older condition shows that $\log(N) \lesssim 1$ for large $N$, which is a contradiction.
\end{proof}

\begin{proof}[Proof of Theorem \ref{thm:counterexample}] 
  Fix exponents $(q_i)_{i=1}^3$ satisfying
  \begin{equation*}
    2 \leq q_i < \infty \quad \text{and} \quad \sum_{i=1}^{3} q_i^{-1} = 1,
  \end{equation*}
  and let $X_i = L^{q_i}(\RR)$.
  Then $X_i$ is $r_i$-intermediate UMD for all $r_i > q_i$, and by choosing $r_i$ close to $q_i$ we can make the sum $\sum_i r_i^{-1} < 1$ arbitrarily close to $1$.
  Let $\map{\Pi}{X_1 \times X_2 \times X_3}{\CC}$ be the trilinear form $\BHF$, and consider the associated trilinear form $\BHF_{\Pi} = \BHF_\BHF$.
  Then $\BHF_{\Pi}$ is the trilinear form dual to $\BHT \otimes \BHT$, 
  so the estimate
  \begin{equation*}
    |\BHF_\BHF(f_1,f_2,f_3)| \lesssim \prod_{i=1}^3 \|f_i\|_{L^{p_i}(\RR;X_i)} \qquad \forall f_i \in \Sch(\R;X_i)
  \end{equation*}
  would imply the boundedness of $\map{\BHT \otimes \BHT}{L^{p_1}L^{q_1} \times L^{p_2}L^{q_2}}{L^{p_{3}'}L^{q_{3}'}}$.
  But by Lemma \ref{lem:Bd} this cannot hold, so $\BHF_\Pi$ does not satisfy any $L^{p_i}$ bounds.
\end{proof}

The most interesting thing about this counterexample is that the Banach spaces $X_i = L^{r_i}(\RR)$ are very well-behaved UMD Banach function spaces.
The problem lies in the choice of underlying trilinear form $\Pi$.
When $\Pi$ is chosen to be the H\"older form as in Example \ref{eg:lebesgue}, the trilinear form $\BHF_\Pi$ satisfies a large range of $L^p$ estimates, as shown in \cite{BM16, LN19, LN20, bN19}.
In fact, even if $\Pi$ is chosen to be a paraproduct, $\BHF_\Pi$ still satisfies a large range of $L^p$ estimates (this is also shown in \cite{BM16}; in this case $\BHF_\Pi$ can be identified with $\BHF \otimes \Pi$).

In the language of \cite{LN20}, the triple $(L^{p_1},L^{p_2},L^{p_3})$, linked with the H\"older form, is a \emph{UMD triple}.
When linked with a paraproduct form, it would still be reasonable to think of this as a UMD triple, although this structure is beyond the scope of \cite{LN20}.
However, when linked with $\Pi = \BHF$, the triple could hypothetically be thought of as a non-UMD triple, even though the spaces $L^{p_i}$ are themselves UMD.

The only remaining question is what happens when
\begin{equation}\label{eq:endpt-condn}
  \sum_{i=1}^3 \frac{1}{r_i} = 1,
\end{equation}
which sits between of our main result and our counterexample.
We conjecture that Theorem \ref{thm:intro-BHT} holds under this endpoint assumption.
In fact, we will see that it follows from an openness property of intermediate UMD spaces that we also conjecture.

\begin{conj}[UMD openness conjecture]
  Let $X$ be an $r$-intermediate UMD space for some $r \in (2,\infty)$.
  Then there exists $\tilde{r} < r$ such that $X$ is $\tilde{r}$-intermediate UMD.
  Equivalently, the set of exponents $r \in [2,\infty)$ such that $X$ is $r$-intermediate UMD is relatively open in $[2,\infty)$.
\end{conj}

This can be thought of as an extension of the intermediate UMD conjecture, which says that every $\infty$-intermediate UMD space (that is, every UMD space) is $r$-intermediate for some $r < \infty$.

Assuming the UMD openness conjecture, the intermediacy condition in Theorem \ref{thm:intro-BHT} can be simply relaxed to \eqref{eq:endpt-condn}.
To see this, assuming \eqref{eq:endpt-condn}, note that there is an index $j \in \{1,2,3\}$ such that $r_j > 2$.
By the UMD openness conjecture, there exists $\tilde{r}_j \in [2,r_j)$ such that $X_j$ is $\tilde{r}_j$-intermediate UMD.
Now we have
\begin{equation*}
  \frac{1}{\tilde{r}_j} + \sum_{i \neq j} \frac{1}{r_i} > 1,
\end{equation*}
and the conclusion follows from Theorem \ref{thm:intro-BHT}.

\begin{rmk}
  Here is a sketch of an argument that the intermediate UMD conjecture implies the UMD openness conjecture.
  If $X$ is $r$-intermediate UMD, then there exist a compatible couple $(Y_1,H_1)$, with $Y$ UMD and $H_1$ Hilbertian, such that $X$ is isomorphic to $[Y_1,H_1]_{2/r}$.
  By the intermediate UMD conjecture, $Y$ is itself $s$-intermediate UMD for some $s \in [2,\infty)$, so there exists a compatible couple $(Y_2,H_2)$ as above with $Y$ isomorphic to $[Y_2,H_2]_{2/s}$.
  Heuristically we should have $H_2 = H_1 =: H$ and
  \begin{equation*}
    X \cong [[Y_2,H]_{2/s},H]_{2/r} = [Y_2,H]_{\theta}, \qquad \theta = \Big(1 - \frac{2}{r}\Big)\frac{2}{s} + \frac{2}{r}
  \end{equation*}
  by the reiteration theorem \cite[Theorem 4.6.1]{BL76}.
  Thus $X$ is $2/\theta$-intermediate UMD, and we have that $2/\theta < r$, which improves the $r$-intermediate UMD property of $X$.
  To make this argument rigourous, one would have to ensure that it is possible to choose the spaces $(Y_2,H_2)$ such that $H_2 = H_1$ and such that $(Y_2,H_1)$ are a compatible couple.
  Of course, if attention is restricted to finite-dimensional spaces (with uniform control of constants), this difficulty is not present.
\end{rmk}


%% file: main.bbl
\providecommand{\bysame}{\leavevmode\hbox to3em{\hrulefill}\thinspace}
\providecommand{\MR}{\relax\ifhmode\unskip\space\fi MR }
\providecommand{\MRhref}[2]{%
  \href{http://www.ams.org/mathscinet-getitem?mr=#1}{#2}
}
\providecommand{\href}[2]{#2}
\begin{thebibliography}{10}

\bibitem{AF03}
R.~A. Adams and J.~J.~F. Fournier, \emph{Sobolev spaces}, 2 ed., Pure and
  Applied Mathematics, vol. 140, Elsevier/Academic Press, 2003.

\bibitem{AU20-Walsh}
A.~Amenta and G.~Uraltsev, \emph{Banach-valued modulation invariant {Carleson}
  embeddings and outer-{$L^p$} spaces: the {Walsh} case}, J. Fourier Anal.
  Appl. \textbf{26} (2020).

\bibitem{AU20-vC}
A.~Amenta and G.~Uraltsev, \emph{Variational {Carleson} operators in {UMD} spaces},
  arXiv:2003.02742, 2020.

\bibitem{BM16}
C.~Benea and C.~Muscalu, \emph{Multiple vector-valued inequalities via the
  helicoidal method}, Anal. PDE \textbf{9} (2016), no.~8, 1931--1988.

\bibitem{BM17}
C.~Benea and C.~Muscalu, \emph{Quasi-{B}anach valued inequalities via the helicoidal method},
  J. Funct. Anal. \textbf{273} (2017), no.~4, 1295--1353.

\bibitem{BM18}
C.~Benea and C.~Muscalu, \emph{Sparse domination via the helicoidal method}, arXiv:1707.05484,
  2018.

\bibitem{BL76}
J.~Bergh and J.~L\"{o}fstr\"{o}m, \emph{Interpolation spaces. {A}n
  introduction}, Springer-Verlag, 1976, Grundlehren der Mathematischen
  Wissenschaften, No. 223.

\bibitem{jB83}
J.~Bourgain, \emph{Some remarks on {B}anach spaces in which martingale
  difference sequences are unconditional}, Ark. Mat. \textbf{21} (1983), no.~2,
  163--168.

\bibitem{dB83}
D.~L. Burkholder, \emph{A geometric condition that implies the existence of
  certain singular integrals of {B}anach-space-valued functions}, Conference on
  harmonic analysis in honor of {A}ntoni {Z}ygmund, {V}ol. {I}, {II}
  ({C}hicago, {I}ll., 1981), Wadsworth Math. Ser., Wadsworth, 1983,
  pp.~270--286.

\bibitem{dB01}
D.~L. Burkholder, \emph{Martingales and singular integrals in {B}anach spaces}, Handbook
  of the geometry of {B}anach spaces, {V}ol. {I}, North-Holland, 2001,
  pp.~233--269.

\bibitem{lC66}
L.~Carleson, \emph{On convergence and growth of partial sums of {F}ourier
  series}, Acta Math. \textbf{116} (1966), 135--157.

\bibitem{CUM18}
D.~Cruz-Uribe and J.~M. Martell, \emph{Limited range multilinear extrapolation
  with applications to the bilinear {H}ilbert transform}, Math. Ann.
  \textbf{371} (2018), 615--653.

\bibitem{CDPO18}
A.~Culiuc, F.~Di~Plinio, and Y.~Ou, \emph{Domination of multilinear singular
  integrals by positive sparse forms}, J. Lond. Math. Soc. (2) \textbf{98}
  (2018), no.~2, 369--392.

\bibitem{DPDU18}
F.~Di~Plinio, Y.~Q. Do, and G.~N. Uraltsev, \emph{Positive sparse domination of
  variational {C}arleson operators}, Ann. Sc. Norm. Super. Pisa Cl. Sci. (5)
  \textbf{XVII} (2018), no.~4, 1443--1458.

\bibitem{DPGTZK18}
F.~Di~Plinio, S.~Guo, C.~Thiele, and P.~Zorin-Kranich, \emph{Square functions
  for bi-{L}ipschitz maps and directional operators}, J. Funct. Anal.
  \textbf{275} (2018), no.~8, 2015--2058.

\bibitem{DPLMV19-3}
F.~Di~Plinio, K.~Li, H.~Martikainen, and E.~Vuorinen, \emph{{Banach}-valued
  multilinear singular integrals with modulation invariance}, arXiv:1909.07236,
  2019.

\bibitem{DPLMV19}
F.~Di~Plinio, K.~Li, H.~Martikainen, and E.~Vuorinen, \emph{Multilinear singular integrals on non-commutative {$L^p$}
  spaces}, arXiv:1905.02139, 2019.

\bibitem{DPO18}
F.~Di~Plinio and Y.~Ou, \emph{Banach-valued multilinear singular integrals},
  Indiana Univ. Math. J. \textbf{67} (2018), no.~5, 1711--1763.

\bibitem{DPO18-2}
F.~Di~Plinio and Y.~Ou, \emph{A modulation invariant {C}arleson embedding theorem outside
  local {$L^2$}}, J. Anal. Math. \textbf{135} (2018), no.~2, 675--711.

\bibitem{DMT17}
Y.~Do, C.~Muscalu, and C.~Thiele, \emph{Variational estimates for the bilinear
  iterated {F}ourier integral}, J. Funct. Anal. \textbf{272} (2017), no.~5,
  2176--2233.

\bibitem{DT15}
Y.~Do and C.~Thiele, \emph{{$L^p$} theory for outer measures and two themes of
  {L}ennart {C}arleson united}, Bull. Amer. Math. Soc. (N.S.) \textbf{52}
  (2015), no.~2, 249--296.

\bibitem{cF73}
C.~Fefferman, \emph{Pointwise convergence of {F}ourier series}, Ann. of Math.
  (2) \textbf{98} (1973), 551--571.

\bibitem{HNVW16}
T.~Hyt\"{o}nen, J.~van Neerven, M.~Veraar, and L.~Weis, \emph{Analysis in
  {B}anach spaces. {V}ol. {I}: {M}artingales and {L}ittlewood-{P}aley theory},
  Ergebnisse der Mathematik und ihrer Grenzgebiete. 3. Folge., vol.~63,
  Springer, Cham, 2016.

\bibitem{HNVW17}
T.~Hyt\"{o}nen, J.~van Neerven, M.~Veraar, and L.~Weis, \emph{Analysis in {B}anach spaces. {V}ol. {II}: {P}robabilistic
  methods and operator theory}, Ergebnisse der Mathematik und ihrer
  Grenzgebiete. 3. Folge., vol.~67, Springer, Cham, 2017.

\bibitem{HL13}
T.~P. Hyt\"{o}nen and M.~T. Lacey, \emph{Pointwise convergence of vector-valued
  {F}ourier series}, Math. Ann. \textbf{357} (2013), no.~4, 1329--1361.

\bibitem{HLP13}
T.~P. Hyt\"onen, M.~T. Lacey, and I.~Parissis, \emph{The vector valued quartile
  operator}, Collect. Math. \textbf{64} (2013), no.~3, 427--454.

\bibitem{LT97}
M.~Lacey and C.~Thiele, \emph{{$L^p$} estimates on the bilinear {H}ilbert
  transform for {$2<p<\infty$}}, Ann. of Math. (2) \textbf{146} (1997), no.~3,
  693--724.

\bibitem{LT99}
M.~Lacey and C.~Thiele, \emph{On {C}alder\'{o}n's conjecture}, Ann. of Math. (2) \textbf{149}
  (1999), no.~2, 475--496.

\bibitem{LMMOV20}
K.~Li, J.~M. Martell, H.~Martikainen, S.~Ombrosi, and E.~Vuorinen,
  \emph{End-point estimates, extrapolation for multilinear {Muckenhoupt}
  classes, and applications}, arXiv:1902.04951, 2020.

\bibitem{LMO18}
K.~Li, J.~M. Martell, and S.~Ombrosi, \emph{Extrapolation for multilinear
  {Muckenhoupt} classes and applications to the bilinear {Hilbert} transform},
  arXiv:1802.03338, 2018.

\bibitem{LT79}
J.~Lindenstrauss and L.~Tzafriri, \emph{Classical {B}anach spaces. {II}},
  Ergebnisse der Mathematik und ihrer Grenzgebiete, vol.~97, Springer-Verlag,
  1979.

\bibitem{LN19}
E.~Lorist and B.~Nieraeth, \emph{Vector-valued extensions of operators through
  multilinear limited range extrapolation}, J. Fourier Anal. Appl. \textbf{25}
  (2019), 2608--2634.

\bibitem{LN20}
E.~Lorist and B.~Nieraeth, \emph{Sparse domination implies vector-valued sparse domination},
  arXiv:2003.02233, 2020.

\bibitem{MT17}
M.~Mirek and C.~Thiele, \emph{A local {$T(b)$} theorem for perfect multilinear
  {C}alder\'{o}n-{Z}ygmund operators}, Proc. Lond. Math. Soc. (3) \textbf{114}
  (2017), no.~1, 35--59.

\bibitem{MPTT04}
C.~Muscalu, J.~Pipher, T.~Tao, and C.~Thiele, \emph{Bi-parameter paraproducts},
  Acta Math. \textbf{193} (2004), no.~2, 269--296.

\bibitem{MS13-2}
C.~Muscalu and W.~Schlag, \emph{Classical and multilinear harmonic analysis.
  {V}ol. {II}}, Cambridge Studies in Advanced Mathematics, vol. 138, Cambridge
  University Press, 2013.

\bibitem{MTT02}
C.~Muscalu, T.~Tao, and C.~Thiele, \emph{Multi-linear operators given by
  singular multipliers}, J. Amer. Math. Soc. \textbf{15} (2002), no.~2,
  469--496.

\bibitem{bN19}
B.~Nieraeth, \emph{Quantitative estimates and extrapolation for multilinear
  weight classes}, Math. Ann. \textbf{375} (2019), 453--507.

\bibitem{gP16}
G.~Pisier, \emph{Martingales in {B}anach spaces}, Cambridge Studies in Advanced
  Mathematics, vol. 155, Cambridge University Press, 2016.

\bibitem{PX03}
G.~Pisier and Q.~Xu, \emph{Non-commutative {$L^p$}-spaces}, Handbook of the
  geometry of {B}anach spaces, {V}ol. 2, North-Holland, 2003, pp.~1459--1517.

\bibitem{RdF86}
J.~L. Rubio~de Francia, \emph{Martingale and integral transforms of {B}anach
  space valued functions}, Probability and {B}anach spaces ({Z}aragoza, 1985),
  Lecture Notes in Math., vol. 1221, Springer, 1986, pp.~195--222.

\bibitem{pS14}
P.~Silva, \emph{Vector-valued inequalities for families of bilinear {H}ilbert
  transforms and applications to bi-parameter problems}, J. Lond. Math. Soc.
  (2) \textbf{90} (2014), no.~3, 695--724.

\bibitem{cT06}
C.~Thiele, \emph{Wave packet analysis}, CBMS Regional Conference Series in
  Mathematics, vol. 105, American Mathematical Society, 2006.

\bibitem{TTV15}
C.~Thiele, S.~Treil, and A.~Volberg, \emph{Weighted martingale multipliers in
  the non-homogeneous setting and outer measure spaces}, Adv. Math.
  \textbf{285} (2015), 1155--1188.

\bibitem{hT78}
H.~Triebel, \emph{Interpolation theory, function spaces, differential
  operators}, North-Holland Mathematical Library, vol.~18, North-Holland, 1978.

\bibitem{gU16}
G.~Uraltsev, \emph{Variational {Carleson} embeddings into the upper 3-space},
  arXiv:1610.07657, 2016.

\bibitem{gU-thesis}
G.~Uraltsev, \emph{Time-frequency analysis of the variational {Carleson} operator
  using outer-measure {$L^p$} spaces}, Ph.D. thesis, {Rheinische
  Friedrich-Wilhelms-Universit\"at Bonn}, 2017.

\bibitem{UW19}
G.~Uraltsev and M.~Warchalski, \emph{Uniform bounds for the bilinear {Hilbert}
  transform in local {$L^1$}}, in preparation, 2020.

\bibitem{mW-thesis}
M.~Warchalski, \emph{Uniform estimates in one- and two-dimensional
  time-frequency analysis}, Ph.D. thesis, {Rheinische
  Friedrich-Wilhelms-Universit\"at Bonn}, 2018.

\end{thebibliography}
